\documentclass[11pt,makeidx]{amsart}

\usepackage{amssymb,mathrsfs, amsmath, amsthm}
\usepackage{url}
\usepackage{titletoc}
\usepackage{enumerate,mathrsfs,xspace}

\allowdisplaybreaks

\usepackage[usenames,dvipsnames]{xcolor}

\definecolor{cobalt}{rgb}{0.0, 0.28, 0.67}
\definecolor{darkblue}{rgb}{0.0, 0.0, 0.55}

\usepackage{tikz}\usetikzlibrary{matrix,arrows,calc,cd,decorations.pathmorphing}

\usepackage[colorlinks,linkcolor=BrickRed,citecolor=darkblue,urlcolor=black,hypertexnames=true,pagebackref=false]{hyperref}

\usepackage{upgreek}
\usepackage{cmap}

\DeclareMathAlphabet{\mathpzc}{OT1}{pzc}{m}{it}

\usepackage{ulem}

\usepackage{graphicx}
\newcommand{\minus}{\scalebox{0.75}[1.0]{$-$}}

\DeclareFontFamily{U}{mathx}{\hyphenchar\font45}
\DeclareFontShape{U}{mathx}{m}{n}{
      <5> <6> <7> <8> <9> <10>
      <10.95> <12> <14.4> <17.28> <20.74> <24.88>
      mathx10
      }{}
\DeclareSymbolFont{mathx}{U}{mathx}{m}{n}
\DeclareFontSubstitution{U}{mathx}{m}{n}
\DeclareMathAccent{\widecheck}{0}{mathx}{"71}
\DeclareMathAccent{\wideparen}{0}{mathx}{"75}

\newtheorem{theorem}{Theorem}[section]
\newtheorem{cor}[theorem]{Corollary}
\newtheorem{lemma}[theorem]{Lemma}

\newtheorem{prop}[theorem]{Proposition}
\newtheorem{definition}[theorem]{Definition}

\newtheorem{remark}[theorem]{Remark}

\newtheorem{example}[theorem]{Example}
\newtheorem{thm}[theorem]{Theorem}
\newtheorem{lem}[theorem]{Lemma}

\newcommand{\M}{\mathbb L}

\renewcommand{\subset}{\subseteq}
\renewcommand{\supset}{\supseteq}

\newcommand{\hmm}{hermitian monic\xspace}
\newcommand{\bmin}{ball-minimal\xspace}
\newcommand{\bminy}{ball-minimality\xspace}
\newcommand{\irrL}{indecomposable\xspace}
\newcommand{\irrLy}{indecomposability\xspace}

\def\bel{\begin{lemma}}
\def\eel{\end{lemma}}

\def\ben{\begin{enumerate}}
\def\een{\end{enumerate}}

\def\bem{\begin{pmatrix}}
\def\eem{\end{pmatrix}}

\def\beq{\begin{equation}}
\def\eeq{\end{equation}}

\def\bep{\begin{proof}}
\def\eep{\end{proof}}

\renewcommand{\qedsymbol}{\rule[.12ex]{1.2ex}{1.2ex}}

\def\ssec{\subsection}

\def\matdg{M_d(\C)^g}

\def\matnjg{M_{n_j}(\C)^g}

\def\ee{e}

\def\cB{\mathcal B}
\def\cD{\mathcal D}
\def\cT{\mathcal T}

\def\cZ{\mathcal Z}

\def\R{{\mathbb R}}

\def\C{\mathbb {C}}
\def\N{\mathbb {N}}
\def\cD{\mathcal D}

\def\cB{\mathcal B}

\def\sW{\mathscr W}

\def\cZ{\mathcal Z}

\def\La{\Lambda}

\def\JJ{J}
\def\GG{f}
\def\interior{\operatorname{int}}
\def\cU{\mathcal{U}}
\def\cV{\mathcal{V}}
\def\rr{q}
\def\pp{p}
\def\qq{q}

\def\hair {\operatorname{hair}}

\def\sZ{\mathscr{Z}}
\def\sM{\mathscr{M}}

\def\spann{\operatorname{span}}

\def\wE{\widehat{E}}
\def\ff{\mathscr{F}}

\def\ran{\operatorname{ran}}

\def\dd{\ell}
\def\FG{G}

\def\adj{\operatorname{adj}}

\newcommand*{\mat}[1]{M_{#1}(\C)}
\newcommand{\Langle}{\mathop{<}\!}
\newcommand{\Rangle}{\!\mathop{>}}

\newcommand{\pxy}{\C\!\Langle x,y\Rangle}

\newcommand{\pxx}{\C\!\Langle x\Rangle}

\makeatletter
\def\moverlay{\mathpalette\mov@rlay}
\def\mov@rlay#1#2{\leavevmode\vtop{
    \baselineskip\z@skip \lineskiplimit-\maxdimen
    \ialign{\hfil$#1##$\hfil\cr#2\crcr}}}
\makeatother

\newcommand{\plangle}{\moverlay{(\cr<}}
\newcommand{\prangle}{\moverlay{)\cr>}}

\newcommand{\rxy}{\C\plangle x,y \prangle}

\newcommand{\mop}{monic pencil\xspace}
\newcommand{\mops}{monic pencils\xspace}

\def\Mre{\M^{\rm re}}
\def\Qre{Q^{\rm re}}
\def\Lre{L^{\rm re}}
\def\LRE{\Lre}
\def\Zre{\cZ^{\rm re}}
\def\cdotb{\boldsymbol{\cdot}}
\def\QQ{\mathfrak{V}}
\def\sG{\mathscr{G}}
\def\hC{\widehat{C}}

\def\sV{\mathscr{V}}
\def\row{\operatorname{row}}
\def\matt{\operatorname{mat}}
\def\yy{y}

\newcommand{\Aalt}{\mathfrak{A}}
\newcommand{\Balt}{\mathfrak{B}}

\newcommand{\backK}{K_*} 
\newcommand{\trace}{\operatorname{trace}}
\newcommand{\exterior}{\operatorname{ext}}
\def\sH{\mathfrak H}
\newcommand{\ttW}{{\tt{W}}}
\newcommand{\ttV}{{\tt{V}}}

\newcommand{\psif}{f}
\newcommand{\AF}{A}
\newcommand{\rt}{t}
\newcommand{\wc}{\widecheck}
\newcommand{\tJ}{\tilde{J}}
\newcommand{\tT}{\tilde{T}}

\newcommand{\sI}{\mathscr{I}}

\newcommand{\mfq}{q}
\newcommand{\mfQ}{Q}
\newcommand{\mfr}{r}
\newcommand{\mfs}{s}
\newcommand{\tmfs}{\tilde{\mfs}}
\newcommand{\mfp}{p}
\newcommand{\shair}{\mathscr{H}}
\newcommand{\NN}{N}
\newcommand{\scS}{\mathscr{S}}

\newcommand{\sE}{\mathscr{E}}
\newcommand{\sF}{\mathscr{F}}

\newcommand{\BB}{\mathbb{B}}
\newcommand{\dom}{\operatorname{dom}}

\newcommand{\midx}{:}

\newcommand{\wbeta}{\widehat{\beta}}

\newcommand{\base}{\varepsilon}
\newcommand{\basr}{\varrho}
\newcommand{\tbeta}{\widetilde{\beta}}
\newcommand{\fE}{\vecd{u}\, \vecd{w}^*}
\newcommand{\vecd}[1]{\underline{#1}}

\DeclareMathOperator{\rg}{rg}
\DeclareMathOperator{\rk}{rk}

\numberwithin{equation}{section}

\newcommand{\df}[1]{{\bf{#1}}{\index{#1}}}

\makeindex
\makeatletter
\newcommand{\mycontentsbox}{%
{
\addtolength{\parskip}{-2.0pt}\footnotesize
\newpage\printindex\tableofcontents}}
\def\enddoc@text{\ifx\@empty\@translators \else\@settranslators\fi
\ifx\@empty\addresses \else\@setaddresses\fi
\newpage\mycontentsbox
}
\makeatother

\contentsmargin{2.55em} 

\title[Bianalytic free maps between spectrahedra and spectraballs]{Bianalytic free maps between\\[.1mm] spectrahedra and spectraballs}

\author[J.W. Helton]{J. William Helton${}^1$}
\address{J. William Helton, Department of Mathematics\\
  University of California \\
  San Diego}
\email{helton@math.ucsd.edu}
\thanks{${}^1$Research supported by the NSF grant
DMS 1500835.}

\author[I. Klep]{Igor Klep${}^{2}$}
\address{Igor Klep, Department of Mathematics, 
University of Ljubljana, Slovenia}
\email{igor.klep@fmf.uni-lj.si}
\thanks{${}^2$Supported by the 
Slovenian Research Agency grants J1-8132, N1-0057 and P1-0222. Partially supported 
by the 
Marsden Fund Council of the Royal Society of New Zealand.}

\author[S. McCullough]{Scott McCullough${}^3$}
\address{Scott McCullough, Department of Mathematics\\
  University of Florida\\ Gainesville 
   }
   \email{sam@math.ufl.edu}
\thanks{${}^3$Research supported by the NSF grants  DMS-1764231}

\author[J. Vol\v{c}i\v{c}]{Jurij Vol\v{c}i\v{c}${}^4$}
\address{Jurij Vol\v{c}i\v{c}, Department of Mathematics\\
  Texas A\& M University \\ College Station}
   \email{volcic@post.bgu.ac.il}
\thanks{${}^4$Research supported by the Deutsche Forschungsgemeinschaft (DFG) Grant No. SCHW 1723/1-1.}

\date{\today}  



\begin{document}
\begin{abstract}
Linear matrix inequalities (LMIs)
  are ubiquitous
in real algebraic geometry,
semidefinite programming, control theory and signal processing.
  LMIs with (dimension free) matrix unknowns 
  are central to the theories of
  completely positive maps and operator algebras,
   operator systems and spaces, and serve as the paradigm for matrix
  convex sets.
The matricial feasibility set of an LMI is called a free spectrahedron.

In this article, the  bianalytic maps between  a very general class of
   {\it ball-like} free spectrahedra 
(examples of which include row or column contractions, and tuples of contractions)
   and arbitrary  free spectrahedra
 are characterized and seen to have an elegant
  algebraic form.
They are all highly structured rational maps.
  In the case that both the domain and
  codomain are ball-like, these
   bianalytic maps  are explicitly determined and the article gives
  necessary and sufficient
conditions
  for the existence of such
   a map with a specified value and derivative  at a  point.
In particular, this result  leads to a classification of automorphism groups of
ball-like free spectrahedra.
   The proofs depend on a novel free Nullstellensatz,
  established only after new tools in free
  analysis are developed  and applied to obtain fine detail,
 geometric in nature locally and algebraic in nature globally,
 about the boundary of ball-like free spectrahedra.
 \end{abstract} 

\maketitle

\numberwithin{equation}{section}

\dottedcontents{section}[3.8em]{}{2.3em}{.4pc} 
\dottedcontents{subsection}[6.1em]{}{3.2em}{.4pc}
\dottedcontents{subsubsection}[8.4em]{}{4.1em}{.4pc}

\section{Introduction}
Fix a positive integer $g$. 
For positive integers $n$, let $M_n(\C)^g$  denote the set of $g$-tuples $X=(X_1,\dots,X_g)$ of $n\times n$ matrices with entries from $\C$.
Given a tuple $E=(E_1,\dots,E_g)$ of $d\times e$ matrices, the sequence
 $\cB_E= (\cB_E(n))_n$  defined by \index{$\cB_E$}
\[
 \cB_E(n) =\{X\in M_n(\C)^g: \| \sum E_j\otimes X_j\|\le 1\}
\]
is a \df{spectraball}. The spectraball at \df{level} one,  $\cB_E(1),$
 is a rotationally invariant closed convex subset of $\C^g.$ 
Conversely, a rotationally invariant closed convex subset
of $\C^g$ can be approximated by sets of the form $\cB_E(1)$.
A spectraball $\cB_E$ is not 
determined by $\cB_E(1)$. For example, letting
$ F_1 =\begin{pmatrix} 1&0\end{pmatrix},$  $F_2 =\begin{pmatrix} 0 & 1 \end{pmatrix},$
and $E_j=F_j^*$, we have 
$\cB_E(1)=\cB_F(1)=\BB^2$, the unit ball in $\C^2$, but $\cB_E(2)\ne \cB_F(2)$.
Indeed, $\cB_F$ (resp. $\cB_E$) is the two variable \df{row ball}
(resp. \df{column ball}) equal the set of pairs $(X_1,X_2)$
such that $X_1 X_1^* + X_2 X_2^* \preceq I$ (resp. $X_1^* X_1 +X_2^*X_2\preceq I$),
 where the inequality 
$T\succeq 0$ indicates the selfadjoint matrix $T$ is positive semidefinite.
Another well-known example is the \df{free polydisc}. It is the spectraball
$\cB_E$ determined by the tuple $E=(e_1e_1^*,\ldots,e_ge_g^*)\in M_{g}(\C)^g,$ where 
$\{e_1,\ldots,e_{g}\}$ is the standard orthonormal basis for $\C^{g}$.
 Thus $\cB_E(n)$ is the set of 
tuples $X\in M_n(\C)^g$ such that $\|X_j\|\le 1$ for each $j.$

For $A\in M_d(\C)^g$, let $L_A(x,y)$ denote the \df{\mop}
\[
 L_A(x,y)  = I+\sum A_jx_j +\sum A_j^* y_j, \index{$L_A(x,y)$}
\]
and let
\[
  \Lre_A(x)=L_A(x,x^*)= I+\sum A_jx_j +\sum A_j^* x_j^*   \index{$\Lre_A(x)$}
\]
 denote the corresponding \df{\hmm  pencil}. The 
 set $\cD_A(1)$ consisting  of  $x\in\C^g$ such that
 $\Lre_A(x)\succeq 0$
 is a \df{spectrahedron}. 
Spectrahedra are basic objects in a number of areas of mathematics;
e.g.~semidefinite programming, convex optimization and in real algebraic geometry \cite{BPR13}. 
They also figure prominently in determinantal representations \cite{Bra11,GK-VVW,NT12,Vin93}, 
in the solution of the Kadison-Singer paving conjecture \cite{MSS15},
the solution of the Lax conjecture \cite{HV07},
 and in systems engineering \cite{BGFB94, Skelton}.

For $A\in M_{d\times e}(\C)^g$, 
the \df{homogeneous linear pencil}  \index{$\Lambda_A(x)$}
$\Lambda_A(x)=\sum_j A_jx_j$ 
evaluates at 
$X\in M_n(\C)^g$
as
\[
\Lambda_A(X) = \sum A_j \otimes X_j \in M_{d\times e}(\C)\otimes M_n(\C). 
\]
In the case  $A$ is square  ($d=e$), 
 the \hmm  pencil $\Lre_A$ evaluates at $X$ as
\[
 \Lre_A(X) = I+\Lambda_A(X) +\Lambda_A(X)^* = I+\sum A_j \otimes X_j + \sum A_j^* \otimes X_j^*.
\]
Thus $\Lre_A(X)^*=\Lre_A(X)$. 
Similarly, if $Y\in M_n(\C)^g$, then $L_A(X,Y)=I+\La_A(X)+\La_{A^*}(Y)$. In particular,
$\Lre_A(X)=L_A(X,X^*)$.

The  \df{free spectrahedron} determined by
$A\in M_r(\C)^g$ is the sequence of sets $\cD_A= (\cD_A(n))$, where\index{$\cD_A$}
\[
\cD_A(n) =\{X\in M_n(\C)^g: \Lre_A(X)\succeq 0\}.
\]
The spectraball $\cB_E$ is a spectrahedron since $\cB_E=\cD_B$ for $B=(\begin{smallmatrix}0&E\\0&0\end{smallmatrix})$.
Free spectrahedra arise naturally in applications such as systems engineering \cite{dOHMP09} and in the theories of matrix convex
sets, operator algebras and operator spaces and completely positive maps \cite{EW,HKMjems,Pau,PSS}. They also provide tractable useful 
relaxations for spectrahedral inclusion problems that arise in semidefinite programming and control theory
such as the matrix cube problem \cite{BN02,HKMS,DDOSS}.

The \df{interior} of the free spectrahedron $\cD_A$ is the sequence \index{$\interior(\cD_A)$}
$\interior(\cD_A)=\left(\interior(\cD_A(n))\right)_n$, where
\[
\interior(\cD_A(n)) =\{X\in M_n(\C)^g:  \Lre_A(X)\succ 0\}.
\]
A \df{free  mapping}
 $\varphi:\interior(\cD_B)\to \interior(\cD_A)$  is 
 a sequence of maps $\varphi_n:\interior(\cD_B(n))\to \interior(\cD_A(n))$
 such that if $X\in \interior(\cD_B(n))$ and $Y\in\interior(\cD_B(m))$,
 then
\[
 \varphi_{n+m}\left ( \begin{pmatrix} X&0\\0&Y\end{pmatrix} \right )
  = \begin{pmatrix} \varphi_n(X) &0\\0& \varphi_m(Y) \end{pmatrix},
\]
 and if $X\in \interior(\cD_B(n))$ and $S$ is an invertible $n\times n$
 matrix such that 
\[
  S^{-1} X S 
   = \begin{pmatrix} S^{-1}X_1 S, \dots, S^{-1} X_g S\end{pmatrix}\in \interior(\cD_B(n)),
\]
then 
\[
 \varphi_n(S^{-1} X S)= S^{-1}\varphi_n(X)S.
\] 
Often we omit the subscript $n$ and write only $\varphi(X)$.
The free mapping $\varphi$ is \df{analytic} if each $\varphi_n$ is analytic.

The central result of this article, Theorem \ref{t:pencil-ball-alt}, explicitly  characterizes
the free bianalytic mappings $\varphi$ between 
$\interior(\cB_E)$ and $\interior(\cD_A)$. These maps 
are birational and highly structured.
Up to affine linear change of variable, they are what 
 we call \df{convexotonic} (see Subsection \ref{sec:ctmaps} below).
 In the special case that $\cD_A=\cB_C$ is also
a spectraball, given $b\in \interior(\cB_C)$ and
a $g\times g$ matrix $M$, Corollary \ref{t:B2B} gives explicit  necessary
and sufficient  algebraic relations
between $E$ and $C$ for the existence of
a free bianalytic mapping $\varphi:\interior(\cB_E)\to \interior(\cB_C)$
satisfying $\varphi(0)=b$ and $\varphi^\prime(0)=M$. As an illustration
of the result, this corollary classifies, from first principles, the free
automorphisms of the matrix balls -- the row and column balls are special cases -- 
and of the free polydiscs.
See Remark \ref{r:aff}\eqref{i:affMT} and Subsubsections \ref{e:FPD} and \ref{e:MT}.

There are two other results we would like to highlight in this introduction.
Theorem \ref{t:hairspans}, establishes an equivalence between an algebraic
{\it irreducibility} condition on the defining polynomial of a spectraball
and  a geometric property of its boundary critical in the study of binalaytic
maps between free spectrahedra. Its 
proof requires detailed information,  both  local and global, about the boundary
of a spectraball,  collected in Section \ref{sec:bianalhedra}.
As a consequence of Theorem  \ref{t:hairspans}, we obtain 
 a  version  of the main result from \cite{AHKM18}
characterizing bianalytic maps between free spectrahedra that send 
the origin to the origin with elegant irreducibility and minimality hypotheses
on the  free spectrahedra replacing our earlier cumbersome geometric conditions.
See Theorem \ref{thm:sdmain} in Subsection \ref{subsec:main1}.
Another consequence of Theorem \ref{t:hairspans}, and an 
essential ingredient in the  proof of Theorem \ref{t:pencil-ball-alt}, is
an of independent interest Nullstellensatz. It is stated as Proposition \ref{p:HKVpBergman} in 
 Subsection \ref{sec:null}. Roughly, it says that a matrix-valued
 analytic free polynomial singular on the boundary of a spectraball 
 is $0$.

\subsection{Convexotonic maps}
\label{sec:ctmaps}
A  $g$-tuple of $g\times g$ matrices $(\Xi_1,\dots,\Xi_g)\in M_g(\C)^g$  satisfying
\begin{equation*} 
 \Xi_k \Xi_j = \sum_{s=1}^g (\Xi_j)_{k,s} \Xi_s,
\end{equation*}
for each $1\le j,k\le g$, is a \df{convexotonic tuple}. 
The expressions
 $p=\begin{pmatrix} p^1 & \cdots & p^g\end{pmatrix}$ 
  and $q=\begin{pmatrix} q^1 & \cdots & q^g\end{pmatrix}$  
whose entries are
\begin{equation*}
 p^i (x)=\sum_j  x_j e_j^* (I-\Lambda_\Xi(x))^{-1} e_i 
           \qquad \text{and} \qquad  q^i (x)= \sum x_j e_j^* (I +\Lambda_\Xi(x))^{-1} e_i, 
\end{equation*}
that is, in row form,
\begin{equation*}
 p(x)= x(I-\Lambda_\Xi(x))^{-1}
\qquad \text{and} \qquad q=x(I+\Lambda_\Xi(x))^{-1}, 
\end{equation*}
are \df{convexotonic maps}.  Here $p$ evaluates at $X\in M_n(\C)^g$ as
$$p(X) = \begin{pmatrix} X_1 & \cdots & X_g\end{pmatrix}
\left(I_{gn}-\sum_{j=1}^g\Xi_j\otimes X_j\right)^{-1}$$
and the output $p(X)\in M_{n\times gn}(\C)=M_n(\C)^g$ is interpreted as a $g$-tuple of $n\times n$ matrices.
It turns out  the mappings $p$
 and $q$ are free rational maps (as explained in Section \ref{s:prelims}) 
and inverses of one another  (see \cite[Proposition 6.2]{AHKM18}).

 Convexotonic tuples arise naturally as
 the structure constants of a finite dimensional algebra. 
 If 
 $A\in M_r(\C)^g$ is linearly independent (meaning the 
 set $\{A_1,\dots,A_g\}\subset M_r(\C)$ is linearly independent)
  and spans an algebra, then,
 e.g.~by Lemma \ref{l:secret} below,  there
 is a uniquely determined convexotonic tuple 
 $\Xi=(\Xi_1,\dots,\Xi_g)\in M_g(\C)^g$ such that
\begin{equation}
 \label{e:JgetsXi}
 A_k A_j = \sum_{s=1}^g (\Xi_j)_{k,s} A_s.
\end{equation}

\subsection{Free bianalytic maps from a  spectraball to a free spectrahedron}
\label{subsec:main2}
 A tuple  $E\in M_{d\times e}(\C)^g$ is \df{\bmin} (for $\cB_E$)
  if there does not exist $E'$ of size $d'\times e'$ with 
 $d'+e'<d+e$ such that $\cB_E=\cB_{E'}$. 
In fact, if $E$ is \bmin and $\cB_{E^\prime} =\cB_E$, then  $d\le d^\prime$ and $e\le e^\prime.$
 by Lemma \ref{l:trans}\eqref{it:lt7}\footnote{See also 
\cite[Section 5 or Lemma 1.2]{HKM11a}.} and $E$ is
unique in the following sense. Given another tuple
$F\in M_{d\times e}(\C)^g$, the tuples $E$ and $F$ are
\df{ball-equivalent} if there exists unitaries
$W$ and $V$ of sizes $d\times d$ and $e\times e$ respectively such that
$F=WEV$. Evidently if $E$ and $F$ are ball-equivalent, then
$\cB_E=\cB_F$. Conversely, if $E$ and $F$ are both \bmin and $\cB_E=\cB_F$,
then $E$ and $F$ are ball-equivalent (see Lemma \ref{l:trans}\eqref{it:lt7}
and more generally \cite{FHL18}).

 Given $A\in M_r(\C)^g$, we say $L_A$ (or $\Lre_A$) is \df{minimal} for a free spectrahedron 
 $\cD$ if $\cD=\cD_A$ and
 if for any other $B\in M_{r^\prime}(\C)^g$ satisfying  $\cD=\cD_B$
 it follows that $r^\prime \ge r$.   A minimal $L_A$ for $\cD_A$ exists and is  unique
 up to unitary equivalence \cite{HKM, Zal17}.
We can now state Theorem \ref{t:pencil-ball-alt}, our
principal result on bianalytic mappings from a spectraball 
 onto a free spectrahedron.  Since the hypotheses of Theorem
\ref{t:pencil-ball-alt} are invariant under affine linear change
of variables, 
 the normalizations $f(0)=0$ and $f^\prime(0)=I$
are simply a matter of convenience.  Given $B\in M_d(\C)^g$, by a free bianalytic 
 map   $f:\interior(\cD_B)\to \interior(\cD_A)$, we mean $f$ is
 a free analytic map and there exists a free analytic map $g:\interior(\cD_A)\to\interior(\cD_B)$
 such that  $g_n(f_n(X))=X$ and
 $f_n(g_n(Y))=Y$  for each $n,$ 
 $X\in \interior(\cD_B(n))$ and $Y\in \interior(\cD_A(n)).$

\begin{theorem}
\label{t:pencil-ball-alt}
 Suppose $E\in M_{d\times e}(\C)^g$ and  $A\in M_r(\C)^g$ are  linearly independent.
 If  $f:\interior(\cB_E)\to \interior(\cD_A)$ is a free  bianalytic mapping
 with $f(0)=0$ and $f^\prime(0)=I_g$,  then $f$ is convexotonic.

If, in addition,  $A$ is minimal for $\cD_A$, then 
  there is convexotonic tuple $\Xi\in M_g(\C)^g$ such that
 equation \eqref{e:JgetsXi} holds, and  $f$ is the corresponding 
 convexotonic map, namely
\begin{equation}
\label{e:Af}
  f(x) = x(I-\Lambda_\Xi(x))^{-1}.
\end{equation}
In particular, $\{A_1,\dots,A_g\}$ spans an algebra.

If $A$ is minimal for $\cD_A$ and  $E$ is \bmin, then $\max\{d,e\}\le r\le d+e$
and there is an $r\times r$ unitary matrix $U$ such that, up to unitary
 equivalence,
\begin{equation}
\label{e:UEE}
 A = U \begin{pmatrix} E & 0 \\ 0 & 0 \end{pmatrix}.
\end{equation}

Conversely, given a linearly independent  $E\in M_{d\times e}(\C)^g,$ 
 an integer  $r\ge \max\{d,e\}$ and an $r\times r$ unitary matrix $U$, let
$A$ be given by equation \eqref{e:UEE}. %
If there is a tuple $\Xi$ such that equation \eqref{e:JgetsXi} holds,
 then $f$ of equation \eqref{e:Af}
 is a free bianalytic map $f:\interior(\cB_E)\to \interior(\cD_A).$
\end{theorem}

\begin{proof}
 See Corollary \ref{c:pba-converse} and Section \ref{sec:winning}. 
\end{proof}

\begin{remark}\rm
\label{r:aff}
\hspace{2em}
\begin{enumerate}[\rm (a)]\itemsep=5pt
\item\label{i:affa}
 The normalizations $f(0)=0$ and $f^\prime(0)=I_g$ can easily be enforced.
 Given a $g\times g$ matrix $\Delta$
and a tuple $C\in M_{d\times e}(\C)^g$, let \index{$\Delta\cdotb C$}
 $\Delta\cdotb C \in M_{d\times e}(\C)^g$ denote the tuple
\begin{equation}
\label{e:MdotB}
 (\Delta\cdotb C)_j = \sum_k \Delta_{j,k} C_k.
\end{equation}
 In the case $f:\interior(\cB_E)\to \interior(\cD_A)$ is bianalytic, but $f(0)=b\ne 0$ or $f^\prime(0)=M\ne I$,
 let  $\lambda:\cD_A\to\cD_F$
denote the affine linear map $\lambda(x) = x\cdotb M + b$, where
\[
F=M \cdotb (\sH A \sH) \quad\text{ and }\quad \sH=\LRE_A(b)^{-1/2}.
\]
By Proposition \ref{prop:aff}, $h = \lambda^{-1} \circ f:\interior(\cB_E) \to \interior(\cD_B)$ is bianalytic  with
$h(0)=0$ and $h^\prime(0)=I_g$ and,  if $A$ is minimal for $\cD_A$,
 then $B$ is minimal for $\cD_B.$ In particular, $f$ is, up to affine linear equivalence, convexotonic.  

Further, with a bit of bookkeeping the algebraic conditions 
of equations \eqref{e:UEE} and \eqref{e:JgetsXi} can be
expressed intrinsically in terms of $E$ and $A$. In the case
$\cD_A$ is a spectraball, these conditions are spelled
out in Corollary \ref{t:B2B} below.
\item 
In the context of Theorem \ref{t:pencil-ball-alt} (and Remark \ref{r:aff}),
 $f^{-1}$ extends analytically to an open set containing $\cD_A$ and if
 $\cD_A$ is bounded, then $f$ extends analytically to an open set
 containing $\cB_E$. The precise result is stated
as Theorem \ref{t:1point1} below. Theorem \ref{t:1point1} is 
 an elaboration on \cite[Theorem 1.1]{AHKM18}. 
\item \label{i:aff3}
 Given $A$ as in equation \eqref{e:UEE} and 
 writing $U=(U_{j,k})_{j,k=1}^2$ in the natural block form, equation \eqref{e:JgetsXi}
 is equivalent to $E_k U_{11} E_j =\sum_s (\Xi_j)_{k,s}E_s$.
\item Corollary \ref{c:rat1} and Theorem \ref{t:rat2} extend Theorem \ref{t:pencil-ball-alt}
 to cases where the codomain is matrix convex\footnote{In the present setting, matrix
convex is the same as the convexity at each level.}, but  not, by assumption, the interior of a
 free spectrahedron assuming the inverse of the bianalytic map is 
 rational.
\item 
  Here is an example of a free spectrahedron that is not a spectraball,
 but is bianalytically equivalent to a spectraball. Let
\[
 E= I_2, \, E_1 = \begin{pmatrix}0&0\\1&0 \end{pmatrix}, \,
 U=\begin{pmatrix} 0&0&1\\1&0&0\\0&1&0\end{pmatrix}
\]
and set 
\[
 A = U \begin{pmatrix} E &0\\0&0\end{pmatrix} \in M_3(\C)^2.
\]
 With $\Xi_1 = (\begin{smallmatrix} 0&1\\0&0\end{smallmatrix})$ and
 $\Xi_2=0$, the tuples $A$ and $\Xi$ satisfy equation \eqref{e:JgetsXi}
and the corresponding convexotonic map is given by 
$f(x_1,x_2) =(x_1,x_2+x_1^2)$. It is thus bianalytic from $\interior(\cB_E)$
to $\interior(\cD_A)$. Moreover, $\cD_A$ is not a spectraball since
$\cD_A(1)$ is not rotationally invariant.   \qed
\end{enumerate} 
\end{remark}

For a matrix $T$ with $\|T\|\le 1$, let \index{$D_T$}
 $D_T$ \index{$D_T$} denote the positive square root of $I-T^*T$. 
Thus, if $T$ is $k\times \ell$, then $D_{T}$ is $\ell\times\ell$
and $D_{T^*}$ is $k\times k$.

\begin{cor}
\label{t:B2B}
 Suppose $E\in M_{d\times e}(\C)^g$ and $C\in M_{k\times \ell}(\C)^g$ are 
 linearly independent and
 \bmin,  $b\in \interior(\cB_C)$ and $M\in M_{g}(\C)$. 
 There exists a free bianalytic mapping
 $\varphi:\interior(\cB_E)\to \interior(\cB_C)$ such 
 that $\varphi(0)=b$ and $M=\varphi^\prime(0)$  if and only if 
 $E$ and $C$ have the same size (that is, $k=d$ and $\ell = e$)
 and there exist $d\times d$ and $e\times e$ unitary
 matrices $\sW$ and $\sV$ respectively and a convexotonic
 $g$-tuple $\Xi\in M_g(\C)^g$  such that
\begin{enumerate}[\rm (a)]\itemsep=5pt
 \item 
 \label{i:ELCE}
  $-E_j \sV^* \Lambda_C(b)^* \sW E_k =  \sum_s (\Xi_k)_{j,s}E_s =(\Xi_k \cdotb E)_j$; and
 \item 
  \label{i:CME}
     $D_{\Lambda_C(b)^*} \sW E_j \sV^* D_{\Lambda_C(b)}  = \sum_s M_{js} C_s = (M\cdotb C)_j,$
\end{enumerate} 
for all $1\le j,k\le g$.    Moreover, in this case $\varphi = \psi \cdotb M + b,$
 where $\psi$ is the convexotonic map associated to $\Xi$; i.e.,  $\psi(x) = x(I-\Lambda_\Xi(x))^{-1}.$
\end{cor}

The proof of Corollary \ref{t:B2B} appears in Subsubsection \ref{sssec:proof1point3}.

\begin{remark}\rm
\label{r:GivenC}
\begin{enumerate}[\rm (a)]\itemsep=5pt
 \item  If $\cB_E$ and $\cB_C$ are bounded (equivalently $E$ and $C$
 are linearly independent \cite[Proposition 2.6(2)]{HKM}), then any free bianalytic map 
  $\varphi:\interior(\cB_E)\to\interior(\cB_C)$
 is, up to an affine linear bijection, convexotonic without any further assumptions (e.g.,
 $C$ and $E$ need not be \bmin). Indeed, simply replace $E$ and $C$ by 
 \bmin $E^\prime$ and $C^\prime$ with $\cB_{E^\prime} = \cB_E$ and $\cB_{C^\prime}=\cB_C$ 
 and apply Corollary~\ref{t:B2B}.
 The \bmin hypothesis allows for an explicit description
 of $\varphi$.\looseness=-1

\item  While $M$ is not assumed invertible, both the condition $M=\varphi^\prime(0)$ (for a
 bianalytic $\varphi$)  and 
 the identity of  Corollary \ref{t:B2B}\eqref{i:CME} (since $E$ is assumed linearly independent) imply it is. 

\item  Assuming $E$ and $C$ of Corollary \ref{t:B2B} are \bmin, 
 by using the relation between $E$ and $C$ from Corollary \ref{t:B2B}\eqref{i:CME},
 item \eqref{i:ELCE}  can be expressed purely in terms of $C$ as
\begin{equation}
 \label{e:onlyC}
 C_j D_{\Lambda_C(b)}^{-1} \Lambda_C(b)^* D_{\Lambda_C(b)^*}^{-1} C_k 
   \in \spann\{C_1,\dots,C_g\}.
\end{equation}
 In particular,   given a \bmin tuple $C\in M_{d\times e}(\C)^g$ 
 and $b\in \interior(\cB_C)$, if equation \eqref{e:onlyC} holds 
 then, for any choice of $M,\sW$ and $\sV$ and solving
 equation \eqref{i:CME} for $E,$ 
 there is a free bianalytic map $\varphi:\interior(\cB_E)\to\interior(\cB_C)$  such 
 that $\varphi(0)=b$ and $\varphi^\prime(0)=M.$ 
\item 
\label{i:affMT}
 Among the results in \cite{MT16} is a complete analysis
of the free bianalytic maps between the free versions of 
matrix ball,  antecedents and special cases of which 
appear elsewhere in the literature such as \cite{HKMSlinglend}
and \cite{Po2}. The connection between the results in \cite{MT16} on 
free matrix balls and Corollary \ref{t:B2B} is worked out
in Subsubsection \ref{e:MT}. 
Subsubsection \ref{e:FPD} gives a complete
classification of free automorphisms of free polydiscs. \qed
\end{enumerate}
\end{remark}

\subsection{Main result on maps between free spectrahedra}\label{subsec:main1}
The article \cite{AHKM18} 
characterizes the triples $(p,A,B)$ such that $p:\cD_A\to\cD_B$
 is bianalytic under 
unconventional geometric hypotheses (sketched in Subsection \ref{sec:introGeom} below), cf.~\cite[\S 7]{AHKM18}.
Here we obtain Theorem \ref{thm:sdmain}  by  converting  those geometric hypotheses to algebraic {\it irreducibility} hypotheses
that we now describe.

For a tuple of rectangular matrices $E=(E_1,\dots,E_g)\in M_{d\times e}(\C)^g$ 
denote \index{$Q_E$} \index{$\M_E$} \index{$\ker(E)$} \index{$\ran(E)$}
$$Q_{E}(x,y) := I - \Lambda_{E^*}(y) \Lambda_{E}(x),\qquad 
\M_E(x,y):=\begin{pmatrix}I & \Lambda_E(x) \\ \Lambda_{E^*}(y) & I\end{pmatrix},$$
$$\ker(E):= \bigcap_{j=1}^g  \ker(E_j) = \ker(\begin{pmatrix} E_1\\ \vdots \\ E_g\end{pmatrix}),
\quad \ran(E) = \ran(\begin{pmatrix} E_1 & \dots E_g\end{pmatrix}).
$$
Thus $\M_E(x,y)=L_F(x,y)$ where
\[
 F=\begin{pmatrix} 0  & E \\ 0 & 0\end{pmatrix}.
\]
We also let $\Mre_E$ denote the \hmm   pencil,
\[
 \Mre_E(x) := \M_E(x,x^*) = L_F(x,x^*) = \Lre_F(x) \index{$\Mre_E(x)$}
\]
 and likewise 
\[
\Qre_E(x)= Q_E(x,x^*).
\]
Observe  $\cB_E=\cD_{\M_E^{\rm re}}:=\{X: \M_E(X,X^*)\succeq 0\}=\cD_F.$
 Finally, for a \mop $L_A$, let
\begin{equation*}\index{$\cZ_{L_A}$}\index{$\Zre_{L_A}$}
 \cZ_{L_A}=\{(X,Y): \det(L_A(X,Y))=0\}, \ \ \ 
 \Zre_{L_A} = \{X: \det(\Lre_A(X))=0\}.
\end{equation*}
 We also  use the  notation  $\cZ_{Q_E}=\cZ_{\M_E}.$

Let $\pxx$ \index{$\pxx$}  denote the \df{free algebra}  of noncommutative polynomials
 in the letters $x=\{x_1,\dots,x_g\}.$ 
Thus elements of $\pxx$ are finite $\C$-linear combinations of words
 in the letters $\{x_1,\dots,x_g\}.$ For each positive integer $n$,
 an  element $p$ of $\pxx$ naturally 
 induces a function, also denoted $p$, mapping $M_n(\C)^g\to M_n(\C)$
 by replacing the letter $x_1,\dots,x_g$  by $n\times n$ matrices
 $X_1,\dots,X_g.$ In this way,
 we view $p$ as a function on the disjoint union of the sets $M_n(\C)^g$ 
 (parameterized by $n$). 
When $e>1$ there are non-constant $F\in \pxx^{e\times e}$  that are
invertible, and
 the appropriate analog of irreducible elements of  $\pxx^{e\times e}$
 reads as follows.
 An $F\in \pxx^{e\times e}$ with $\det f(0)\neq 0$ is an \df{atom} \cite[Chapter 3]{Coh}
 if $F$ does  not factor; i.e., $F$ cannot be
written
 as $F=F_1F_2$ for some non-invertible $F_1,F_2\in\pxx^{e\times e}$.
As a consequence of Lemma \ref{l:trans}\eqref{it:lt5}
  below, we will see that
 if $Q_E$ is an atom, $\ker(E)=\{0\}$ and $\ker(E^*)=\{0\}$,
 then $E$ is \bmin.

\begin{theorem}
	\label{thm:sdmain}
 Suppose $A\in M_d(\C)^g$, $B\in M_e(\C)^g$ and
\begin{enumerate}[\rm (a)]
 \item\label{it:sdmaina} $\cD_A$ is bounded;
 \item\label{it:sdmainb}
 $Q_A$ and $Q_B$ are atoms, $\ker(B)=\{0\}$ 
 and  $A^*$ is \bmin;
\item $t>1$ and $p:\interior(t\cD_A)\to M(\C)^g$ and $q:\interior(t\cD_B)\to M(\C)^g$
are free bianalytic mappings;
\item $p(0)=0,$ $p^\prime(0)=I$, $q(0)=0$ and $q^\prime(0)=I$.
\end{enumerate}
 If   $q(p(X))=X$ and $p(q(Y))=Y$ for $X\in \cD_A$
 and $Y\in \cD_B$ respectively, then 
 $p$ is convexotonic,  $A$ and $B$ are of the same size
 $d=e$, and there exist  $d\times d$ unitary matrices
 $Z$ and $M$ and a convexotonic $g$-tuple $\Xi$  such that 
\begin{enumerate}[\rm (1)]
 \item\label{it:secret0}
 $p$ is the   convexotonic map
     $p=x(I-\Lambda_\Xi(x))^{-1}$, where
  for each $1\le j,k\le g$, 
 \begin{equation}
   \label{it:secretalg0}
  A_k (Z-I)  A_j = \sum_s (\Xi_j)_{k,s} A_s;
\end{equation}
in particular, 
     the tuple $R=(Z-I)A$ spans an algebra with multiplication table $\Xi$,
\[
 R_k R_j = \sum_s (\Xi_j)_{k,s} R_s;
\]
\item \label{it:BMZEM0} $B_j= M^* ZA_jM$ for $1\leq j\leq g$.
\end{enumerate}
\end{theorem}

\begin{proof}
See Section \ref{sec:proof1}.
\end{proof}

\ssec{Geometry of the boundary vs irreducibility}
\label{sec:introGeom}
At the core of the proofs of our main theorems in this paper
 is a richness of the
geometry of the boundary, $\partial \cB_E,$  of a  \index{$\partial \cB_E$}
spectraball, $\cB_E$.
We shall show that  a  (rather ungainly) key  geometric  property of the 
boundary of $\cB_E$ is 
equivalent to the defining polynomial $Q_E$ of $\cB_E$ 
 being an atom and $\ker(E)=\{0\}.$

To describe the geometric structure involved, fix $E\in M_{d\times e}(\C)^g.$
 The \df{detailed boundary} \index{$\widehat{\partial \cB_E}$}
 $\widehat{\partial \cB_E}$  of $\cB_E$ is the sequence of sets 
\[
\widehat{\partial \cB_E }(n) :=
\left\{
(X,v)\in\mat{n}^g\times [\C^e\otimes \C^n]  \midx  X \in \partial \cB_E , \
 v\ne 0, \,  \Qre_E (X, X^*) v = 0  \right\}.
\]
For $n\in\N$, let \df{$\widehat{\partial^1 \cB_E}(n)$} denote 
 the points $(X,v)$ 
 in $\widehat{\partial \cB_E}(n)$ such that $\dim\ker \Qre_E(X,X^*)=1.$  For a
 vector $v \in \C^e\otimes \C^n = \C^{en},$ partitioned as
 $$ v = \bem v_1 \\ v_2 \\ \vdots \\ v_n
      \eem  $$
      for $v_k \in \C^e$, 
define $\pi(v) = v_1$.
The geometric property important to mapping studies is that
$\pi(\widehat{ \partial^1 \cB_E})$ contain enough vectors to span
$\C^e$ or better yet to hyperspan $\C^e$.
Here a set $\{u^1,\dots,u^{e+1}\}$ of vectors in $\C^e$
\df{hyperspans} $\C^e$  provided  each
$e$ element subset spans; i.e., is a basis of $\C^e.$

\begin{thm}
\label{t:hairspans}
Let $E\in M_{d\times e}(\C)^g$. Then
\ben[\rm (1)]
\item\label{i:hairspans}
 $E$ is \bmin  if and only if
$\pi(\widehat{ \partial^1 \cB_E})$
 spans $\C^e$.

\item
\label{i:hairhyperspans}
 $Q_E$ is an atom and $\ker (E)=(0)$
if and only if
$\pi(\widehat{ \partial^1 \cB_E} )$
contains a hyperspanning set for  $\C^e$.\looseness=-1
\een
\end{thm} 

\begin{proof} Part \eqref{i:hairspans} is established in Proposition
\ref{p:hairspans},
while \eqref{i:hairhyperspans} is Proposition \ref{p:hairhyperspans}.
\end{proof}

\subsection{A Nullstellensatz}
\label{sec:null}
 Theorem \ref{t:pencil-ball-alt} uses the following Nullstellensatz
 whose proof depends upon Theorem \ref{t:hairspans}.

\begin{prop}
 \label{p:HKVpBergman}
   Suppose $E=(E_1,\dots,E_g)\in M_{d\times e}(\C)^g$ is \bmin
   and $V\in \pxx^{\dd \times e}$
   is a (rectangular) matrix polynomial.
   If $V$ vanishes on $\widehat{\partial \cB_E};$ that is $V(X)\gamma=0$
  whenever $(X,\gamma)\in \widehat{\partial \cB_E}$, then $V=0.$
\end{prop}

\begin{proof}
 See Subsection \ref{sec:aux}.
\end{proof}

\subsection{An overview of the proof of Theorem~\ref{t:pencil-ball-alt}}
We are now in a position to convey, in broad strokes, an outline of the 
proof  of Theorem \ref{t:pencil-ball-alt}.  The conversely  direction
is an immediate consequence of Proposition  \ref{prop:properobvious} 
(see Corollary~\ref{c:pba-converse}) of Section 
 \ref{s:prelims}. Its proof reflects the fact that convexotonic maps are bianalytic
between certain special spectrahedral pairs. Proposition \eqref{prop:properobvious} is also the  starting
 point for the proof of the more challenging converse.  Given the tuple  $A,$ let
 $J=(J_1,\dots, J_h)$ denote a basis for the algebra  spanned by
  $A$ with $J_j=A_j,$ for $1\le j\le g.$ Proposition \ref{prop:properobvious} says 
 that $\cD_J$ and $\cB_J$ are bianalytic via the convexotonic
 map  associated 
  to  the convexotonic $h$-tuple $\Xi$ determined by the tuple $J$ via  
  equation \eqref{e:JgetsXi} (with $J$ in place of $A$). Starting with the free bianalytic map $f:\cB_E\to \cD_A$, observe that
 $G=\varphi\circ \iota \circ f:\cB_E\to \cB_J$ is a free proper map  satisfying $G(0)=0$ and
  $G^\prime(0) = \begin{pmatrix} I_g & 0_{g\times(h-g)} \end{pmatrix},$ 
 where $\iota:\cD_A\to \cD_J$ is the inclusion,  since 
  $\varphi(0)=0$ and $\varphi^\prime(0)=I_h.$   An argument  
  that uses Proposition \ref{p:HKVpBergman} produces a representation for $G$
   that can be thought of as an analog of the Schwarz Lemma (see equation \eqref{e:agr2}).
   In simple cases,
\begin{equation}
 \label{e:bs}
  G(x) = \begin{pmatrix} x & 0 \end{pmatrix}
\end{equation}
 from which it follows that the $g$-tuple $\widehat{\Xi}\in M_g(\C)^g$ defined by
\[
 (\widehat{\Xi}_j)_{s,t}  = (\Xi_j)_{s,t}, \ \ 1\le j,s,t \le g
\]
 is convexotonic and thus $A$  spans an algebra. 
 Hence  $h=g,$ the map $\varphi$ (and hence $\varphi^{-1}$) is
 convexotonic  and   $f=\varphi^{-1}.$
  In general  only a weaker version of equation \eqref{e:bs} holds, an inconvenience 
 that does not conceptually
 alter the argument, but one that does make the proof more technical.

\section{Free rational maps and convexotonic maps}
\label{s:prelims}
In this section we review the notions of a free set and free rational function 
and provide further background on  free functions and mappings. In particular,
convexotonic maps are seen to be free rational mappings. In Subsection \ref{ssec:ct} we show
how algebras of matrices give rise to convexotonic bianalytic maps between free spectrahedra.
See Theorem \ref{t:1point1}.

\subsection{Free sets, free analytic functions and mappings}
Let $M(\C)^g$ \index{$M(C)^g$} denote the sequence $(M_n(\C)^g)_n$.  
A \df{subset} $\Gamma$ of $M(\C)^g$ is a sequence $(\Gamma_n)_n$ where 
$\Gamma_n \subseteq M_n(\C)^g$. (Sometimes we  write $\Gamma(n)$ in place of $\Gamma_n.$)  
The subset $\Gamma$ is a \df{free set} if it is closed under direct sums and simultaneous unitary similarity.
Examples of such sets include spectraballs and free spectrahedra introduced above.
 We say the free set $\Gamma=(\Gamma_n)_n$ is \df{open} if each $\Gamma_n$ is open.
 Generally adjectives are applied level-wise to free sets unless noted otherwise.

A \df{free function}   $f:\Gamma\to M(\C)$ is a sequence of functions
 $f_n:\Gamma_n\to M_n(\C)$ that \df{respects intertwining}; that is,
 if $X\in\Gamma_n$, $Y\in\Gamma_m$, $T:\C^m\to\C^n$,
  and
 \[
  XT=(X_1T,\dots, X_gT)
   =(T Y_1,\dots, TY_g)=T Y,
 \]
  then $f_n(X) T =  T f_m(Y)$. 
In the case $\Gamma$ is open, $f$ is \df{free analytic} if each $f_n$
 is analytic in the ordinary sense.  
We refer the reader to   \cite{Voi04,KVV14,AM14,AM15,HKMforbill,HKM11a}
for a fuller discussion of free sets and functions.
For further results, not already cited, on 
free bianalytic and proper free analytic maps see
\cite{Po2,MS,KS,HKMSlinglend,HKM11b,SSS18}
and the references therein.

A \df{free mapping} $p:\Gamma\to M(\C)^h$ is a tuple
 $p=\begin{pmatrix} p^1 & p^2 & \cdots & p^h\end{pmatrix}$
 where each $p^j:\Gamma \to M(\C)$ is a free function.
 The free mapping $p$ is \df{free analytic} if each $p^j$ is a free analytic
  function.    If $h=g$ and $\Delta \subset M(\C)^g$ is a free set,
  then $p:\Gamma\to \Delta$ is \df{bianalytic} if 
  $p$ is analytic and $p$ has an inverse, that is 
  necessarily free and analytic,  $q:\Delta\to \Gamma$.

\subsection{Free rational functions and mappings}
Based on the results of \cite[Theorem 3.1]{KVV09}
and \cite[Theorem 3.5]{Vol17} a \df{free rational function regular at $0$}
 can, for the purposes of this article,
 be defined with minimal overhead as an expression of the form
\begin{equation}
\label{e:realizer}
r(x)= c^* \big(I-\Lambda_S(x)\big)^{-1} b,
\end{equation}
where, for some positive integer $s$, we have
 $S\in M_s(\C)^g$ and  $b,c\in\C^s.$
The expression $r$ 
 is known as a \df{realization}.\index{Realization} Realizations are easy
to manipulate and a powerful tool as developed in the series of papers
 \cite{BGM05,BGM06a,BGM06b} of  Ball-Groenewald-Malakorn;
see also \cite{Coh,BR}.
The realization $r$ is  evaluated in the obvious fashion on 
a tuple $X\in M_n(\C)^g$ as long as $I-\Lambda_S(X)$ is invertible. 
Importantly, free rational functions are  free analytic.
 
Given a tuple $T\in M_k(\C)^g$, let \index{$\mathscr{I}$}
\begin{equation}
\label{e:msI}
\sI_T = \{X\in M(\C)^g: \det(I-\Lambda_T(X))\ne 0\}.
\end{equation}
A realization 
$ \tilde{r}(x) = \tilde{c}^* (I-\Lambda_{\widetilde{S}})^{-1}\tilde{b}$
is \df{equivalent} to the realization $r$ as in \eqref{e:realizer} if $r(X)=\tilde{r}(X)$ for
$X\in \sI_{S}\cap \sI_{\tilde{S}}$.  A free rational function is an equivalence
class of realizations and 
we identify 
$r$ with its equivalence class and refer to it as a free rational function.
 The  realization \eqref{e:realizer} is \df{minimal}
if $s$ is the minimum size among all realizations equivalent to $r$. 
By \cite{KVV09,Vol17}, if $r$ is minimal and $\tilde{r}$ is equivalent
to $r$, then  $\sI_S \supset \sI_{\widetilde{S}}$.  Moreover, the results
in \cite{Vol17} explain precisely, in terms of evaluations,
the sense in which $\sI_S$ deserves to be called
 the \df{domain of the free rational function} $r,$ 
 denoted $\dom(r)$.

A free polynomial $p$ is a free rational function regular at $0$ and, as is well
known, its domain is $M(\C)^g$. 
 If $f$ and $g$ are free rational functions regular at $0$, then
so are $f+g$ and $f g.$ Moreover, $\dom(f+g)$ and $\dom(fg)$ both contain
$\dom(f)\cap \dom(g)$ as a consequence of \cite[Theorem 3.10]{Vol18}.
Free rational functions regular at $0$ are determined by their evaluations near $0$; that
is if $f(X)=g(X)$ in some neighborhood of $0$ in $\dom(f)\cap \dom(g)$, then $f=g$.
In what follows, we often omit {\it regular at $0$} when it is understood from context.
We refer the reader to \cite{Vol17,KVV09} for a fuller discussion of
the domain of a free rational function.

A \df{free rational mapping} $p$ is a tuple of rational functions $p=\begin{pmatrix} p^1 & \cdots & p^g\end{pmatrix}.$ 
 The domain of $p$ is the intersection of the domains
of the $p^j$.
By \cite[Proposition 1.11]{AHKM18},
if $r$ is a free rational mapping  with no singularities on 
a bounded free spectrahedron
$\cD_A$, then
there is a %
$t>1$ such that $r$ has no singularities on $t\cD_A$.

\subsection{Algebras and convexotonic maps}\label{ssec:ct}
Theorem \ref{t:1point1} below  is an expanded version of \cite[Theorem 1.1]{AHKM18}.
To begin we discuss a sufficient condition
for a tuple $X\in M_n(\C)^g$ to lie in $\dom(p)$, the domain of a convexotonic mapping
\[
 p=\begin{pmatrix} p^1 & \cdots & p^g\end{pmatrix} 
  = x(I-\Lambda_\Xi(x))^{-1}.
\]
Since
\[
p^j = %
	\sum_{k=1}^g x_k \left [ e_k^* (I-\Lambda_{\Xi}(x))^{-1} e_j \right],
\]
it follows that $\sI_{\Xi}\subset \cap \dom(p^j) = \dom(p)$. 
Now suppose $R\in M_N(\C)^g$ and $f_{k,s,a,b},g_{k,s,a,b}, h_{k} \in \pxx$  and let
$r^k$ denote the free rational function 
\[
r^k(x) =  %
  \sum_{s,a,b} f_{k,s,a,b}(x) [e_a^*\left (I-\Lambda_R(x)\right )^{-1}e_b]\, g_{k,s,a,b}(x) + h_{k}.
\]
If $r^j=p^j$ in some neighborhood of $0$ lying in $\sI_{\Xi} \cap \sI_{R}$,
then $r^j$ and $p^j$ represent the same free rational function. In particular, 
$\sI_R \subset \dom(p^j)$ and therefore $\sI_{R} \subset \dom(p)$.

Let $\exterior(\cD_B)$ denote the sequence \index{$\exterior$}
$(\exterior(\cD_B(n)))_n$ where $\exterior(\cD_B(n))$ is the complement of
$\cD_B(n)$. Likewise let $\partial \cD_B(n)$ denote the boundary of
$\cD_B(n)$ and let $\partial \cD_B$  denote the sequence 
$(\partial \cD_B(n))_n$. \index{$\partial \cD_B$}

\begin{thm}
\label{t:1point1}
Suppose $\Aalt,\Balt\in M_r(\C)^g$ are linearly independent,
$U\in M_r(\C)^g$ is unitary and  $\Balt=U\Aalt$.  If there exists
a tuple $\Xi\in M_g(\C)^g$ such that
\[
 \Aalt_\ell (U-I)\Aalt_j = \sum_{s=1}^g (\Xi_j)_{\ell,s} \Aalt_s,
\]
then $\Xi$ is convexotonic and the convexotonic maps
$p$ and $q$ associated to $\Xi$ are 
bianalytic maps between $\cD_\Aalt$ and $\cD_{\Balt}$ in the following sense.
\begin{enumerate}[\rm (a)]
\item \label{i:11a}
   $\interior(\cD_\Aalt)\subset \dom(p)$, $\interior(\cD_{\Balt})\subset \dom(q)$; and
  $p:\interior(\cD_\Aalt)\to \interior(\cD_{\Balt})$ is bianalytic.
\item \label{i:11d}
 If $X\in \exterior(\cD_\Aalt)\cap \dom(p)$, then $p(X)\in \exterior(\cD_{\Balt})$.
\item \label{i:11e}
   If $X\in \partial \cD_\Aalt \cap \dom(p)$, then $p(X) \in \partial \cD_{\Balt}$.
\item  \label{i:11b}
   If $\cD_{\Balt}(1)$ is bounded, then $\cD_\Aalt\subset \dom(p)$.
\end{enumerate}
\end{thm}

Before taking up the proof of Theorem \ref{t:1point1}, 
we  prove the following proposition and collect a few of its consequences
that  will be used  in the sequel.

\begin{prop}[{\cite[Proposition 1.3]{AHKM19}}]
\label{prop:properobvious}
  Suppose $J\in M_d(\C)^g$ is linearly independent and spans an algebra with convexotonic tuple $\Xi$
 (as in equation \eqref{e:JgetsXi} with $J$ in place of $A$). Let
  $p=x(I-\Lambda_\Xi(x))^{-1}$ and $q=x(I+\Lambda_\Xi(x))^{-1}$ denote the
 corresponding convexotonic maps.
\begin{enumerate}[\rm (i)]
 \item  \label{i:obv1} 
   $\interior(\cB_J)\subset \dom(p)$
 and $p:\interior(\cB_J)\to \interior(\cD_J).$ 
 \item \label{i:obv4} $\cD_J\subset \dom(q)$ and 
    $q:\interior(\cD_J)\to \interior(\cB_J)$ 
    and  $q(\partial \cD_J)\subset \partial \cB_J$. 
\item \label{i:obv6} $p:\interior(\cB_J)\to\interior(\cD_J)$ and 
  $q:\interior(\cD_J)\to\interior(\cB_J)$ are birational inverses
   of one another.
\item \label{i:obv2}  If $X\in\dom(p)$, but $X\notin \interior(\cB_J),$ then $p(X)\notin\interior(\cD_J)$.
 \item  \label{i:obv3} If $\cD_J$ is bounded, then the domain of $p$ contains $\cB_J$ and
 $p(\partial \cB_J)\subset  \partial \cD_J.$

\item \label{i:obv5}  If  $Y\in \dom(q),$ but $Y\notin \cD_J,$ then $q(Y)\notin \cB_J.$
\end{enumerate}
\end{prop}

\begin{lemma}
	\label{lem:inverseok}
	Suppose $F\in M_d(\C)^g$.
	If $I+\Lambda_F(X)+\Lambda_F(X)^* \succeq 0$, then $I+\Lambda_F(X)$
	is invertible. 
\end{lemma}

\begin{proof}
	Arguing the contrapositive, suppose $I+\Lambda_F(X)$ is not invertible. 
	In this case
	there is a unit vector $\gamma$ such that
	\[
	\Lambda_F(X)\gamma = -\gamma.
	\]
	Hence,
	\[
	\langle (I+\Lambda_F(X)+\Lambda_F(X)^*)\gamma,\gamma\rangle
	= \langle \Lambda_F(X)^*\gamma,\gamma\rangle = 
	\langle \gamma,\Lambda_F(X)\gamma\rangle = 
	-1. \qedhere
	\]
\end{proof}

\begin{lem}\label{lem:onehalf}
	Let $T\in M_d(\C)$. Then
	\begin{enumerate}[\rm (a)]
		\item \label{i:oha}  $I+T+T^*\succeq0$ if and only if $I+T$ is invertible and
		$\|(I+T)^{-1}T\|\leq1$;
		\item \label{i:ohb} $I+T+T^*\succ0$ if and only if
		$I+T$ is invertible and $\|(I+T)^{-1}T\|<1$.
               \item \label{i:ohc} If $\|T\|<1$, then $I-T$ is invertible and
                     $I+(I-T)^{-1}T +  \left ( (I-T)^{-1}T\right )^*\succ 0.$
                \item \label{i:ohd} If $\|T\|=1$ and $I-T$ is invertible, then 
                 $I+(I-T)^{-1}T +  \left ( (I-T)^{-1}T\right )^*$ is 
                   positive semidefinite and singular. 
	\end{enumerate}
\end{lem}

\begin{proof}
	Item \eqref{i:oha} follows from the chain of equivalences,
	\[
	\begin{split}
	\|(I+T)^{-1}T\|\leq1 \quad & \iff
	\quad 
	I- \big((I+T)^{-1}T\big)\big((I+T)^{-1}T\big)^*\succeq0
	\\
	& \iff \quad 
	I- (I+T)^{-1}T T^*(I+T)^{-*}\succeq0  \\
	& \iff \quad
	(I+T)(I+T)^* - TT^*\succeq0 \\
	& \iff \quad
	I+T+T^*\succeq0.
	\end{split}
	\]
	The proof of item \eqref{i:ohb}  is the same. 

        The proof of \eqref{i:ohc}
        is routine. Indeed, it is immediate that $I-T$ is invertible and
\[
   I+(I-T)^{-1}T + \left ( (I-T)^{-1}T\right )^*
      = (I-T)^{-1}\, \left (I-TT^*\right) \, (I-T)^{-*} \succ 0.
\]
 The proof of item \eqref{i:ohd} is similar.
\end{proof}

\begin{proof}[Proof of Proposition \ref{prop:properobvious}]
 Compute
\[
\begin{split}
\Lambda_{\JJ}(\rr(x))\, \Lambda_\JJ(x)
& =   \sum_{s,k=1}^g q^s(x) x_k \JJ_s \JJ_k 
= \sum_{j=1}^g  \sum_{s=1}^g q^s(x) \left [ \sum_{k=1}^g x_k (\Xi_k)_{s,j} \right ] \JJ_j \\
& =  \sum_{j=1}^g \sum_{s=1}^g q^s(x) (\Lambda_{\Xi}(x))_{s,j} \JJ_j 
=  \sum_{j=1}^g \sum_{t=1}^g x_t  \left [ \sum_{s=1}^g (I+\Lambda_\Xi(x))^{-1}_{t,s} (\Lambda_{\Xi}(x))_{s,j} \right ] \JJ_j \\ 
& =  \sum_{j=1}^g   \sum_{t=1}^g x_t [(I+\Lambda_\Xi(x))^{-1} \Lambda_\Xi(x)]_{t,j} \JJ_j.
\end{split}
\]
Hence,
\[
\Lambda_\JJ(\rr(x))\, (I+\Lambda_\JJ(x))
=  \sum_{j=1}^g   \sum_{t=1}^g x_t [(I+\Lambda_\Xi(x))^{-1} (I+\Lambda_\Xi(x))]_{t,j} \JJ_j 
=  \Lambda_{\JJ}(x).
\]
Thus, as free (matrix-valued) rational functions regular at $0$,
\begin{equation}
\label{eq:LJr}
\Lambda_{\JJ}(\rr(x)) = (I+\Lambda_{\JJ}(x))^{-1} \, \Lambda_{\JJ}(x)=:F(x).
\end{equation}

Since $J$ is linearly independent, given $1\le k\le g$, 
there is a linear functional $\lambda$ such that 
$\lambda(J_j)=0$ for  $j\ne k$ and $\lambda(J_k)=1$. Applying
$\lambda$ to equation \eqref{eq:LJr}, gives
\begin{equation}
\label{eq:lambdaLJr}
q^k(x) = \lambda(F(x)).
\end{equation}
Since $\lambda(F(x))$ is a free rational function whose domain contains
\[
\mathscr{D} = \{X: I+\Lambda_{\JJ}(X) \mbox{ is invertible}\},
\]
the same is true for $q^k$. (As a technical matter, each side of equation
\eqref{eq:lambdaLJr} is a rational expression. Since they are defined and
agree on a neighborhood of $0$, they determine the same free rational
function. It is the domain of this rational function that contains
$\mathscr{D}.$ See \cite{Vol17}, and also \cite{KVV09}, for full details.)
By Lemma \ref{lem:inverseok}, $\mathscr{D}$ contains $\cD_\JJ,$
(as $X\in \cD_\JJ$ implies $I+\Lambda_{\JJ}(X)$ is invertible). Hence
the domain of the free rational mapping  $\rr$ contains   $\cD_\JJ$.
By Lemma \ref{lem:onehalf} and equation \eqref{eq:LJr}, $\rr$ maps the interior of $\cD_{\JJ}$ into the 
interior of $\cB_{\JJ}$ and the boundary of $\cD_{\JJ}$ into
the boundary of $\cB_{\JJ}$. Thus item \eqref{i:obv4} is proved.

Similarly, 
\begin{equation}
\label{eq:LJp}
 (I-\Lambda_{\JJ}(x))^{-1} \, \Lambda_\JJ(x) =\Lambda_{\JJ}(\pp(x)).
\end{equation}
 Arguing as above shows
 the domain of $\pp$ contains the set
\[
\mathscr{E} =\{X: I-\Lambda_{\JJ}(X) \mbox{ is invertible}\},
\]
which in turn contains $\interior(\cB_{\JJ})$
(since $\|\Lambda_{\JJ}(X)\|<1$ allows for an
application of Lemma \ref{lem:onehalf}). By Lemma \ref{lem:onehalf}
and equation \eqref{eq:LJp},
 $\pp$ maps the interior of
$\cB_{\JJ}$ into the interior of $\cD_{\JJ},$ 
proving item \eqref{i:obv1}.  Since $p$ and $q$ are 
formal rational inverses of one another, it follows 
 from items \eqref{i:obv1} and \eqref{i:obv4}
 that they are inverses of one another as
 maps between $\cD_J$ and $\cB_J,$ proving 
 item \eqref{i:obv6}.   Further, if  $X$ is in the boundary
of $\cB_{\JJ}$, then for $t\in \mathbb C$ and $|t|<1$, 
we have $\pp(tX)\in \interior(\cD_{\JJ})$ and
\[
 \Lambda_J(\pp(tX))= (I-\Lambda_{\JJ}(tX))^{-1} \, \Lambda_\JJ(tX).
\]
Assuming $\cD_{\JJ}$ is bounded, it follows that
  $I-\Lambda_{\JJ}(X)$ is invertible and thus, by Lemma \ref{lem:onehalf},
  $X$ is in the domain
of $\pp$ and $\pp(X)$ is in the boundary of $\cD_{\JJ},$ proving item \eqref{i:obv3}.
 Finally, to prove item \eqref{i:obv2}, suppose $X\notin\interior(\cB_J),$ but
 $p(X)\in \interior(\cD_J).$ By item \eqref{i:obv1}, there is a $Y\in \interior(\cB_J)$
 such that $p(Y)=p(X).$ By item \eqref{i:obv4}, $p(Y)=p(X)\in \dom(q)$ and therefore,
 $Y=q(p(Y))=q(p(X))=X,$ a contradiction. The proof of \eqref{i:obv5} is similar. 
\end{proof}

The converse  portion of Theorem \ref{t:pencil-ball-alt} is an immediate consequence
 of Proposition \ref{prop:properobvious}, stated below as Corollary \ref{c:pba-converse}.

\begin{cor}
\label{c:pba-converse}
 Suppose $E\in M_{d\times e}(\C)^g$ is linearly independent, $r\ge \max\{d,e\}$,
the $r\times r$ matrix  $U$ is unitary and 
\[
 A= U\begin{pmatrix} 0 & E \\0 &0 \end{pmatrix}.
\]
If  there exists a tuple $\Xi\in M_g(\C)^g$ such that equation \eqref{e:JgetsXi}
holds, then $\Xi$ is convexotonic and the associated convexotonic map 
 $p$ is a bianalytic mapping $\interior(\cB_E)=\interior(\cB_A)\to \interior(\cD_A)$.
 Moreover, $\cD_A\subset \dom(q)$ and 
$q(\partial \cD_A)\subset  \partial \cB_A,$ where $q=x(I+\Lambda_\Xi(x))^{-1}$
 is the inverse of $p.$
\end{cor}

\begin{proof}
By the definition of $A$ we have $\cB_A=\cB_E$. The rest follows by Proposition \ref{prop:properobvious}.
\end{proof}

In the case $\JJ$ does not span an algebra, we have the following variant
of Proposition \ref{prop:properobvious}. It says that
each free spectrahedron can be mapped properly 
to a bounded spectraball and is used in the proof of 
 Theorem \ref{t:pencil-ball-alt}.
 Recall a mapping  between topological spaces 
is \df{proper} if the inverse image of each  compact sets is compact.
Thus, for free open sets $\cU\subseteq M(\C)^g$ and $\cV\subseteq M(\C)^h$, 
 a free mapping $f:\cU\to \cV$ is proper if each
$f_n :\cU_n\to \cV_n$ is proper. For perspective,
given  subsets $\Omega\subset \C^g$
 and $\Delta \subset \C^h$ (that are not necessarily closed), 
 and a proper  analytic  map $\psi:\Omega\to\Delta$, 
 if   $\Omega\ni z^j \to \partial \Omega$, then $\psi(z^j) \to \partial \Delta.$
 \cite[page 429]{Kra92}.\looseness=-1

\begin{cor}\label{cor:obvious}
Let $A\in M_d(\C)^g$ and assume $A$ is  linearly independent.
Let $C_{g+1},\ldots,C_h\in M_d(\C)$ be any matrices such that 
 the tuple $\JJ=(\JJ_1, \dots, \JJ_{h})=(A_1,\dots,A_g,C_{g+1},\dots,C_h)$ 
is a basis for the algebra  generated by the tuple $A$. 
Let $\Xi\in M_h(\C)^h$ denote the convexotonic tuple associated to $J$, let 
$p:\interior(\cB_J)\to\interior(\cD_J)$ denote the corresponding convexotonic map, 
let $q$ denote the inverse of $p,$  and let 
 $\iota:\interior(\cD_A)\to \interior(\cD_{\JJ})$ denote the inclusion.  
 Then  we have the commutative diagram
\begin{center}
\begin{tikzcd}
& \interior(\cB_J) \ar[dd,"p"',"\cong"]
\\  {} \\
\interior(\cD_A)  
\ar[uur,dashed, "f"]
\ar[r, hook,"\iota"']
& \interior(\cD_J)
\end{tikzcd}
\end{center}
and 
 the mapping
 \begin{equation} \label{e:GG}
\GG(x) = \qq\circ \iota(x)=
  \begin{pmatrix} x_1 & \cdots & x_g & 0 & \cdots &0 \end{pmatrix} 
    \, \Big ( I+\sum_{j=1}^g \Xi_j x_j \Big )^{-1}
    \end{equation}
is (injective)  proper and  
extends analytically to a neighborhood of $\cD_A$.
\end{cor}

\begin{proof}
 By Proposition \ref{prop:properobvious}, $p:\interior(\cB_J)\to \interior(\cD_J)$
 is birational and the domain of its inverse $q$ contains $\cD_J$ and 
 maps $\partial \cD_J$ into $\partial \cB_J.$ In particular $q$ is proper.

 Given $X\in M(\C)^g$,  letting $Y=\begin{pmatrix} X& 0 \end{pmatrix}$,
\[
 \Lambda_J(Y)
 = \sum_{j=1}^h J_j \otimes Y_j = \sum_{j=1}^g A_j\otimes X_j.
\]
 Hence $\LRE_J(\begin{pmatrix} X& 0 \end{pmatrix}) = \LRE_A(X)$
 and it follows that $X\in \interior(\cD_A)$ if
 and only if $Y\in \interior(\cD_J)$.
 Hence, we obtain a mapping $\iota:\interior(\cD_A)\to \interior(\cD_J)$
 defined by $\iota(X)=Y$.

 Fix $m\in\mathbb N$ and  suppose $K\subset \interior(\cD_J(m))$ is compact and let 
  $\backK=\iota^{-1}(K)\subset \cD_A(m)$.
 If   $(X^n)$ is a sequence from $\backK$, then 
  $Y^n = \begin{pmatrix} X^n& 0 \end{pmatrix}$ is  a sequence
 from $K$. Since $K$ is compact,  $(Y^n)_n$ has a subsequence $(Y^{n_j})_j$
 that converges to some $Y\in K$. It follows that 
  $Y=\begin{pmatrix} X& 0 \end{pmatrix}\in K\subset \interior(\cD_J)$ 
 for some $X\in \backK$. Hence $(X^{n_j})_j$ converges to $X$
 and we conclude that $\backK$ is compact. Thus $\iota$ 
 is proper.  Since $q$ is also proper, $f=q\circ \iota$ is too. 
 Letting $z=(z_1,\dots,z_{h})$ denote an $h$ tuple of freely noncommuting 
indeterminates,
\[
\qq(z) = z (I+\Lambda_{\Xi}(z))^{-1}
\]
and thus $f$ takes the form of equation \eqref{e:GG}.
\end{proof}

\subsection{Proof of Theorem \ref{t:1point1}}
\begin{lemma}
\label{l:secret} 
   Suppose $\FG\in M_{d\times e}(\C)^g$  
 is linearly independent,
 $C\in M_{e\times d}(\C)$ and  $\Psi\in M_g(\C)^g$. 
 If\looseness=-1
\[
  \FG_\ell C \FG_j =\sum_{s=1}^g (\Psi_j)_{\ell,s} \FG_s,
\]
 then the tuple $\Psi$ is convexotonic.  Moreover, letting
 $T=CG\in M_e(\C)^g$,
\begin{equation}
\label{e:conmore}
 G_\ell T^\alpha = \sum_{s=1}^g (\Psi^\alpha)_{\ell,s} G_s.
\end{equation}

 In particular,
 if $A\in M_d(\C)^g$ is linearly independent and spans an algebra,
 then the tuple $\Psi$ uniquely determined by equation 
 \eqref{e:JgetsXi} is convexotonic.
\end{lemma}

Note that the hypothesis implies $T$ spans an algebra
 (but not that $T$ is linearly independent).

\begin{proof}
    Routine calculations give
\[
 (\FG_\ell T_j) T_k = \sum_{t=1}^g (\Psi_j)_{\ell,t} \FG_t\, T_k 
 = \sum_{s,t=1} (\Psi_j)_{\ell,t} (\Psi_k)_{t,s} \FG_s
  = \sum_s (\Psi_j\, \Psi_k)_{\ell,s} \FG_s.
\]
 On the other hand
\[
 \FG_\ell (T_j T_k) = \FG_\ell C (\FG_j T_k) =  \sum_t \FG_\ell (\Psi_k)_{j,t}T_t 
  = \sum_{s,t}  (\Psi_t)_{\ell,s} (\Psi_k)_{j,t} \FG_s.
\]
 By independence of $\FG$, 
\[
(\Psi_j \Psi_k)_{\ell,s} =  \sum_{t}  (\Psi_k)_{j,t}  (\Psi_t)_{\ell,s}
\]
 and therefore
\[
\Psi_j \Psi_k  = \sum_{t}  (\Psi_k)_{j,t}  \Psi_t.
\] 
Hence $\Psi$ is convexotonic. 

A straightforward induction argument establishes the identity \eqref{e:conmore}.
\end{proof}

\begin{prop}
\label{prop:Nstatz}
Suppose $A,B\in M_{\rt}(\C)^g$ are linearly independent,
$U\in M_{\rt}(\C)^g$ is unitary,  $B=UA$ and there exists
a convexotonic tuple $\Xi\in M_g(\C)^g$ such that
\[
 A_\ell (U-I)A_j  = \sum_{s=1}^g (\Xi_j)_{\ell,s} A_s.
\]
Letting $p$ denote the associated convexotonic map, 
$R$ the tuple $(U-I)A=B-A$  and
\[
  Q(x) = I-\Lambda_R(x),
\]
\begin{enumerate}[\rm (a)]
\item \label{i:Nstatz1} we have
\begin{equation*}
\big(I+\Lambda_B(p(x))\big)Q(x)=I+\Lambda_A(x); 
\end{equation*}
\item \label{i:Nstatz3} if $Z\in \dom(p)$, then 
\begin{equation}
\label{e:Nstatzpre}
\big(I+\Lambda_B(p(Z)) \big)Q(Z) = I+\Lambda_A(Z),
\end{equation}
 and
\begin{equation}
 \label{e:Nstatz}
     Q(Z)^* \LRE_B(p(Z)) Q(Z) = \LRE_A(Z);
\end{equation}
\item \label{i:Nstatz4}   if $Z\in M(\C)^g$ and  $Q(Z)$ is invertible, then $Z\in \dom(p)$
 and equation \eqref{e:Nstatz} holds. 
\end{enumerate}
\end{prop}

\begin{proof}
Item \eqref{i:Nstatz1} is straightforward, so we merely outline a proof. From Lemma \ref{l:secret},
 for words $\alpha$ and $1\le j\le g$,
\[
 A_j R^\alpha = \sum_{s=1}^g (\Xi^\alpha)_{j,s} A_s.
\]
Hence 
\[
 B_j R^\alpha =  \sum_{s=1}^g (\Xi^\alpha)_{j,s} B_s,
\]
from which it follows that, 
letting $\{e_1,\dots, e_g\}$ denote the standard basis for $\C^g$,
\[
\begin{split}
 \Lambda_B(p(x)) & = \sum_s B_s p^s(x) 
 = \sum_{s=1}^g \sum_{j=1}^g x_j \, [e_j^* (I-\Lambda_\Xi(x))^{-1} e_s] \\
& = \sum_{n=0}^\infty \, \sum_{j,s=1}^g  x_j \, [e_j^* \Lambda_\Xi(x)^n e_s]
   = \sum_{n=0}^\infty  \, \sum_{|\alpha|=n}  [\sum_{j,s=1}^g (\Xi^\alpha)_{j,s} B_s ]  x_j \alpha
  = \sum_{n=0}^\infty \sum_{j=1}^g B_j x_j \sum_{|\alpha|=n}  R^{\alpha}  \alpha \\
& = \sum_{j=1}^g B_j x_j \sum_{n=0}^\infty \Lambda_R(x)^n 
  = \Lambda_B(x) (I-\Lambda_R(x))^{-1}.
\end{split}
\]
In particular,
\[
\begin{split}
 \big(I+\Lambda_B(p(x))\big)Q(x) &=
 \big(I+\Lambda_B(p(x))\big)(I-\Lambda_R(x)) \\
 &  = I -\Lambda_R(x) +\Lambda_B(x)  = I+\Lambda_A(x),
\end{split}
\]
since $R=B-A$. This computation also shows if both $\|\Lambda_\Xi(Z)\|<1$
and $\|\Lambda_R(Z)\|<1$,  then equation \eqref{e:Nstatzpre} holds.
Since both sides of equation \eqref{e:Nstatzpre} are rational functions, equation
\eqref{e:Nstatzpre} holds whenever $Z\in \dom(p)$.  Finally, using
 $\Lambda_B(p(x))Q(x)=\Lambda_B(x)$ as well as $R=B-A$ and $B=UA,$
\[
\begin{split}
  Q(Z)^* \LRE_B(p(Z)) Q(Z) &=  Q^*(Z)Q(Z)+ Q(Z)^*\Lambda_B(Z) +\Lambda_B(X)^*Q(Z) \\
  & = I+\Lambda_A(Z) +\Lambda_A(Z) +\Lambda_B(Z)^*\Lambda_B(Z)-\Lambda_A(Z)^*\Lambda_A(Z) = \LRE_A(Z).
\end{split}
\]

 a routine calculation
shows that equation \eqref{e:Nstatzpre} implies equation \eqref{e:Nstatz}.

Since $B\in M_{\rt}(\C)^g$ is linearly independent, for each $1\le k\le g$
 there exists a linear functional $\lambda_k:M_{\rt}(\C)\to \C$ such that
 $\lambda_k(B_k)=1$ and $\lambda_k(B_j)=0$ if $j\ne k$. 
For each $k$, there is a matrix $\Psi_k\in M_{\rt}(\C)$ such that $\lambda_k(T)=\trace(T\Psi_k).$
Writing $\Psi_k = \sum_s  v_{k,s} u_{k,s}^*$ for vectors $u_{k,s},v_{k,s}\in \C^t$, 
\[
\lambda_k(T) = \sum_s u_{k,s}^* Tv_{k,s}.
\]
Let 
\[
 r^k(x)=\sum_{\ell,s} (u_{k,s}^*+ u_{k,s}^*A_\ell x_\ell) (I-\Lambda_R(x))^{-1} v_{k,s} -\lambda_k(I).
\]
Hence, for $X\in M_n(\C)^g$ sufficiently close to $0$, and with $W=Q^{-1}$
 and $\Phi_k=\lambda_k\otimes I_n,$ 
\[
\begin{split}
	p^k(X)& =  \Phi_k \left (\Lambda_B(p(X)) )
     = \Phi_k\left([I_{\rt}\otimes I_n +\Lambda_A(X)\right ]W(X)-I_{\rt}\otimes I_n \right) \\
	& = \sum_{\ell,s} [ u_{k,s}^*\otimes I + (u_{k,s}^*A_j \otimes I_n)(I_{\rt}\otimes X_j)] 
          (I_{\rt}\otimes I_n-\Lambda_R(X))^{-1}[v_{k,s}\otimes I_n] - \lambda_k(I)\otimes I_n\\
  & = r^k(X).
\end{split}
\]
Thus, in the notation  of equation \eqref{e:msI}, $\sI_R\subset \dom(p)$;
 that is, if $Q(Z)=I-\Lambda_R(Z)$ is invertible, 
    then $Z\in \dom(p),$ proving
 item \eqref{i:Nstatz4}.
\end{proof}

\begin{proof}[Proof of Theorem \ref{t:1point1}]
That $\Xi$ is convexotonic follows from Lemma \ref{l:secret}.
Let $p$ denote the resulting convexotonic map. 
 Let $R=\Balt-\Aalt=(U-I)\Aalt$ and $Q(x)=I-\Lambda_R(x).$ 
From Proposition \ref{prop:Nstatz},
\begin{equation}
\label{e:QLBQ}
 Q(X)^* \LRE_\Balt(p(X)) Q(X) = \LRE_\Aalt(X),
\end{equation}
holds whenever $Q(X)$ is invertible.

 Let $X\in \interior(\cD_\Aalt(n))$ be given.  
 The function 
 $F_X(z) = \Lambda_\Balt(p((1-z)X))$  
 is a $M_d(\C)\otimes M_n(\C)$-valued rational function (of the single complex variable $z$
 that is regular at $z=1$).  Suppose $\lim_{z\to 0} F_X(z)$ exists and let $T$ denote
 the limit. In that case,
\[
\begin{split}
 Q(X)^* (I+T+T^*) Q(X) & = \lim_{z\to 0} Q((1-z)X)^* (I+F_X(z)+F_X(z)^*) Q((1-z)X) \\
  & = \LRE_\Aalt(X)\succ 0
  \end{split}
\]
and therefore $Q(X)$ is invertible (and $I+T+T^*\succ 0$). 
 Hence, if $\lim_{z\to 0} F_X(z)$ exists, then $Q(X)$ is invertible.

We now show the limit  $\lim_{z\to 0} F_X(z)$ must exist, arguing by contradiction. Accordingly,
suppose this limit fails to exist. Equivalently, $F_X(z)$ has a pole at $0$.
In this case there exists a $M_d(\C)\otimes M_n(\C)$ matrix-valued function $\Psi(z)$ 
analytic and never $0$ in a neighborhood of $0$ and a positive 
integer $m$ such that $F_X(z) = z^{-m} \Psi(z)$.   Since $\Psi(0)\ne 0$,
there is a vector $\gamma$ such that $\langle \Psi(0)\gamma,\gamma\rangle \ne 0$
(since the scalar field is $\C$).  Choose a real number $\theta$
such that $\kappa:=e^{-im \theta} \langle \Psi(0)\gamma,\gamma\rangle <0$. Hence,
 for $t$ real and positive,
\[
\begin{split}
 \langle &(F_X(te^{i \theta}) +  F_X(te^{i\theta})^*)\gamma,\gamma \rangle \\
   & = t^{-m} \langle [e^{-im \theta} \Psi(te^{i\theta})
     +e^{im\theta}\Psi(te^{i\theta})^*]\gamma,\gamma\rangle
\\
& = t^{-m} \left [ 2\langle e^{-im\theta} \Psi(0)\gamma,\gamma\rangle + 
 \langle [ e^{-im\theta} [\Psi(te^{-i\theta})-\Psi(0)]\gamma,\gamma\rangle
    + e^{im\theta} \langle [\Psi(te^{-i\theta})^* -\Psi(0)^*]\gamma,\gamma\rangle  \right ]\\
& \le 2t^{-m} [\kappa + \delta_t],
\end{split}
\]
where $\delta_t$ tends to $0$ as $t$ tends to $0$. Hence, for $0<t$ sufficiently small,
\[
 \langle \LRE_\Balt(p((1-te^{-im\theta})X))\gamma,\gamma\rangle
 =  \langle (I+F_X(te^{i\theta}) + F_X(te^{i\theta})^*)\gamma,\gamma \rangle <0,
\]
contradicting the fact that 
 $(1-t e^{-im\theta}) X\in \interior(\cD_\Aalt)\cap \dom(p)$
for all $0<t$ sufficiently small.  At this point
we have shown if $X\in \interior(\cD_\Aalt)$, then $Q(X)$ is invertible
 and therefore, by Proposition~\ref{prop:Nstatz}, $X\in \dom(p).$ Further, 
 if $X\in \interior(\cD_\Aalt)$, then, by equation \eqref{e:QLBQ},
\[
 Q(X)^* \LRE_\Balt (p(X)) Q(X) = \LRE_\Aalt(X) \succ 0
\]
and thus $\LRE_\Balt(p(X))\succ 0$; that is $p(X)\in \interior(\cD_\Balt)$, 
 By symmetry,  the same is true for $q.$ Consequently, $p:\interior(\cD_\Aalt)\to\interior(\cD_\Balt)$
 is bianalytic with inverse $q:\interior(\cD_{\Balt})\to\interior(\cD_{\Aalt}),$
 proving item \eqref{i:11a}.

 If $X\in \exterior(\cD_{\Aalt})\cap \dom(p),$ then $\LRE_\Balt(p(X)) \not\succeq 0$
 by Proposition~\ref{prop:Nstatz}\eqref{i:Nstatz3} and equation \eqref{e:Nstatz},
 proving item \eqref{i:11d}.

Now suppose $\cD_\Balt(1)$ is  bounded and $Z\in \partial \cD_\Aalt(n)$.
By \cite[Proposition 2.4]{HKM}, $\cD_\Balt(n)$ is also bounded. 
For $0<t<1$, we have  $tZ\in\dom(p)$ (by item \eqref{i:11a})
 and hence $\varphi$, defined on $(0,1)$ by $\varphi_Z(t):=p(tZ)$,
 maps into $\interior(\cD_\Balt(n))$ and is thus bounded. It follows
 that $G_Z(t) = \Lambda_\Balt(\varphi_Z(t))$ is also a bounded function
 on $(0,1)$.  Arguing by contradiction, suppose $Q(Z)=I-\Lambda_R(Z)$ is not invertible.
Thus there is a unit vector $\gamma$ 
  such that $Q(zZ)\gamma  = (1-z)\gamma.$
 For $0<t<1$, equation \eqref{e:QLBQ} gives,
\[
 (1-t)^2 \langle  \LRE_\Balt(\varphi_Z(t))\gamma,\gamma\rangle
  = 1 -  t [- \langle [\Lambda_\Aalt(Z)+\Lambda_\Aalt(Z)^*]\gamma,\gamma\rangle].
\]
 Since the left hand side converges to $0$ as $t$ approaches $1$ from below, 
 the right hand equals $1-t$. Hence 
\[
(1-t) \langle  \LRE_\Balt(\varphi_Z(t))\gamma,\gamma\rangle = 1,
\]
 and we have arrived at a contradiction, as the left hand side converges to $0$ 
 as $t$ tends to $1$ from below.  Hence $Q(Z)$ is invertible.
 By Proposition \ref{prop:Nstatz}\eqref{i:Nstatz4}, if $\cD_{\Balt}$ is bounded,
then $\cD_{\Aalt}\subset \dom(p)$, proving item \eqref{i:11b}.

Suppose  $X\in \dom(p)\cap \partial \cD_\Aalt.$ Since $\dom(p)$ is open,
 $tX\in \dom(p)$
 for $t\in\mathbb R$ sufficiently close to $1.$ Further   $p(tX)\in \interior(\cD_A)$
for $t<1$ and $p(tX)\in \exterior(\cD_{\Balt})$ for $t>1.$
By continuity,   $p(X)\in \partial \cD_{\Balt},$  proving item \eqref{i:11e}.
\end{proof}

\section{Minimality and \irrLy}\label{sec:prelim}
A  \mop $L_A=L_A(x,y)$ of size $e$ is \df{\irrL} 
if its coefficients $\{A_1,\dots,A_g,A_1^*,\dots,A_g^*\}$ generate
$\mat{e}$ as a $\C$-algebra.\footnote{Previously, in \cite{KV}
such pencils were called irreducible.}  
A collection of   sets $\{S_1,\dots,S_k\}$ is \df{irredundant} if
$\bigcap_{j\ne \ell} S_j \not\subseteq S_\ell$  for all $\ell$.
A collection $\{L_{A^1},\dots,L_{A^k}\}$ of \mops
is \df{irredundant} if $\{\cD_{A^j}:1\le j\le k\}$ is irredundant.

\begin{lemma}
\label{l:Z+}
Given $B\in M_r(\C)^g$, there exists a reducing subspace $\sM$ for 
$\{B_1,\dots,B_g\}$ 
such that, with $A=B|_{\sM}$,
the \mop $L_A$ is minimal for $\cD_B=\cD_A.$

If $L_A$ and $L_B$ are both minimal and $\cD_A=\cD_B$, then
$A$ and $B$ are unitarily equivalent. In particular $A$ and $B$
have the same size.

Given a \mop $L_A(x,y)= I+\sum A_j x_j  +\sum A_j^* y$, 
there is a $k$ and  \irrL  \mops $L_{A^j}$ such that 
\[
 L_A=\bigoplus_{j=1}^k L_{A^j} = L_{\bigoplus_{j=1}^k A^j},
\]
where the direct sum is in the sense of an orthogonal direct
sum decomposition of the space that $A$ acts upon.
Moreover, $L_A$ is minimal if and only if $\{L_{A^j}:1\le j\le \ell\}$ is irredundant.
\end{lemma}

\begin{proof}
Zalar \cite{Zal17} (see also \cite{HKM}) establishes this result
over the reals, but the proofs work (and are easier) over $\C$;
it can also be deduced from the results in \cite{KV} 
and \cite{HKV}.
\end{proof}

Note if $E$ is \bmin then $\ker(E)=\{0\}$ and $\ker(E^*)=\{0\},$  an observation that will be used 
repeatedly in the sequel.

\begin{lemma}\label{l:trans}
Let $E$ be a $g$-tuple of $d\times e$ matrices and assume $\ker(E^*)=\{0\}$ and $\ker(E)=\{0\}.$
\begin{enumerate}[{\rm (1)}]
	\item \label{it:lt1} We have
	\begin{equation}\label{e:sa}
	\begin{pmatrix}I & 0 \\ \Lambda_{E^*} & I\end{pmatrix}
	\begin{pmatrix}I & 0 \\ 0 & Q_E\end{pmatrix}
	\begin{pmatrix}I & \Lambda_E \\ 0 & I\end{pmatrix}
	=\M_E. 
	\end{equation}
   \item \label{it:lt2} The \mop $\M_{E}$ is \irrL if 
    and only if $Q_E$ is an atom. %
	\item \label{it:lt3}  $E$ is \bmin
     if and only if $\Mre_{E}$ is minimal.
\item \label{it:lt3+} If $A\in M_N(\C)^g$ and $A_mA_j=0$ for all $1\le j,m\le g$
  then,  $\dim\rg A + \dim\rg A^* \le N$ and for any $s\ge \dim\rg A$
and $t\ge \dim \rg A^*$ with $s+t= N$, 
 there exists a tuple  $F \in M_{s\times t}(\C)^g$ such that $A$ is unitarily equivalent to
\[
  \begin{pmatrix} 0 & F \\ 0 & 0 \end{pmatrix}.
\]
 	\item \label{it:lt4} If $L_A$ is minimal and $\cD_A$ is a spectraball,
  then there exist \bmin tuples $F^1,\dots,F^k$ such that each $\M_{F^j}$ is an \irrL \mop,
  $\{\cB_{F^1},\dots,\cB_{F^k}\}$ is irredundant  and 
  $L_A$ is unitarily equivalent to $\M_{F^1}\oplus\cdots\oplus \M_{F^k}.$
\item \label{it:lt4-} If $A$ is \bmin, then $L_A$ is minimal. 
  \item  \label{it:lt8}
   If $E$ is \bmin, then, up to unitary equivalence,
   $Q_E=Q_{E^1}\oplus\cdots\oplus Q_{E^k}$, where the  $Q_{E^j}\in \pxy^{e_j\times e_j}$ 
are atoms, $\ker (E^j)=\{0\}$ for all $j$, and the spectraballs $\cB_{E^j}$ are 
irredundant.
 \item \label{it:lt5} If $Q_E$ is an atom, %
           then $E$ is \bmin.
 \item \label{it:lt7} 
 If  $E$ \bmin, $F\in M_{k\times \ell}(\C)^g$ 
and $\cB_E=\cB_F$, then there
 is a tuple $R\in M_{(k-d)\times (\ell-e)}(\C)^g$
   and unitaries $U,V$  of sizes $k\times k$ and $\ell\times \ell$
 respectively such that
 $\cB_E\subseteq \cB_R$ and 
\begin{equation}
\label{e:FUERV}
 F = U\begin{pmatrix} E & 0 \\ 0 & R \end{pmatrix} V.
\end{equation}
In particular,
 \begin{enumerate}[\rm (a)]
  \item $d\le k$ and $e\le \ell$;
 \item  if $F\in M_{d\times e}(\C)^g$ is \bmin too, then $E$ and $F$ are ball-equivalent.
  \end{enumerate}
\end{enumerate}
\end{lemma}

Item \eqref{it:lt7} can be interpreted in terms of 
completely contractive maps and as special cases of the rectangular operator spaces of \cite{FHL18}.
Indeed, letting $\sE$ and $\sF$ denote the spans of $\{E_1,\dots,E_g\}$
and $\{F_1,\dots,F_g\}$ respectively, the inclusion $\cB_E \subset \cB_F$ is equivalent to the 
mapping $\Phi:\sE\to\sF$ defined by $\Phi(E_j)=F_j$ being completely contractive. Hence
 $\cB_E=\cB_F$ if and only if $\Phi$ is completely isometric.

\begin{proof}
\eqref{it:lt1} Straightforward. 

\eqref{it:lt2} By \eqref{e:sa}, $Q_E$ and $\M_{E}$ are stably associated, cf. 
\cite[Section 4]{HKV}. Hence $\M_{E}$ does not factor in $\pxy^{(d+e)\times (d+e)}$ 
if and only if $Q_E$ does not factor in $\pxy^{e\times e}$ by 
 \cite[Section 4]{HKV}.
 Next, $\M_{E}$  is \irrL 
if and only if it does not factor  and 
\[
 \ker(\begin{pmatrix} 0&E\\0&0 \end{pmatrix}) \cap \ker(\begin{pmatrix} 0& 0\\ E^* & 0 \end{pmatrix}) = \{0\}
\]
(\cite[Section 2.1 and Theorem 3.4]{HKV}).
 Thus $\M_{E}$ is \irrL if and only if $Q_E$ does not factor. 

\eqref{it:lt3} Let $L_B$ be minimal for $\cD_B=\cB_E$ and let $N$ denote the size of $B$. 
  By  \cite[Theorem 1.1(2)]{EHKM17} there exists positive integers $s,t$
  such that $s+t=N$ and a tuple $F\in M_{s\times t}(\C)^g$ such that 
\[
 B = \begin{pmatrix} 0&F\\0&0\end{pmatrix}.
\]
Thus $\cB_E=\cB_F$. On the other hand, with
\[
 A= \begin{pmatrix} 0&E\\0&0\end{pmatrix},
\]
$\cD_A=\cB_E$ too. By minimality of $B$, $s+t\le d+e$. 
If $E$ is \bmin, then, since  $\cB_E=\cB_F,$ we have
$s+t\ge d+e$ and hence $\Lre_A=\Mre_E$ is minimal. On the other hand,
if $\Mre_E$ is minimal, then $\Mre_E$ and $L_B$ have
the same size, $N=s+t = d+e$  and thus $E$ is \bmin.

\eqref{it:lt3+} Let $\mathscr{R} = \rg A$ and $\mathscr{R_*}=\rg A^*$.  
Since $A_mA_j=0$ it follows that $\mathscr{R}$ and $\mathscr{R_*}$ are orthogonal
and also that $A_m \mathscr{R}=0$ and $A_m^* \mathscr{R_*}=0$
for
$1\le m\le g$.  In particular, $\dim\mathscr{R}+\dim\mathscr{R_*} \le N.$
	Letting $V$ and $V_*$ denote
the inclusions of $\mathscr{R}$ and $\mathscr{R}_*$ into $\C^N$ respectively,
\begin{equation}\label{e:cor}
 A= \begin{pmatrix} 0 & 0 & V^* A V_*\\ 0 & 0 &0\\ 0&0&0\end{pmatrix},
\end{equation}
with respect to the decomposition 
$\C^N = (\mathscr{R}\oplus\mathscr{R_*})^\perp \oplus \mathscr{R}\oplus \mathscr{R_*}$. Now any choice of $s\ge \dim \mathscr{R}$ and $t\ge \dim \mathscr{R_*}$ with $s+t=N$ applied to \eqref{e:cor} gives the desired decomposition.

\eqref{it:lt4}
  Since  $L_A$ is minimal, 
by Lemma \ref{l:Z+},  $L_A$ is 
unitarily equivalent to $L_{A^1}\oplus\cdots\oplus L_{A^k}$ for some \irrL  
irredundant \mops
$L_{A^1},\dots,L_{A^k}.$  Let $N_j$ denote the size
of $A^j$.   Now suppose $\cD_A$ is a spectraball.
Thus, there exists $m,\ell$ and a \bmin tuple $G\in M_{m\times\ell}(\C)^g$ such that 
$\cD_A = \cB_G$. By item \eqref{it:lt3}  $\Mre_G$ is minimal for $\cD_A$.
Thus
\[
B:=\begin{pmatrix} 0 & G \\0 & 0 \end{pmatrix}\in M_{m+\ell}(\C)^g
\]
is unitarily equivalent to $A^1\oplus\cdots\oplus A^k$ by Lemma \ref{l:Z+}. 
Since $B_mB_j=0$ for $1\le j,m\le g$, it follows that 
$A^\ell_m A^\ell_j =0$ for all $j,m,\ell$. By item \eqref{it:lt3+},
there exists $s_j,t_j$ such that $s_j+t_j=N_j$ and tuples 
$F^j \in M_{s_j\times t_j}(\C)^g$ such that, up to unitary equivalence,
\[
 A^j = \begin{pmatrix} 0& F^j \\ 0 & 0 \end{pmatrix}  \in M_{N_j}(\C)^g.
\]
Moreover, since $L_A$ is minimal and $\cD_A=\cap_{j=1}^k \cB_{F^j}$,  each $F^j$ is \bmin.

\eqref{it:lt4-} Given a tuple $A\in M_d(\C)^g$, observe that $X\in \cB_A$ if and only if
 $S\otimes X\in \cD_A$, where 
\[
 S=\begin{pmatrix}0 &1\\0 & 0\end{pmatrix}.
\]
Thus, if $B\in M_r(\C)^d$ and $\cD_B=\cD_A$, then $\cB_B=\cB_A$ and by \bminy, 
$r\ge d$. Hence $L_A$ is minimal.

\eqref{it:lt8} Combine items \eqref{it:lt3}, \eqref{it:lt4} and \eqref{it:lt2}
in that order.

\eqref{it:lt5}
  By item \eqref{it:lt2},  $\M_E$ is \irrL. %
  For a pencil $L$, \irrLy of $L$ implies minimality of $\Lre$  by Lemma \ref{l:Z+}. 
   Thus $\Mre_E$ is minimal
   and hence $E$ is \bmin by item \eqref{it:lt3}.

\eqref{it:lt7}
 Let
\[
 A= \begin{pmatrix} 0&E\\0 &0\end{pmatrix} \in M_{d+e}(\C)^g.
\]
 By item \eqref{it:lt3}, $\Lre_A=\Mre_E$ is minimal.
 Since $\Mre_F$ defines 
 $\cB_E,$ there is a reducing subspace $\mathscr{M}$  for 
\[
B= \begin{pmatrix} 0 & F\\0& 0\end{pmatrix} \in M_{k+\ell}(\C)^g
\]
such that the restriction of $B$ to $\mathscr{M}$ is unitarily equivalent to  $A$
by Lemma \ref{l:Z+}. Thus, 
there is unitary $Z\in M_{k+\ell}(\C)$ and a tuple $C\in M_{(k+\ell)-(d+e)}(\C)^g$
such that,  with
respect to the decomposition $\mathscr{M}\oplus \mathscr{M}^\perp$,
\[
  B= Z^*\begin{pmatrix}  A & 0 \\ 0 & C \end{pmatrix} Z.
\]
Since $B_mB_j=0$ for all $j,m$, we have $C_mC_j=0$ too.  
Further, using \bminy of $E$,  $\ell\ge\rk F^*F =\rk E^*E +\rk C^*C = e+\rk C^*C$.
Thus $\dim \rg C \le \ell-e$. Likewise, $\dim\rg C^* \le k-d$. 
By item \eqref{it:lt3+}, there exists a tuple $R\in M_{(k-d)\times(\ell-e)}(\C)^g$ such that,
up to unitary equivalence, 
\[
 C= \begin{pmatrix} 0 & R  \\ 0 & 0 \end{pmatrix}.
\]
Thus, letting $G=(\begin{smallmatrix} E&0\\0&R\end{smallmatrix})\in M_{k\times\ell}(\C)^g$,
\[
\begin{pmatrix} 0 & F\\ 0 & 0  \end{pmatrix} X
  = X \begin{pmatrix} 0 & G \\ 0&0 \end{pmatrix}
\]
for some unitary matrix  $X$.  Writing  $X=(X_{j,k})_{j,k=1}^2$ with respect to the
decomposition $\C^k\oplus \C^\ell$, it follows that
\[
 X_{11} G = FX_{22}, \ \ X_{21}G=0, \ \ FX_{21}=0.
\]
Hence $F X_{22}X_{22}^* = F$ and $X_{11}^*X_{11}G= G$. Thus $X_{11}$ is isometric
on $\rg G$ and therefore $X_{11}$ extends to a  unitary mapping $U$ on all of $\C^k$
such that $UG=X_{11}G$. Similarly, $X_{22}^*$ is isometric on $\rg F^*$
and hence $X_{22}^*$ extends to a unitary $V$ on all of $\C^\ell$ such that
$V F^*= X_{21}^* F^*.$  Finally, $UG=X_{11}G=FX_{22}=FV^*$. Hence
equation \eqref{e:FUERV} holds, which implies $\cB_E=\cB_F= \cB_E\cap \cB_R.$
Thus $\cB_E\subset \cB_R$ and the remainder of item \eqref{it:lt7} follows.
\end{proof}

Minimality and \irrLy  of \mops 
are preserved under an affine linear change of variables.

\begin{prop}\label{prop:aff}
Consider a \hmm  pencil $\LRE_A$  and an affine linear change of variables
$\lambda:x\mapsto xM+b$ for some invertible $g\times g$ matrix $M$
and vector $b\in\C^g$. 
If $\LRE_A(b)\succ0,$  then $\lambda^{-1}(\cD_A)= \cD_{F}$, where
\beq\label{eq:aff}
F=M \cdotb (\sH A \sH) \quad\text{ and }\quad \sH=\LRE_A(b)^{-1/2}.
\eeq
Further,
\begin{enumerate}[\rm (1)]
\item
\label{it:affineirr}
$L_A$ is \irrL if and only if $L_F$ \irrL;
\item
\label{it:affinemin}
$L_A$ is minimal if and only if $L_F$ is minimal.
\end{enumerate}
\end{prop}

\begin{proof}
Equation \eqref{eq:aff} is proved in \cite[\S 8.2]{AHKM18}.

Turning to item \eqref{it:affineirr}, let us first settle the special case $M=I$. 
If $L_A$ is not \irrL,  then there is a common
 non-trivial reducing subspace $\mathscr{M}$ for $A$.
 It follows that $\mathscr{M}$ is reducing for 
 $\LRE_A(b)$ and hence for $F=\sH A \sH$.

 Now suppose $L_F$ is not \irrL; that is,
 there is a non-trivial reducing subspace
 $\mathscr{N}$ for $F=\sH A\sH$.  Since 
\[
\sH (\LRE_A(b)-I)\sH = \sH (\Lambda_A(b)+\Lambda_A(b)^*)\sH 
  = \Lambda_{F}(b)+\Lambda_F(b)^*,
\]
 we conclude that
\[
 \big(I-\LRE_A(b)^{-1}\big)\mathscr{N}= \sH(\LRE_A(b) -I)\sH \mathscr{N}\subseteq \mathscr{N}.
\]
 Hence $\mathscr{N}$ is invariant for $\LRE_A(b)^{-1}$. Since 
 $\mathscr{N}$ is finite dimensional and $\LRE_A(b)^{-1}$ is invertible,
 $\LRE_A(b)^{-1}\mathscr{N} = \mathscr{N}$ and 
 consequently $\sH\mathscr{N}=\mathscr{N}$.  
 Because $F=\sH A \sH$ it is now evident that $\mathscr{N}$ 
 is reducing for $A$.

 Now consider the special case $b=0$. A subspace $\mathscr{M}$
 reduces $A$ if and only if it reduces $M\cdotb A$.  Combining
 these two special cases proves item \eqref{it:affineirr}.

Finally we prove item \eqref{it:affinemin}. 
By Lemma \ref{l:Z+}, $L_A$
is unitarily equivalent to 
$\bigoplus_{j=1}^\ell L_{A^j}$, where the  $L_{A^j}$ are  \irrL 
\mops.
Now 
$L_F$ is unitarily equivalent to 
$\bigoplus_{j=1}^\ell L_{F^j}$, where $F^j=M\cdotb (\sH A^j \sH)$. By item \eqref{it:affineirr}, each of these summands $L_{F^j}$ is \irrL. 
Furthermore, since $\Psi$ is bijective it is clear that  $\bigcap_{k\ne i} \cD_{A^k} \subset \cD_{A^i}$ if and only if  $\bigcap_{k\ne j} \cD_{F^k}\subseteq \cD_{F^j}.$ 
Therefore $\{L_{A^j}:1\le j\le \ell\}$ is irredundant if and only if $\{L_{F^j}: 1\le j\le \ell\}$ is irredundant. 
Hence $L_A$ is minimal for $\cD_A$ if and only if $L_F$ is minimal for $\cD_F$, again
by Lemma \ref{l:Z+}.
\end{proof}

\begin{example}\rm
Even with $M=I$, the property \eqref{it:affineirr} of Proposition \ref{prop:aff} fails for a general positive definite $\sH$ and $F$ as in \eqref{eq:aff}.
For example, let
\[
A=\begin{pmatrix}
2 & 4 & 2 & 0 \\
1 & 2 & 2 & 2 \\
0 & 0 & 2 & 4 \\
0 & 0 & 1 & 2
\end{pmatrix}, \qquad \sH=\begin{pmatrix}
2 & 1 & 0 & 0 \\
1 & 2 & 1 & 0 \\
0 & 1 & 2 & 1 \\
0 & 0 & 1 & 2
 \end{pmatrix}^{-1}.
\]
Then $L_A$ is \irrL, but since
\[
F=\begin{pmatrix}
0 & 1 & 0 & 0 \\
0 & 0 & 0 & 0 \\
0 & 0 & 0 & 1 \\
0 & 0 & 0 & 0
\end{pmatrix},
\]
the \mop $L_F$ is clearly not.
\qed
\end{example}

\begin{remark}\rm
  Suppose  $E\in M_{d\times e}(\C)^g$ and $C\in M_g(\C)$
  is invertible. If  $E$ is \bmin, then ${C \cdotb E}$
 (see equation \eqref{e:MdotB}) is \bmin.
\qed \end{remark}

\section{Characterizing bianalytic maps between spectrahedra}
\label{sec:bianalhedra}
In this section we prove Theorem \ref{thm:sdmain} and
Proposition \ref{t:hairspans}, stated as 
 Propositions \ref{p:hairspans} and \ref{p:hairhyperspans} below.
A major accomplishment,  exposited in Subsection \ref{sssec:eig}, is the 
reduction of the eig-generic type hypotheses of \cite{AHKM18}
to various natural and cleaner algebraic
conditions on the corresponding pencils defining spectrahedra.

\begin{lemma}\label{lem:Zdense}
Let $L_A$ be a \mop. The set $
\{(X,X^*)\midx X\in \Zre_{L_A}(n)\}$ 
	is Zariski dense in the set
	$\cZ_{L_A}(n)$
	 for every $n$. Likewise, $\{(X,X^*)\midx X\in \cZ_{Q_A}^{\rm re} (n)\}$
   is Zariski dense in  $\cZ_{Q_A} (n)=\{(X,Y)\in M_n(\C)^{2g}: \det Q_A(X,Y)=0\}$.
\end{lemma}

\begin{proof}
The first statement holds by \cite[Proposition 5.2]{KV}.
 The second follows immediately from the first.
\end{proof}

\subsection{The detailed boundary}
Let $\rho$ be a \df{hermitian $d \times d$  free matrix  polynomial} with $\rho(0) =I_d$. 
Thus $\rho \in \pxy^{d\times d}$ and $\rho(X,X^*)^*=\rho(X,X^*)$ for all $X\in M(\C)^g$. 
The \df{detailed boundary} of $\cD_{\rho}$ is the sequence of sets \index{$\widehat{\partial \cD_{\rho}}$}
$$\widehat{\partial \cD_{\rho} }(n) := 
\left\{
(X,v)\in\mat{n}^g\times (\C^{dn}\setminus\{0\}) \midx X \in \partial \cD_\rho , \  \rho(X, X^*) v = 0 
\right\}$$
over $n\in\N$. The nomenclature and notation are somewhat misleading  
 in that   $\widehat{\partial D_{\rho} }$
is not determined by the set $ \cD_{\rho} $ but
by its defining polynomial $\rho$. Denote also
\[
\widehat{ \partial^1 \cD_{\rho}}(n) := \left\{
(X,v) \in \widehat{ \partial \cD_{\rho} }(n) \midx \dim\ker(\rho(X,X^*))=1
\right\}.
\]
For $ (X,v) \in \widehat{ \partial^1 \cD_{\rho}(n) }$, 
 we call $v$ the \index{$\widehat{ \partial^1 \cD_{\rho}(n)}$}
\df{hair} at $X$.
Letting 
\[
\pi_1:\mat{n}^g\times \C^{dn}\to\mat{n}^g \quad\text{ and }\quad \pi_2:\mat{n}^g\times \C^{dn}\to\C^{dn}
\]
 denote the  canonical projections, set 
$$ \partial^1 \cD_{\rho}(n)= \pi_1\left(\widehat{ \partial^1 \cD_{\rho}}(n)\right),
\qquad \hair \cD_{\rho}(n)= \pi_2\left(\widehat{ \partial^1 \cD_{\rho}}(n)\right).$$  
Observe  $\widehat{\partial \cB_E }(n):=  \widehat{\partial \cD_{Q_E} }(n)$, etc.

\subsubsection{Boundary hair spans}
In this subsection we connect the notion of boundary hair to \bmin{}ity. 
Given a tuple $E\in M_{d\times e}(\C)^g,$ a subset $\scS\subset \widehat{ \partial^1 \cB_E}$  is
closed under unitary similarity if for each $n$, each $(X,v)\in \widehat{ \partial^1 \cB_E}(n)$
and each $n\times n$ unitary $U$, we have $(UXU^*,(I_e\otimes U)v)\in \scS(n)$. Assuming
$\scS \subset \widehat{ \partial^1 \cB_E}$ is closed under unitary similarity, let
\[
 \pi(\hair \scS) = 
  \Big\{u\in \C^e: \exists n\in\N,\, \exists v\in \scS(n)\cap \hair \cB_E(n): \,
    v = u\otimes e_1 +\sum_{j=2}^n u_j \otimes e_j \Big\},
\]
where $\{\ee_1,\dots,\ee_n\}$ is the standard basis for $\C^n$.
Because $\scS$ is invariant under unitary similarity, 
the definition of $\pi(\hair \scS)$ does not actually depend on the choice of
 orthonormal basis
for $\C^n$.
Thus, for instance, $\pi(\hair{ \partial^1 \cB_E})$ is the set of those vectors
$u\in \C^e$ such that there exists an $n$, a pair $(X,v)\in M_n(\C)^g \oplus [\C^e \otimes \C^n]$
and a unit vector $h\in \C^n$ such that $\Qre_E(X)\succeq 0$, $\dim\ker(\Qre_E(X))=1$,  $\Qre_E(X)v=0$ and
$u= (I_e\otimes h^*)v.$  For notational convenience we write $\pi(\hair\cB_E)$ as shorthand
 for $\pi(\hair \widehat{\partial^1\cB_E}).$  \index{$\pi(\hair\cB_E)$}
 
\begin{prop}
\label{p:hairspans}
A tuple  $E\in M_{d\times e}(\C)^g$ is \bmin  if and only if $\pi(\hair\cB_E)$ 
 spans $\C^e$ and $\ker(E^*)=\{0\}.$  Moreover, if $\pi(\hair\cB_E)$ spans $\C^e,$
 then there exists a positive integer $r$\footnote{While it is not needed
 here, $r$ can be chosen at most $e.$} and pairs 
  $(\alpha^a,\gamma^a) \in \widehat{\partial^1 \cB_E(r)}$ for 
  $1\le a\le e$ such that, writing 
  $\gamma^a = \sum_{t=1}^r \delta^a_t\otimes e_t\in \C^e\otimes \C^r$
  the set $\{\delta^a_1: 1\le a\le e\}$ spans $\C^e.$
\end{prop}

\begin{proof}
 Suppose $E$ is \bmin and let $e^\prime \le e$ denote the dimension of
 the span of  $\pi(\hair\cB_E).$ 
Let 
$$\cT_E=\{(X,X^*)\colon X\in \partial^1 \cD_{Q_E}=\partial^1 \cB_E\}.$$
 Let $W$ denote the inclusion of $\spann\pi(\hair\cB_E)$ into $\C^e$. 
Observe that
\[
W^*Q_E(x, y)W =W^*W-W^*\Lambda_{E^*}(y)\Lambda_E(x)W=Q_{EW}(x, y).
\]
Thus $\cB_E\subset\cB_{EW}$ and moreover
$(X,v) \in\widehat{\partial^1 \cD_{Q_E} =\partial^1\cB_E}$
 implies 
$$\Qre_{EW}(X) (W^*\otimes I)v=(W^*\otimes I)\Qre_E(X) v=0,
$$
so $\cT_E \subseteq \cZ_{\M_{EW}}$. Since $\partial^1 \cD_{\M_E}=\partial^1 \cD_{Q_E}=\partial^1\cB_E$ 
by equation \eqref{e:sa}, 
$\Mre_E$ (equivalently $\M_E$) is minimal by Lemma \ref{l:trans}\eqref{it:lt3}, and
$\cT_E$ is Zariski dense in $\cZ_{\M_E}$  by \cite[Corollary 8.5]{HKV},
 it follows that   $\cZ_{\M_E} \subseteq  \cZ_{\M_{EW}}$. 
 Since are convex sets containing $0$ in their interiors, and their boundaries are
 contained in  $\cZ_{\M_E}$ and $\cZ_{\M_{EW}}$ respectively, 
 the inclusion $\cZ_{\M_E}\subset \cZ_{\M_{EW}}$ implies $\cB_{EW}\subseteq\cB_E$. 
 Indeed, if $X\in \cB_{EW}$ but $X\notin\cB_E,$ then there 
 is a $0<t<1$ such that $tX\in \partial \cB_E \cap \cB_{EW}.$
 Thus $(tX,tX^*)\in \cZ_{\M_E}\subset \cZ_{\M_{EW}}.$  Consequently
 $\Qre_{EW}(tX)$ has a kernel and finally $\Qre_E(X)\not\succeq 0,$
 contradicting $X\in \cB_{EW}.$
 Hence $E$ and $EW$ define the same spectraball. Since  $EW$ is a $d\times e'$-tuple 
 and $E$ is \bmin and $d\times e$, Lemma \ref{l:trans}\eqref{it:lt7}
 implies $e'\ge e.$ Thus $e^\prime=e$ and $\pi(\hair \cB_E)$ spans 
 $\C^e.$ %
  If $\ker(E^*)\ne \{0\}$ then $E$ is not \bmin.
 Hence we have shown, if $E$ is \bmin, then $\pi(\hair\cB_E)$ spans 
and $\ker(E^*)=\{0\}.$

To prove the converse, suppose  $F\in M_{k\times \ell}(\C)^g$ is not \bmin, but $\ker(F^*)=\{0\}.$  
Let $\shair_F \subset \C^\ell$ denote the span of $\pi(\hair\cB_F).$
It suffices to show
$\shair_F\ne \C^\ell$. 
Let $E\in M_{d\times e}(\C)^g$ be \bmin with $\cB_F=\cB_E$.  By Lemma \ref{l:trans}\eqref{it:lt7},
 $d\le k$ and $e\le \ell$ and, letting $d^\prime = k-d$ and $e^\prime = \ell -e,$  there is a tuple
$R\in M_{d^\prime\times e^\prime}(\C)^g$ and  $k\times k$ and  $\ell\times \ell$ unitary matrices
 $U$ and $V$ respectively so that equation \eqref{e:FUERV} holds
and $\cB_E\subset \cB_R$.
 Note that $e^\prime \ne 0$  since  $\ker(F^*)=\{0\}$ and further
\[
Q_F = V^* \begin{pmatrix} Q_E &0 \\ 0 & Q_R \end{pmatrix} V = V^* (Q_E\oplus Q_R) V.
\]
Without loss of generality, we may assume $V=I$. 

Suppose $X\in \partial^1 \cB_F(n)$ and $0\ne v \in \C^\ell \otimes \C^n$ is in the kernel
of $\Qre_F(X)$.
With respect to the decomposition of
$\C^\ell \otimes \C^n  = [\C^e \otimes \C^n]  \oplus [\C^{e^\prime} \otimes \C^n]$, decompose
$v= u\oplus u^\prime$. It follows that $0=\Qre_F(X) v = \Qre_E(X)u \oplus \Qre_R(X)u^\prime$ and hence
both $\Qre_E(X)u=0$ and $\Qre_R(X) u^\prime =0$.  Therefore, 
$\left(\begin{smallmatrix} 0\\u^\prime \end{smallmatrix}\right)$ is in the kernel
of $\Qre_F(X)$.  On the other hand, $X\in \partial \cB_E(n)$. Hence there is
a $0\ne w\in \C^e\otimes \C^n$  such that $\Qre(X)w=0.$ Thus 
$0\ne \left(\begin{smallmatrix} w\\0 \end{smallmatrix}\right)$ is in the kernel
of $\Qre_F(X).$ Since the dimension of the kernel of $\Qre_F(X)$ is one, $u^\prime =0$
and therefore $\shair_F \subset \C^e \oplus \{0\} \subsetneq \C^e\oplus \C^{e^\prime} =\C^\ell.$

To prove the moreover portion of the proposition, 
 note that the assumption that the $\pi(\hair\cB_E)$ spans
 implies the existence of $n_1,\dots,n_e\in\mathbb N$ and
  pairs $(\alpha^a,\gamma^a)\in M_{n_a}(\C)^g\times [\C^e\otimes \C^{n_a}]$
  such that, writing $\gamma^a = \sum_{t=1}^{n_a} \delta^a_t\otimes e_t$,
 the set $\{\delta^a_1:1\le a\le e\}$ spans $\C^e.$ By choosing
  $r=\max\{n_a:1\le a\le e\}$ and padding $\delta^a$ and $\gamma^a$
  by zeros as needed, it can be assumed that $n_a=r$ for all $a.$
\end{proof}

\subsection{From basis to hyperbasis}

Call an $e+1$-element  subset $\mathcal U= \{u^1,\dots,u^{e+1}\}$ of $\C^e$  a \df{hyperbasis} if each
$e$-element subset of $\mathcal U$ is a basis.  
This notion critically enters the genericity 
conditions considered in \cite{AHKM18}.

\begin{lemma}
\label{l:hyperhair}
Given  $E\in M_{d\times e}(\C)^g$ and $n\in\N,$ if $\cZ_{Q_E}(n)$ is an 
irreducible hypersurface in $M_n(\C)^{2g}$, 
\[
   \{(X,X^*)\colon X\in\partial^1\cB_E(n)\}
\]
is Zariski dense in $\cZ_{Q_E}(n)$, and $\pi(\hair \cB_E)$
spans $\C^e,$ then $\pi(\hair\cB_E)$ contains a hyperbasis for $\C^e$.
\end{lemma}

\begin{proof}
 By Proposition \ref{p:hairspans} there exist a positive integer $r$,
 tuples  $X^1,\dots,X^e\in\partial^1\cB_E(r)$ and 
 vectors $\gamma^j=\sum_{t=1}^r \delta^j_t\otimes e_t\in \ker(\Qre_E(X^j))
  \subseteq \C^e\otimes\C^r,$
 such that $\{\delta^j_1:1\le j\le e\}$ spans $\C^e.$ Note too that
  $\delta^j_1 = (I\otimes \basr_1^*)\delta^j,$ where
 $\{\basr_1,\dots\basr_r\}$ is the standard orthonormal basis for $\C^r.$

 If $X\in\partial^1\cB_E(n)$, then the adjugate matrix,  $\adj(\Qre_E(X)),$ 
 is of rank one and its range is  $\ker (\Qre_E(X))$. 
 Let $M_{(i)}$ denote the $i$-th column of a matrix $M$
 and suppose $\gamma =\sum_{t=1}^r \delta_t\otimes e_t$ spans $\ker(\Qre_E(X)).$
 It follows that  $(I\otimes \basr_1^*) \adj(\Qre_E(X^k))_{(i)} = \mu  \delta_1$
 for some $\mu\in \mathbb C.$ Moreover, 
 for every $k=1,\dots,e$ there exists $1\le i_k\le er$ such that 
 $\ker (\Qre_E(X^k)) = \spann (\adj \Qre_E(X^k))_{(i_k)},$ 
 and hence $(I\otimes e_1^*) \adj(\Qre_E(X^k))_{(i_k)} = \mu_k \delta^k_1$
 for some $\mu_k\ne 0.$ 
 Now consider
\begin{equation}\label{e:adj}
v(t,X,Y):=\sum_{k=1}^e t_k\, (I\otimes \basr_1^*) \adj(Q_E(X,Y))_{(i_k)} \in \C^e
\end{equation}
as a vector of polynomials in indeterminates $t=(t_1,\dots,t_e)$ and entries 
of $(X,Y)$ (i.e., coordinates of $\mat{r}^{2g}$).  Let $\{\base_1,\dots,\base_e\}$
 denote the standard basis for $\C^e$. 
For every $k$ we have
 $v(\base_k,X^k,X^{k*})=(I\otimes \basr_1^*)\adj(\Qre(X^k))_{(i_k)}  = \mu_k \delta^k_1 \ne 0.$
 Since the complements of zero sets are Zariski open and dense in the affine space,
 for each $k$ the set $U_k =\{t\in \C^g: v(t,X^k X^{k*})\ne 0\}\subset \C^g$ 
 is open and dense and thus so is $\bigcap_{k=1}^e U_k$. Hence there exists 
$\lambda\in\C^e$ such that $v(\lambda,X^k,X^{k*})\neq0$ for every $k$. 
 Now define the  map
\[
  u:\cZ_{Q_E}(n)\to \C^e, \qquad u(X,Y):=v(\lambda,X,Y).
\]
Note that $u$ is a polynomial map by \eqref{e:adj} and,
 for $X\in \partial^1\cB_E(r)$ and 
 $0\ne \delta =\sum_{t=1}^r \delta_t\otimes\basr_t\in \ker(\Qre_E(X)),$
\[
 u(X,X^*) =\sum_{s=1}^e \lambda_s (I\otimes \basr_1^*) \adj(\Qre_E(X))_{(i_s)}
      = \sum_{s=1}^e \lambda_s \nu_s \delta^1 = \nu\delta_1,
\] 
 for some $\nu\in\C.$ In particular, if $U(X,X^*)\ne 0,$ then 
 $u(X,X^*)\in \pi(\hair \cB_E).$
\[
 0\ne  u(X^k,X^{k*}) =\nu_k  \delta^1_k,
\]
 for each $k$ and hence  $u(X^1,X^{1*}),\dots,u(X^e,X^{e*})$ form a basis of $\C^e.$ 
 Therefore,
\[
  u(X,Y)=\sum_{k=1}^e r_k(X,Y) u(X^k,X^{k*}) 
\]
for $(X,Y)\in \cZ_{Q_E}(n)$, where $r_k$ are polynomial functions 
on $M_r(\C)^{2g}$.  In particular, $r_k(X^j,X^{j*})=\updelta_{j,k},$
 where $\updelta$ is  the Kronecker  delta function.
 \index{$\updelta$} \index{Kronecker delta function}

Suppose that the product $r_1\cdots r_e\equiv0$ on
$$\{(X,X^*)\colon X\in\partial^1\cB_E(n)\} \subset \cZ_{Q_E}.$$
Then $r_1\cdots r_e\equiv0$ on $\cZ_{Q_E}(n)$ by the  Zariski denseness hypothesis. 
Therefore $r_k\equiv0$ on $\cZ_{Q_E}(n)$ for some $k$ by the irreducibility hypothesis, contradicting
 $r_k(X^k,X^{k*})=1$.
Consequently there exists $X^0\in\partial^1\cB_E(n)$ such that 
$r_1(X^0,X^{0*})\cdots r_e(X^0,X^{0*})\neq0$. By the construction it follows
 that $\{u(X^0,X^{0*}),u(X^1,X^{1*}),\dots,u(X^e,X^{e*})\}
 \subseteq \pi(\hair\cB_E)$ forms a hyperbasis of $\C^e.$
\end{proof}

\begin{prop}\label{p:hairhyperspans}
Let $E\in M_{d\times e}(\C)^g$. Then $Q_E$ is an atom and $\ker (E)=\{0\}$
if and only if $\pi(\hair\cB_E)$ contains a hyperbasis of $\C^e$.
\end{prop}

\begin{proof}
 Let $\iota$ denote the inclusion of $\rg(E)$ into $\C^d$ and let $\wE=\iota^*E.$
 Note that $\cB_E=\cB_{\wE}$ and thus $\pi(\hair \cB_E)=\pi(\hair\cB_{\wE}).$
 Further   $Q_E=Q_{\wE}$ and   $\ker(\wE)=\ker(E)$
 and $\ker(\wE^*)=\{0\}.$  It follow that $Q_E$ is an
 atom if and only if $Q_{\wE}$ is an atom; $\ker(E)=\{0\}$
 if and only if $\ker(\wE)=\{0\};$ and $\pi(\hair\cB_E)$ 
 contains a hyperbasis of $\C^e$ if and only if $\pi(\hair\cB_{\wE})$
 does. Thus, by replacing $E$ with $\wE$ we may assume that
 $\ker(E^*)=\{0\}.$

$(\Rightarrow)$ 
  Suppose $Q_E$ is an atom and $\ker (E)=\{0\}$ and $\ker(E^*)=\{0\}.$
  By  Lemma \ref{l:trans}\eqref{it:lt2}, $\M_E$ (equivalently
 $\Mre_E$) is indecomposable. By \cite[Proposition~3.12]{KV}\footnote{Irreducible in \cite{KV} is
 indecomposable here}, $\cZ_{\M_E}$ is
 an irreducible free locus. 
 By \cite[Corollary~3.6]{HKV}, $\cZ_{\M_E}(n)$ is an irreducible hypersurface  for large enough $n$.
 Thus, by \cite[Corollary~8.5]{HKV}, $\partial^1\cB_E(n)=\partial^1 \Qre_{E}(n)$
 is Zariski dense in $\Zre_{\Mre_E}(n)$ for large enough $n.$ 
 Thus $\{(X,X^*):X\in \partial^1\cB_E(n)\}$ is Zariski dense in 
  $\{(X,X^*): X\in M_n(\C)^g, \,  \det \Mre_E(X)=0\}$ for large enough $n.$ 
 By Lemma \ref{lem:Zdense} it now follows that $\{(X,X^*): X\in \partial^1\cB_E\}$
 is Zariski dense in $\cZ_{\M_E}=\cZ_{Q_E} =\{(X,Y): \det Q_E(X,Y)=0\}.$
 Thus the assumptions of
  Lemma \ref{l:hyperhair} are satisfied for some $n\in\N$, so 
   $\pi(\hair\cB_E)$ contains a hyperbasis for  $\C^e$.
 
$(\Leftarrow)$ Suppose $Q_E$ is not an atom. 
  If $E$ is not \bmin, then $\pi(\hair\cB_E)$ does not span $\C^e$
  by Proposition \ref{p:hairspans},
 since  $\ker(E^*)=\{0\}$. If $E$ is \bmin, 
then $\Mre_E$ is minimal but not \irrL by Lemma \ref{l:trans} items \eqref{it:lt2} and \eqref{it:lt3}. Thus  $\Mre_E$ decomposes non-trivially  as $\Mre_{E^1}\oplus \Mre_{E^2}$ by Lemma \ref{l:trans}\eqref{it:lt4}. Hence $Q_E$ decomposes as $Q_{E^1}\oplus Q_{E^2}$. Letting $e_i\ge 1$ denote the size of $Q_{E^i}$, 
$$\pi(\hair\cB_E)\subseteq
\left(\C^{e_1}\oplus \{0\}^{e_2}\right)\cup \left( \{0\}^{e_1}\oplus \C^{e_2} \right).$$
Thus $\pi(\hair\cB_E)$ cannot contain a hyperbasis for $\C^e=\C^{e_1}\oplus \C^{e_2}$.
\end{proof}

\begin{remark}\rm
\ben[\rm(1)]
\item
Note that $Q_E$ is an atom, $\ker (E)=\{0\}$ and $\ker (E^*)=\{0\}$ 
 (or equivalently, $\M_E$ is \irrL)
 if and only if the centralizer of
\[  
\bem 0 & {E_1}\\ 0 & 0  \eem, 
\dots
\bem 0 & {E_g} \\ 0 & 0  \eem, 
\bem 0 & 0 \\ {E_1^*}  & 0  \eem, 
\dots
\bem 0 & 0\\ {E_g^*}  & 0  \eem, 
\] 
is trivial. Verification of this fact amounts to checking whether a system of linear equations 
has a solution.
\item
If $\M_E$ is \irrL, then so is $L_E$. Indeed, if $L_E=L_{E^1}\oplus L_{E^2}$, then $\M_E$ equals $\M_{E^1}\oplus \M_{E^2}$ up to a canonical shuffle.

However, the converse is not true. For example,
with
$
\Lambda(x)=\bem 0 & x_2 \\ x_1 & 0\eem,
$
\[
I + \Lambda(x) + \Lambda^*(y)= \bem 1& x_2 +y_1 \\ x_1 + y_2 & 1 \eem 
\]
is an \irrL \mop, but
$$I - \Lambda \Lambda^*
= \bem 1- x_1 y_1 & 0\\
   0 & 1-x_2y_2
   \eem
   $$
   factors.\qed
   \een
\end{remark}

\subsection{The eig-generic conditions}\label{sssec:eig}

In this subsection we connect the various genericity assumptions on tuples in $M_d(\C)^g$ used in \cite{AHKM18} 
to clean, purely algebraic conditions of the 
corresponding \hmm  pencils, see Proposition \ref{p:eigony}.
We begin by recalling these assumptions precisely.

\begin{definition}[\protect{\cite[\S 7.1.2]{AHKM18}}]\rm
 \label{def:generic-weak}
  A tuple $A\in \matdg$  is \df{weakly eig-generic} if there exists an $\ell\le d+1$ and, for $1\le j\le \ell$,   positive integers $n_j$  and tuples $\alpha^j\in \matnjg$ such that
\ben[\rm (a)]
 \item \label{it:oneD}
   for each $1\le j\le \ell$, the eigenspace corresponding to the
   largest eigenvalue of $\Lambda_A(\alpha^j)^*\Lambda_A(\alpha^j)$ has dimension one and hence is spanned by a vector  $u^j = \sum_{a=1}^{n_j}  u^j_a\otimes e_a$; and 
 \item \label{it:span} the set  $\mathscr U =\{u^j_a: 1\le j\le \ell, \, 1\le a \le n_j\}$ contains a hyperbasis for  $\ker(A)^\perp =\rg(A^*)$. 
\een
 The tuple is \df{eig-generic} if it is weakly eig-generic and 
   $\ker(A)=\{0\}$ (equivalently, $\rg(A^*)=\C^d$).

  Finally, a tuple $A$ is \df{$*$-generic} (resp.~\df{weakly $*$-generic}) if there exists an $\ell\le d$ and tuples $\beta^j\in M_{n_j}(\C)^g$ such that
the kernels of
   $I-\Lambda_A(\beta^j)\Lambda_A(\beta^j)^*$ have dimension one and are spanned by vectors $\mu^j = \sum \mu^j_a \otimes e_a$ for which the set $\{\mu^j_a: 1\le j \le \ell, 1\le a\le n_j\}$ spans $\C^d$ (resp.~$\rg(A)=\ker(A^*)^\perp$). 
\end{definition}

\begin{remark}\label{r:equal}\rm
One can replace $n_j$ with $\sum_{j=1}^\ell n_j$ in Definition \ref{def:generic-weak}, so we can without loss of generality assume $n_1=\cdots=n_g$. 
\qedhere
\qed \end{remark}

Mixtures of these generic conditions were critical assumptions in the main theorems of
\cite{AHKM18}. The next proposition gives elegant and much more familiar replacements for them.

\begin{prop}
\label{p:eigony}
Let $A\in\mat{d}^g$.

\begin{enumerate}[{\rm(1)}]
	\item \label{i:eigony1} $A$ is eig-generic if and only if $Q_A$ is an atom and $\ker (A)=\{0\}$. 
	\item \label{i:eigony2} $A$ is $*$-generic and $\ker(A)=\{0\}$ if and only if 
${A^*}$ is \bmin.
	\item \label{i:eigony3} Let $\iota$ denote the inclusion of  $\rg(A^*)$
  into $\C^d.$   Then $A$ is weakly eig-generic if and only if $Q_{A\iota}$ is an atom and $\ker (A\iota)=\{0\}$. 
	\item \label{i:eigony4} Let $\iota$ denote the inclusion of $\rg(A)$ into $\C^d.$
         Then $A$ is weakly $*$-generic and $\ker(A)=\{0\}$ if and only if ${A^*\iota}$ is \bmin. 
\end{enumerate}
\end{prop}

\begin{proof}
 It is immediate from the definitions that if $\pi(\hair\cB_A)$ contains
 a hyperbasis,   then $A$ is eig-generic.  On the other hand,
 if $(\alpha,u) \in \widehat{\partial^1 \cB_E}.$
 then $u$ is an eigenvector of  $\Lambda_A(\alpha)^*\Lambda_A(\alpha)$
 corresponding to its largest eigenvalue $1$. Writing 
 $u=\sum_{a=1}^n u_a\otimes e_a\ne 0,$ each $u_a\in \pi(\hair\cB_E)$ because 
 if $U$ is a unitary matrix, then $(U \alpha U^*,Uu)\in \widehat{\partial^1 \cB_E}.$
 Hence $\pi(\hair\cB_A)$ contains a hyperbasis if and only if
 $A$ is eig-generic and therefore item \eqref{i:eigony1}   Follows from Proposition \ref{p:hairhyperspans} and Remark \ref{r:equal}.

 A similar argument to that above shows $\pi(\hair \cB_{A^*})$ spans
 if and only if $A$ is $*$-generic. Thus item \eqref{i:eigony2} follows from the 
 Proposition \ref{p:hairspans} and Remark \ref{r:equal}.

Item \eqref{i:eigony3} follows from \eqref{i:eigony1} since
 $A\iota$ is eig-generic and $\ker(A\iota)=\{0\}.$

Item \eqref{i:eigony4} follows from \eqref{i:eigony2} since $\iota^*A$
 is weakly $*$-generic and $\ker(\iota^*A)=\ker(A).$
\end{proof}

\subsection{Proof of Theorem \ref{thm:sdmain}}
\label{sec:proof1}

We use Proposition \ref{p:eigony}.
In the terminology of \cite{AHKM18},
assumptions \eqref{it:sdmaina}
and \eqref{it:sdmainb}
imply that $A$ is eig-generic and $*$-generic,
and $B$ is eig-generic, since the \bmin hypothesis on $A^*$ implies $\ker(A)=\{0\}.$ 
Theorem \ref{thm:sdmain} thus follows
from \cite[Corollary 7.11]{AHKM18} once it is verified
that the assumptions imply $\cD_B$ is bounded,
$p(\partial \cD_A) \subset \partial \cD_B$ and
$q(\partial \cD_B)\subset \partial \cD_A$.
For instance, if $X\in \partial \cD_A$, but $p(X)\in \interior(\cD_B)$,
then there is a $Z\notin \cD_A$ such that $p(Z)\in \cD_B$. But then,
$Z=q(p(Z))\in \cD_A$, a contradiction.
\qed

\section{Bianalytic maps between spectraballs and free spectrahedra}
\label{sec:maindeal}
In this section we prove the rest of our 
  main results, Proposition \ref{p:HKVpBergman},
and then Theorem \ref{t:pencil-ball-alt} and its Corollary \ref{t:B2B}.

\subsection{The proof of Proposition \ref{p:HKVpBergman}}
\label{sec:aux}
 Throughout this subsection, we 
 fix a tuple $E\in M_{d\times e}(\C)^g,$ a positive integer $M$
 and an $F\in \pxx^{1\times e}$  of degree degree at  most  $M.$
  Write  $F=\begin{pmatrix} F^1 & \cdots & F^e\end{pmatrix}$ and 
 \[ 
 F^s = \sum_{|w|\le M} F^s_w w,
\] 
 where {$|w|$} denotes the \df{length of the word} $w$
 and $F^s_w\in\C.$

  Let $S$ denote the tuple of shifts on the truncated Fock space
 $\sF_M$ with orthonormal basis the  words of length at most $M$ 
 in the freely noncommuting variables $\{x_1,\dots,x_g\}.$
 When viewing a word $w$ as an element of the finite dimensional
 Hilbert space $\sF_M$ we will write $\vecd{w}.$
 Thus $S_\ell \vecd{w} = \vecd{x_\ell w}$ if $|w|<M$ and $S_\ell \vecd{w} =0$ if $|w|=M.$
  Let $P$ denote the projection of $\sF_M$ onto
 the subspace $\sF_{M-1}$ and note that $S_k^*S_\ell=P$ if $k=\ell$
 and $S_k^*S_\ell=0$ if $k\ne \ell.$

 Given a matrix $\beta =(\beta_{j,k})_{j,k=1}^g \in M_g(M_r(\C))$ and
  words $u,w$ of the same  length $N$,
\[
 u= x_{j_1} x_{j_2}\cdots x_{j_N}, \ \ 
 w= x_{k_1} x_{k_2} \cdots x_{k_N},
\] 
 let
\[
 \wbeta_{u,w}  = \beta_{k_1,j_1} \beta_{k_2,j_2}\cdots \beta_{k_N,j_N}.
\]
 In particular, $\beta_{j,k} =\wbeta_{x_k,x_j}$ 
\begin{equation}
 \label{e:wbetaproducts}
 \wbeta_{u,w}\wbeta_{x_j,x_k} 
  = \beta_{k_1,j_1} \beta_{k_2,j_2}\cdots \beta_{k_N,j_N} \beta_{k,j}
  = \wbeta_{ux_j,wx_k}.
\end{equation}
 Let 
\[
 (\beta\cdotb S)_j = \sum_{k=1}^g  \beta_{j,k}\otimes  S_k
\]
 and $\beta\cdotb S = ( (\beta\cdotb S)_1,\dots, (\beta\cdotb S)_g).$

\begin{lemma}
 \label{l:combinatorics}
    Given $1\le N\le M$ and a  word $w$ of length $N$, 
\[
 (\beta\cdotb S)^w =  \sum_{|u|=N}  \wbeta_{u,w} \otimes S^u.
\]
\end{lemma}

\begin{proof}
We induct on $N.$ For $N=1$ and $w=x_t$,
\[
(\beta\cdotb S)^w=  \sum_{k=1}^g \beta_{t,k} \otimes S_k =\sum_{k} \wbeta_{x_k,x_t} S_k
  = \sum_{|u|=1} \wbeta_{u,x_t} S^u
\]
 Now suppose the result holds for $N.$ Let $v$ be a word of length $N$
 and consider the word $w=vx_t$ of length $N+1.$ Using the induction
 hypothesis and equation \eqref{e:wbetaproducts},
\[
 \begin{split}
 (\beta \cdotb S)^{w} & =  (\beta\cdotb S)^{v} (\beta\cdotb S)^{x_t}
    = [\sum_{|u|=N} \wbeta_{u,v} \otimes S^u]\, [\sum_{k} \beta_{t,k}\otimes S_k]\\
    & =  \sum_{|u|=N}\sum_{k=1}^g  \wbeta_{u,v}\beta_{t,k} \otimes S^u S_k
    = \sum_{|u|=N} \sum_{k=1}^g  \wbeta_{ux_k,vx_t} \otimes S^{ux_k} \\
    & =  \sum_{|z|=N+1} \wbeta_{z,w} \otimes S^{z}.  \qedhere
\end{split}
\]
\end{proof}

 Given $N$, let \df{$\sG_N$} denote the subspace of $\sF_M$ spanned by
 words of length $N.$  Thus the words of length $N$ form
 an orthonormal basis for $\sG_N.$ 
 Given words
 $u,w\in \sG_N,$  let $\vecd{u} \, \vecd{w}^*$ 
 denote the linear mapping on $\sG_N$
 determined by $\vecd{u} \, \vecd{w}^* \vecd{v}=\langle \vecd{v},\vecd{w}\rangle \vecd{u},$
  for words $v\in \sG_N.$
Let 
\[
 B(\beta,N) = \sum_{|u|=N=|w|} \wbeta_{u,w} \otimes \fE= \begin{pmatrix} \wbeta_{u,w} \end{pmatrix}_{|u|=N=|w|}
  \in M_r(\C)\otimes  M_{g^N}(C),
\]
  where the second equality is understood in the sense of unitary equivalence.
 In particular, $B(\beta,1) = \begin{pmatrix} \beta_{k,j} \end{pmatrix}_{j,k=1}^g.$

\begin{lemma}
 \label{l:stillinvertible}
   For each positive integer $N$  the set of $\beta\in M_g(M_r(\C))$ such that
   $B(\beta,N)$ is invertible is open and dense.
\end{lemma}

\begin{proof}
  For the second statement, observe that $B(I,N)$ is the identity 
  matrix since, with $\beta_{j,k}=\updelta_{j,k}I_r,$ we have
  $\wbeta_{u,w} =\delta_{u,w} I_r.$  Hence the mapping 
  $\psi:M_g(M_r(\C))\to \C$ defined by $\psi(\beta)=\det B(\beta,N)$
  is a polynomial in the entries of $\beta$ that is not identically zero.
  Thus $\psi$  is nonzero on an open dense set and the result follows.
\end{proof}

 For notational purposes, let $1$ denote the emptyword $\vecd{\emptyset} \in \sF_M.$
 Let $\{\base_1,\base_2,\dots,\base_e\}$  denote 
  the standard orthonormal basis for $\C^e.$ 

\begin{lemma}
 \label{l:onea}
 Suppose $\beta\in M_g(M_r(\C))$ and $\gamma= \sum_{s=1}^e \base_s \otimes \gamma_s\in \C^e\otimes \C^r.$
 If  
\[
  \sum_{s=1}^e  F^s(\beta\cdotb S)[\gamma_s \otimes 1] =0,
\]
 then, for $1\le N\le M$ and each word $u$ of length $N,$
\[
 \sum_{|w|=N} \wbeta_{u,w} [\sum_{s=1}^e  F^s_w \gamma_s] = 0.
\]
 Moreover, if  $B(\beta,N)$ is invertible, then
\[
 \sum_{s=1}^e  F^s_w \gamma_s =0
\]
 for each word $|w|=N.$
\end{lemma}

\begin{proof}
 Since $F^s_w\in \C,$ by Lemma \ref{l:combinatorics},
\[
 \sum_{N=0}^M \sum_{|w|=N} F^s_w (\beta\cdotb S)^w
  =\sum_{N=0}^M \sum_{|u|=N} \left [\sum_{|w|=N}  F^s_w \wbeta_{u,w}  \right] \otimes S^u.
\]
 Thus, 
\[
0=\sum_{s=1}^e F^s(\beta\cdotb S)[\gamma_s\otimes 1]
   =\sum_{N=0}^M  \sum_{|u|=N} \left ( \sum_{|w|=N} \wbeta_{u,w} 
    [\sum_{s=1}^e F^s_w \gamma_s]  \right ) \otimes \vecd{u}
\]
 and the first part of the result follows.

 To prove the second part,  let
\[
 y = \sum_{|v|=N} y_v\otimes \vecd{v} \in \C^r\otimes \sG_N,
\]
where $y_v=\sum_{s=1}^e F^s_v \gamma_s \in \C^r.$
 Thus 
\[
\begin{split}
    B(\beta,N) y & = \sum_{|u|=N=|w|} \wbeta_{u,w}\otimes \fE  
      \sum_{|v|=N} y_v \otimes v \\
   & = \sum_{|u|=N=|w|} \wbeta_{u,w} y_w \otimes \vecd{u}  =
     \sum_{|u|=N} [ \sum_{|w|=N} \wbeta_{u,w} y_w ]\otimes  \vecd{u} = 0.
\end{split}
\]
 Hence  if  $B(\beta,N)$ is invertible, then
  $y=0$ and therefore $\sum_{s=1}^e F^s_w  \gamma_s=0$
 for each $|w|=N.$
\end{proof}

  We continue to let $\{\base_1,\base_2,\dots,\base_e\}$ 
 denote the standard basis for $\C^e.$ Let 
   $\{\basr_1,\dots,\basr_r\}$ denote 
  the standard orthonormal basis for $\C^r.$

\begin{prop}
 \label{p:ifspan}
 Fix $1\le N\le M.$
 If there exist a positive integer $r$ and
   $(\beta^a, \gamma^a)\in M_g(M_r(\C))\times [\C^e\otimes \C^r]$ 
     for $1\le a\le e$ such that, 
 \begin{enumerate}[\rm (a)]
  \item writing 
\[ 
  \gamma^a =\sum_{t=1}^r \delta^a_t\otimes \basr_t 
\]
      the vectors $\{\delta^a_1: 1\le a\le e\}$ span $\C^e$; 
   \item \label{i:ifspan2} 
         $B(\beta^a,N)$ is invertible for each $1\le a\le e;$ 
  \item \label{i:ifspan3} $F(\beta^a\cdotb S)[\gamma^a\otimes 1]=0$ for each $1\le a\le e$, 
 \end{enumerate} 
 then $F^s_w=0$ for each $1\le s \le e$ and $|w|=N.$
\end{prop}

\begin{proof}
 Note that 
\[
0 = F(\beta^a\cdotb S)[\gamma^a \otimes 1] = 
  \sum_{s=1}^e F^s(\beta^a \cdotb S) [\gamma^a_s \otimes 1].
\]
 Thus items \eqref{i:ifspan2} and \eqref{i:ifspan3} validate
 the hypotheses of  Lemma \ref{l:onea}, and hence  $\sum_s F^s_w \gamma^a_s =0$
 for each $|w|=N$ and  $1\le a\le e.$  Writing
  $\gamma^a 
         = \sum_{s=1}^e \base_s \otimes \gamma^a_s$,
 it follows that 
\[
 \sum_{s=1}^e  [\basr_1^*\gamma^a_s] \base_s 
   =  (I\otimes \basr_1^*) \sum_{s=1}^e  \base_s\otimes \gamma^a_s 
   = \delta^a_1 = \sum_{s=1}^e [\base_s^* \delta^a_1]\base_s.
\]
 Therefore $\basr_1^*\gamma^a_s = \base_s^* \delta^a_1$ and consequently,
 for $|w|=N,$
\[
 0 =     \sum_{s=1}^e  F^s_w  [\basr_1^* \gamma^a_s] 
     = \sum_{s=1}^e F^s_w [\base_s^*\delta^a_1] 
    = F_w \delta^a_1,
\]
 where $F_w = \begin{pmatrix} F^1_w & \dots & F^e_w\end{pmatrix}\in \C^{1\times e}.$
 Since, by hypothesis, $\{\delta^a_1:1\le a\le e\}$ spans $\C^e$ it follows
 that $F_w=0$ whenever $|w|=N.$  Thus $F^s_w=0$ for $1\le s\le e$
 and $|w|=N.$
\end{proof}

 Given $\beta=(\beta_{j,k}) \in M_g(M_r(\C))$, let
\[
 (E\cdotb \beta)_k = \sum_{j=1}^g E_j\otimes \beta_{j,k}
\]

\begin{lemma}
 \label{l:schifty}
   For  $\beta \in M_g(M_r(\C))$,
\[
\begin{split}
 \Lambda_E(\beta\cdotb S) & =   %
       \sum_k  (E\cdotb \beta)_k  \otimes S_k\\
 \Qre_E(\beta\cdotb S) 
     & =  [I -\sum_k  (E\cdotb \beta)_k^* (E\cdotb \beta)_k ]
      \otimes 
       P \, +  \, I\otimes (I-P), 
\end{split}
\]
 where $P$ is the projection of $\sF_M$ onto $\sF_{M-1}.$
\end{lemma}

\begin{proof}
 Compute, %
\[
   \Lambda_E(\beta\cdot S)
  =\sum_{j=1}^g  E_j\otimes (\sum_{k=1}^g  \beta_{j,k}\otimes S_k) 
  =  \sum_{k=1}^g  [\sum_{j=1}^g  E_j \otimes \beta_{j,k}]\otimes S_k
  = \sum_{k=1}^g  (E\cdotb \beta)_k \otimes S_k,
\]
 and thus
\[
 \Lambda_E(\beta\cdot S)^* \Lambda_E(\beta\cdot S)
      = [\sum_{k=1}^g   (E\cdotb \beta)_k^* (E\cdotb \beta)_k] \otimes P
\]
and the result follows. 
\end{proof}

\subsubsection{The hair spanning condition}
 A subset $\{(\alpha^a,\gamma^a):1\le a\le e\}\subset M_r(\C)^g \times
 [\C^e\otimes \C^r]$ is a \df{boundary spanning set}  for $\cB_E$ if 
  each $(\alpha^a,\gamma^a)\in \widehat{\partial \cB_E}$ and,
 writing $\gamma^a=\sum_{t=1}^r \delta^a_t \otimes \basr_t$, 
 the set $\{\delta^a_1: 1\le a\le e\}$ spans $\C^e.$ This
 set is a \df{boundary hair spanning set} for $\cB_E$ if moreover  
 $(\alpha^a,\gamma^a) \in \widehat{\partial^1 \cB_E}$ for each $a.$
 By Proposition \ref{p:hairspans}, if $E$ is \bmin, then 
 there exists a boundary hair  spanning set for $\cB_E.$

\begin{prop}
 \label{p:hairspantospan}
 Fix $1\le N\le M.$
 If $E\in M_{d\times e}(\C)^g $ is \bmin, then there exists
 a positive integer $r$ and a subset $\{(\beta^a,\gamma^a):1\le a\le e\}$
  of $M_g(M_r(\C))\otimes [\C^e\otimes \C^r]$ 
 such that $B(\beta^a,N)$ is invertible for each $1\le a\le e$ and 
  $\{(\beta^a\cdotb S,\gamma^a\otimes 1): 1\le a\le e\}$ is 
  a boundary spanning set for $\cB_E.$  
\end{prop}

The proof of Proposition \ref{p:hairspantospan} uses the following
special case  of a standard result from the theory of perturbation of matrices
\cite[Chapter 2, Section 4]{Kato}.

\begin{lemma}
 \label{l:preserve}
   Suppose $R\in M_d(\C),$  $I-R\succeq 0$ and $\ker(I-R)$ is one-dimensional
   and spanned   $v\in \C^d.$ For each $\epsilon>0,$ 
   there is a $\mu>0$ such that if $Q\in M_d(\C)$ is 
   self-adjoint and $\|Q\|<\mu$, then there is a $c>0$ and 
   $w\in \C^d$ such that $I-c(R+Q)\succeq 0,$ 
    $\ker(I-c(R+Q))$ is spanned by $w$ and $\|v-w\|<\epsilon.$
\end{lemma}

\begin{proof}[Proof of Proposition~\ref{p:hairspantospan}]
  Since $E$ is \bmin, there is an $r$ and  a boundary hair spanning set
  $\{(\alpha^a,\zeta^a):1\le a\le e\} \subset M_r(\C)^g\times [\C^e\otimes \C^r]$
  for $\cB_E$ by    Proposition \ref{p:hairspans}. 
  In particular, writing $\zeta^a =\sum_{t=1}^r \chi^a_t \otimes\basr_t,$
  the set $\{\chi^a_1:1\le a\le e\}$ spans $\C^e.$
   There  is an $\epsilon>0$ such that, if 
   $\tau^a=\sum_{t=1}^g \tau^a_t\otimes \rho_t$ and 
  $\| \zeta^a-\tau^a\|<\epsilon$ for each $1\le a\le e,$ 
  then the set $\{\tau^a_1: 1\le a\le e\}$ spans $\C^e.$

  Fix $1\le a\le e$ and let, for $1\le j,k\le g,$
\[
 \tbeta^a_{j,k} = \begin{cases} \alpha^a_j & \ \mbox{ if } k=1\\
                                  0 & \ \mbox{ if } k>1. \end{cases}
\]
 Thus 
\[
   I-[\sum_{j=1}^g E_j \otimes \tbeta^a_{j,1}]^* \, 
       [\sum_{j=1}^g E_j \otimes \tbeta^a_{j,1}]
  = \Qre_E(\alpha^a)
\]
 is positive semidefinite with kernel 
  spanned by $\zeta^a.$
 By Lemmas \ref{l:stillinvertible} and \ref{l:preserve},
 there exists a $\beta^a\in M_g(M_r(\C))$ 
  such that $B(\beta^a,N)$ is invertible and 
\begin{equation}
 \label{e:afterperturb}
   R(\beta^a):= I-\sum_{k=1}^g \left ( [\sum_{j=1}^g E_j\otimes \beta_{j,k}]^* \,
          [\sum_{j=1}^g E_j\otimes \beta_{j,k}]\right )
      = I -\sum_{k=1}^g (E\cdotb \beta)_k^* \, (E\cdot \beta)_k 
\end{equation} 
  is positive semidefinite and has kernel spanned by
  a vector $\gamma^a$ such that $\|\zeta^a-\gamma^a\|<\epsilon.$
  In particular, writing $\gamma^a=\sum_{t=1}^r \delta^a_t \otimes\basr_t,$ 
  from the first paragraph of the proof, 
  the set  $\{\delta^a_1:1\le a\le e\}$ spans $\C^e.$ 

  To complete the proof,  observe, using $R(\beta^a)$ defined
 in equation \eqref{e:afterperturb} and Lemma \ref{l:schifty}, that
\[
\Qre_E(\beta^a \cdotb S) 
  = R(\beta^a)\otimes P \, + \, I\otimes (I-P).
\]
  It follows that $\{(\beta^a\cdotb S,\gamma^a\otimes 1):1\le a\le e\}$ 
 is a boundary spanning set for $\cB_E.$
\end{proof}

\subsubsection{Proof of Proposition \ref{p:HKVpBergman}}
 Suppose $E$ is \bmin\footnote{It is enough to assume that
  $PE$ is \bmin, where $P$ is the projection of $\C^d$ onto
 $\rg(E).$}  and $F\in \pxx^{1\times e}$ vanishes on 
 $\widehat{\partial \cB_E}$ and  has degree at most $M.$ 
 
 Fix $1\le N\le M.$ 
 By Proposition \ref{p:hairspantospan}, there exists an $r>0$ and 
 $(\beta^a,\gamma^a)\in M_g(M_r(\C))\times [\C^e\otimes \C^r]$
 such that $\{(\beta^a\cdotb S,\gamma^a\otimes 1):1\le a\le e\}$
 is a boundary spanning set for $\cB_E$ and  $B(\beta^a,N)$
 is invertible for each $1\le a\le e.$
   Since $(\beta^a\cdotb S,\gamma^a)\in \widehat{\partial \cB_E},$
 it follows that   $0= F(\beta\cdotb S) \gamma^a.$ 
 An application of Proposition \ref{p:ifspan} implies 
 $F^s_w=0$ for all $1\le s\le e$ and $|w|=N.$  
 Hence $F^s_w=0$ for all $1\le s\le e$ and $|w|\le M$ and
 therefore $F=0.$ 
 To complete the proof, given $V\in \pxx^{\ell\times e}$ that
 vanishes on $\widehat{\partial \cB_E}$, apply
 what has already been proved to each row of $V$ to conclude $V=0.$
\qed

\subsection{Theorem~\ref{t:pencil-ball-alt}}
\label{sec:winning}
In this subsection we prove the first part  Theorem~\ref{t:pencil-ball-alt}.
(The conversely portion was already proved as Corollary~\ref{c:pba-converse}.)

A free analytic mapping $f$ into $M(\C)^h$ defined in a neighborhood of $0$
 of $M(\C)^g$ has a power series expansion (\cite{HKMforbill,Voi04,KVV14}),
\begin{equation}
\label{e:powf}
 f(x) = \sum_{j=0}^\infty G_j(x) =\sum_{j=0}^\infty \sum_{|\alpha|=j} f_{\alpha} x^\alpha,
\end{equation}
where 
$f_\alpha \in \C^{1\times h}$. The term $G_j$ is  the \df{homogeneous of degree $j$}
 part of $f$. It is a polynomial mapping $M(\C)^g\to M(\C)^h$.

\begin{lemma}
\label{lem:useproper}
 Let $E\in M_{d\times e}(\C)^g$
  and $B\in M_r(\C)^h$. 
  Suppose 
 $f:\interior(\cB_E)\to  \interior(\cD_B)$ is proper.
 For each positive integer $N$ there exists a free polynomial
 mapping $p=p_N$ of degree at most $N$ such that if 
  $X\in \cB_E$ is nilpotent of order $N$, then 
 $f_X(z):=f(zX)=p(zX)$ for $z\in \C$ with $|z|<1$.  Further,
 if $X\in \partial \cB_E$ (equivalently $\|\Lambda_E(X)\|=1$),
 then $p(X)\in \partial \cD_B$.  
\end{lemma}

\begin{proof}
 Fix a positive integer $N$.
The  series expansion of equation \eqref{e:powf} converges as written
 on $\mathcal N_\epsilon =\{X\in M(\C)^g : \sum X_j X_j^* \prec \epsilon^2\}$
 for any $\epsilon>0$ such that $N_\epsilon \subset \interior(\cB_E)$
 \cite[Proposition 2.24]{HKMforbill}. In particular, if $X\in \cB_E$
 is nilpotent of order $N$ and $|z|$ is small, then 
\[
 f_X(z):=f(zX)=  \sum_{j=1}^N G_j(zX) = \sum_{j=1}^N \left [\sum_{|\alpha|=j} 
   f_\alpha \otimes X^\alpha \right ] z^j =: p(zX).
\]
It now follows that $f_X(z)=p(zX)$ for $|z|<1$ (since $zX\in \interior(\cB_E)$
 for such $z$ and both sides are analytic in $z$ and agree on a
 neighborhood of $0$).

 Now suppose $X\in\partial \cB_E(n)$ (still nilpotent of order $N$).
 Since $f:\interior(\cB_E)\to \interior(\cD_B)$, it follows that
 $\LRE_B(p(tX))\succ 0$ for $0<t<1$.  Thus $\LRE_B(p(X))\succeq 0$.
  Arguing by contradiction, suppose
 $\LRE_B(p(X))\succ 0$; that is $p(X)\in \interior(\cD_B(n))$. Hence
 there is an $\eta$ such that 
\[
\overline{B}_\eta(p(X)):=\{Y\in M_n(\C)^g: \|Y-p(X)\|\le \eta\}
  \subset \interior(\cD_B(n)).
\]
Since $K=\overline{B}_\eta(p(X))$ is compact,
$L=f_n^{-1}(K)\subset \interior(\cB_E)$ is also compact by
the proper hypothesis on $f$ (and hence on  each 
$f_n:\interior(\cB_E(n))\to \interior(\cD_B(n))$).
On the other hand, for $t<1$ sufficiently large, $tX\in L$,
but $X\notin \interior(\cB_E(n))$, and we have arrived
 at the contradiction that $L$ cannot be compact.
\end{proof}

\begin{remark}\rm
\label{rem:useproper}
 In view of Lemma \ref{lem:useproper}, for $X\in \partial \cB_E$
 nilpotent we let $f(X)$ denote $f_X(1)$.  Observe also, if $g=h$,
 $f(0)=0,$  $f^\prime(0)=I_g$ and $X\in \cB_E$ is nilpotent of 
 order two, then $f(X)=X$.
\qed\end{remark}

\begin{lemma}
\label{l:agr2}
 Suppose $B\in M_r(\C)^g$ and $\QQ\in M_{r\times u}(\C)$
 and let $\mathscr{B}$ denote the algebra generated by $B$.
 Let $h$ denote the dimension of $\mathscr{B}$ as a vector space.  If
  $\{B_1\QQ,\dots,B_g\QQ\}$ is linearly independent, then there 
  exists a $g\le t\le h$ and a basis $\{J_1,\dots,J_h\}$ of
  $\mathscr{B}$ such that
\begin{enumerate}[\rm (1)]
 \item $J_j=B_j$ for $1\le j\le g$;
 \item \label{it:JQind} $\{J_1\QQ,\dots,J_t\QQ\}$ is linearly independent; and 
 \item $J_j\QQ=0$ for $t<j\le h$.
\end{enumerate}

 Letting $\Xi\in M_h(\C)^h$ denote the convexotonic tuple associated to $J$,
\[
 (\Xi_j)_{\ell,k} = 0 \ \mbox{ for }  j>t, \, k\le t \mbox{ and } 1\le \ell\le h.
\]
\end{lemma}

\begin{proof}
 The set $\mathscr{N}=\{T\in \mathscr{B}: T\QQ=0\}\subseteq \mathscr{B}$ is a subspace(in fact a left ideal). Since $\{B_1\QQ,\dots,B_g\QQ\}$ is 
 linearly independent,  the subspace  $\mathscr{M}$ of $\mathscr{B}$
 spanned by  $\{B_1,\dots,B_g\}$  has dimension $g$ and satisfies 
 $\mathscr{M}\cap \mathscr{N}=\{0\}$. Thus there is a $g\le t\le h$
 such that  $h-t$ is the dimension
 of $\mathscr{N}.$ Choose a basis  $\{J_{t+1},\dots,J_h\}$ for $\mathscr{N}$. 
   Thus the set $\{B_1,\dots,B_g,J_{t+1},\dots,J_h\}$
 is linearly independent and $g\le t\le h.$ Extend it to a basis
 $\{J_1,\dots,J_h\}$. To see that item \eqref{it:JQind}
 holds, we argue by contradiction. If $\{J_1\QQ,\dots,J_t\QQ\}$ is linearly
 dependent, then some linear combination of $\{J_1,\dots,J_t\}$
 lies in $\mathscr{N}$.

 The last statement is a consequence of the fact that $\mathcal N$
 is a left ideal. Indeed,  since  the tuple $\Xi$ satisfies,
\[
 J_\ell J_j =\sum_{k=1}^h (\Xi_j)_{\ell,k} J_k
\]
 for $1\le j,\ell\le h$ we have, for  $j>t$ and $1\le \ell \le h$,
\[
0= J_\ell J_j\QQ = \sum_{k=1}^h (\Xi_j)_{\ell,k} J_k \QQ 
  =\sum_{k=1}^t (\Xi_j)_{\ell,k} J_k \QQ.
\]
 By independence of $\{J_k\QQ: 1\le k\le t\}$, it follows that
 $(\Xi_j)_{\ell,k}=0$ for $k\le t$.
\end{proof}

\begin{lemma}
\label{l:DimpliesB}
Let $E\in M_{d\times e}(\C)^g$ and $\AF\in M_{r}(\C)^g$. If there is a
proper free  analytic mapping $f:\interior(\cB_E)\to \interior(\cD_{\AF})$
such that $f(0)=0$ and $f^\prime(0)=I$, then $\cB_E=\cB_{\AF}$.
\end{lemma}

\begin{proof}
 We perform the off diagonal trick. Given a tuple $X$, let 
\[
S_X = \begin{pmatrix} 0 & X \\ 0 & 0 \end{pmatrix}.
\]
  Suppose $X\in M_n(\C)^g$ and $\|\Lambda_E(X)\|=1$. It 
 follows that $\|\Lambda_E(S_X)\|=1.$  Thus
  $S_X\in \partial \cB_E$.  Since 
 $f:\interior(\cB_E)\to \interior(\cD_A)$
 is proper with $f(0)=0$ and $f^\prime(0)=I$ 
 (and $S_X$ is nilpotent), $f(S_X)=S_X$
 (see Remark \ref{rem:useproper}), and $S_X\in \partial \cD_{\AF}$.
  Thus  $I+\Lambda_{\AF}(S_X)+\Lambda_{\AF}(S_X)^*$ is positive semidefinite
 and has a (non-trivial) kernel. Equivalently,
\[
  1=\|\Lambda_{\AF}(S_X)\|=\|\Lambda_{\AF}(X)\|.
\]
  Hence, by homogeneity, $\|\Lambda_E(X)\|=\|\Lambda_{\AF}(X)\|$
 for all $n$ and $X\in M_n(\C)^g$. Thus $\cB_E=\cB_{\AF}$.
\end{proof}

\subsubsection{Proof of Theorem~\ref{t:pencil-ball-alt}}
 We assume, without loss of generality, that $E$ is \bmin.
 We will now show $f$ is convexotonic.

 Lemma \ref{l:DimpliesB} applied to the
 proper free analytic mapping $f:\interior(\cB_E)\to \interior(\cD_A)$
 gives $\cB_E=\cB_A.$ 
 Applying  Lemma \ref{l:trans}\eqref{it:lt7} there exist
 $r\times r$ unitary matrices  $\ttW$ and $\ttV$ such that
 $A=\ttW (\begin{smallmatrix} E & 0 \\ 0 & R\end{smallmatrix}) \ttV^*,$
 where $R\in M_{(r-d)\times(e-d)}(\C)^g$ and  $\cB_E\subset \cB_R$.
 Replacing $A$ with the unitarily equivalent tuple $\ttV^* A\ttV$, we assume 
\begin{equation}
\label{e:agr1}
 A = U  \begin{matrix}   \begin{matrix} e & r\minus e \end{matrix} & \\
 \begin{pmatrix} E \, & 0\, \\ 0\,  & R\,   \end{pmatrix} &\!\!\!\! \begin{matrix} d\hfill \\ r\minus d \end{matrix}  \end{matrix}
 \
\end{equation}
where
\begin{equation}
\label{e:U}
 U = \ttV^* \, \ttW =  
  \begin{matrix}   \begin{matrix} d & r\minus d \end{matrix} & \\
 \begin{pmatrix} U_{11} & U_{12}\\U_{21} & U_{22} \end{pmatrix} &\!\!\!\! \begin{matrix} e\hfill \\ r\minus e \end{matrix}  \end{matrix}.
\end{equation}
  With respect to the orthogonal  decomposition in equation \eqref{e:agr1}, let
\[
 \QQ=\begin{pmatrix} I_{e} \\ 0_{r-e,e} \end{pmatrix} \in M_{r\times e}(\C).
\]
 We will use later the  fact that if $\Qre_E(X)\succeq 0$ and
 $\Qre_E(X)\gamma =0$, then $\Qre_A(X)\QQ\gamma =0$. For now observe
\begin{equation}
\label{eq:AR}
A_j \QQ = U \begin{pmatrix} E_j \\ 0 \end{pmatrix}.
\end{equation}
 Thus, since $\{E_1,\dots,E_g\}$ is linearly independent,
 the set $\{A_1\QQ,\dots, A_g\QQ\}$ is linearly independent.

We now apply Lemma \ref{l:agr2} to $A$ 
in place of $B$ and obtain
a basis $\{J_1,\dots,J_h\}$ for $\mathscr{A}$, the algebra
 generated by $\{A_1,\dots,A_g\}$, and a $g\le t\le h$
 such that $J_j=A_j$ for $1\le j\le g$, 
 the set $\{J_j\QQ: 1\le j\le t\}$ is linearly independent
 and $J_j\QQ=0$ for $t<j\le h$. Let $\upxi\in M_h(\C)^h$ 
 denote the convexotonic tuple associated to $J$
 and let $\Xi=-\upxi.$ Thus $(\Xi_j)_{\ell,k}=0$
 for $j>t,$ $k\le t,$ and all $\ell$ and 
\[
  J_\ell J_j = - \sum_{s=1}^h (\Xi_j)_{\ell,s} J_s.
\]
Let  $\varphi:\interior(\cD_J)\to \interior(\cB_J)$
  denote the convexotonic map
\[
 \varphi(x) =x (I-\Lambda_{\Xi}(x))^{-1}
\]
 from Proposition \ref{prop:properobvious}. 
Let $\iota:\cD_A\to \cD_{J}$ denote the inclusion. 
 By Corollary \ref{cor:obvious}  the composition
  $\varphi\circ \iota$ is proper from $\interior(\cD_A)$
 to  $\interior(\cB_J)$. Hence, 
 $\ff=\varphi\circ \iota \circ f$ is
  proper from $\interior(\cB_E)$ to $\interior(\cB_J).$
Further $\ff(0)=0$ and $\ff^\prime(0)=\begin{pmatrix} I_g & 0 \end{pmatrix}$
 because essentially the same is true for each of the components
 $f,\iota,\varphi$. Thus
  $\ff(x) = \begin{pmatrix} x & 0 \end{pmatrix} + \rho(x)$,
 where $\rho(0)=0$ and $\rho^\prime(0)=0$.

Write 
\[
 \ff= \begin{pmatrix} \ff^{1} & \dots & \ff^{h}\end{pmatrix}.
\]
 Expand $\ff$ as a power series,
\[
  \ff= \sum H_j = \sum_{j=1}^\infty \sum_{|\alpha|=j} \ff_\alpha \, \alpha,
\]
where $H_j$ is the homogeneous of degree $j$ part of $\ff$.  Thus,
\[
 H_j = \begin{pmatrix} H_j^{1} & \dots & H_j^{h} \end{pmatrix}
\]
 and $H_1(x) = \begin{pmatrix} x & 0 \end{pmatrix}.$
Likewise,  
\[
\ff_{x_j}(x) = \begin{pmatrix} 0 & \dots &0 & x_j & 0 &\dots & 0\end{pmatrix}
\]
 for $1\le j\le g$ and $\ff_{x_j}=0$ for $j>g$.  

 The next objective is to show $H_m^s=0$ for  $m\ge 2$ and $s\le t$. 
 Given a positive integer $m$, let $S$ denote the $(m+1)\times (m+1)$ 
 matrix, indexed by $j,k=0,1,\dots,m,$ with $S_{a,a+1}=1$ and $S_{a,b}=0$
 otherwise. Thus $S$ has ones on the first super diagonal and $0$ 
  everywhere else and $S^{m+1}=0.$ 
  Let  $Y\in \cB_E$ be given.  Since $S\otimes Y$ is 
 nilpotent with $(S\otimes Y)^\alpha=0$ if $\alpha$ is a word with $|\alpha|>m$,
 Lemma \ref{lem:useproper} (and Remark \ref{rem:useproper})
 imply $\ff(S\otimes Y)\in \cB_J$; that is
 if $\|\Lambda_E(Y)\|\le 1$, then $\|\Lambda_J(\ff(S\otimes Y))\|\le 1.$ 
  Let $\sZ^j = \ff^j(S\otimes Y)=\sum_{\mu=1}^m S^\mu\otimes H_\mu^j(Y).$
  With respect to the natural block matrix decomposition,
  $\sZ^j_{0,m} = H^j_m(Y)$ and $\sZ^j_{m-1,m}=H^j_1(Y)$.
 Thus $\sZ^j_{m-1,m}=Y_j$
 for $1\le j\le g$ and $\sZ^j_{m-1,m}=H^j_1(Y)=0$ for $j>g.$
 Now $\|\Lambda_J(\sZ)\|\le 1$ is equivalent to 
 $I-\Lambda_J(\sZ)^*\Lambda_J(\sZ)\succeq 0$. Thus,
\[
 I-\Lambda_A(Y)^* \Lambda_A(Y) -\Lambda_J(H_m(Y))^* \Lambda_J(H_m(Y)) \succeq 0.
\]
 Multiplying on the right by $\QQ\otimes I$ and on the left by $\QQ^*\otimes I$,
\[
 I - \Lambda_{A\QQ}(Y)^* \Lambda_{A\QQ}(Y) 
  - \Lambda_{J\QQ}(H_m(Y))^* \Lambda_{J\QQ}(H_m(Y)) \succeq 0.
\]
 By equation \eqref{eq:AR} 
  $\Lambda_{A\QQ}(Y)^*\Lambda_{A\QQ}(Y) = \Lambda_E(Y)^*\Lambda_E(Y)$, and hence,
\begin{equation}
\label{e:QY0}
\begin{split}
\Qre_E(Y) - & \Lambda_{J\QQ}(H_m(Y))^* \Lambda_{J\QQ}(H_m(Y))\\
 & =   I - \Lambda_{E}(Y)^* \Lambda_{E}(Y)  
  - \Lambda_{J\QQ}(H_m(Y))^* \Lambda_{J\QQ}(H_m(Y)) \succeq 0.
\end{split}
\end{equation}

Let $V(y)=\Lambda_{J\QQ}(H_m(y)).$ If $(Y,\gamma)\in \widehat{\partial \cB_E},$
 then $\Qre_E(Y)\gamma=0$ and hence, by equation \eqref{e:QY0},
 $V(Y)\gamma=0.$ Thus $V$ vanishes on $\widehat{\partial \cB_E}$
 and hence $V=0$ by Proposition \ref{p:HKVpBergman}; that is
\[
 0 = V(y) = \Lambda_{J\QQ}(H_m(y)) = \sum_{j=1}^h  J_j\QQ \, H_m^j(y) 
  = \sum_{j=1}^t  J_j\QQ\, H_m^j(y).
\]
 Since $\{J_1\QQ,\dots,J_t \QQ\}$ is linearly independent, 
 it follows that $H_m^j(y)=0$ for all $1\le j\le t$ and all $m\ge 2$.
Hence,
\[
 \ff(x) =
  \begin{pmatrix} x & 0 & \Psi(x)\end{pmatrix} 
\]
 where the $0$ has length $t-g$ and $\Psi$ has length $h-t$ and moreover,
 $\Psi(0)=0$ and $\Psi^\prime(0)=0.$

 Let $\psi$ denote the inverse of $\varphi$, 
\[
 \psi(x) = x(I+\Lambda_{\Xi}(x))^{-1}.
\]
 Thus, $\psi\circ \ff= \iota \circ f = \begin{pmatrix} f(x) &0&0\end{pmatrix}$
 and consequently,
\begin{equation}
 \label{e:agr2}
 \begin{pmatrix} f(x) &0&0\end{pmatrix}
  = \begin{pmatrix} x & 0 & \Psi(x)\end{pmatrix} 
 \left((I+\Lambda_{\Xi}( \begin{pmatrix} x & 0 & \Psi(x)\end{pmatrix}))\right)^{-1}.
\end{equation}
 Rearranging gives,
\begin{equation}
\label{e:agr3}
 \begin{pmatrix} x & 0 & \Psi(x)\end{pmatrix}  
   = \begin{pmatrix} f(x) &0&0\end{pmatrix} 
  (I+\Lambda_{\Xi}( \begin{pmatrix} x & 0 & \Psi(x)\end{pmatrix})).
\end{equation}
We now examine the $k$-th entry on the right hand side
 of equation \eqref{e:agr3}. First,  
\[
\begin{split}
 \big(I+\Lambda_{\Xi}( \begin{pmatrix} x & 0 & \Psi(x)\end{pmatrix})\big)_{\ell,k}
 &= 
 \big(I+\sum_{j=1}^g \Xi_j x_j  +\sum_{j=t+1}^h \Xi_j \Psi_{j-t}\big)_{\ell,k} \\
  &= I_{\ell,k}  + \sum_{j=1}^g (\Xi_j)_{\ell,k} x_j + \sum_{j=t+1}^h (\Xi_j)_{\ell,k}\Psi_{j-t}.
\end{split}
\]
 Since $(\Xi_j)_{\ell,k}=0$
 for $j>t$ and $k\le t$ (see Lemma \ref{l:agr2}), if $k\le t$, then 
\begin{equation*}
 \big(I+\Lambda_{\Xi}( \begin{pmatrix} x & 0 & \Psi(x)\end{pmatrix})\big)_{\ell,k}
  = I_{\ell,k}  + \sum_{j=1}^g (\Xi_j)_{\ell,k} x_j
\end{equation*}
for all $\ell.$
 Hence, the right hand side of equation \eqref{e:agr3},   for $g<k\le t$ (so that 
  $I_{\ell,k}=0$ for $\ell\le g$) is,
\begin{equation}
\label{eq:glessk}
\sum_{\ell=1}^g  f^\ell(x)\, 
\big(I+\Lambda_{\Xi}( \begin{pmatrix} x & 0 & \Psi(x)\end{pmatrix})\big)_{\ell,k}
  =  \sum_{j,\ell=1}^g  (\Xi_j)_{\ell,k} f^\ell(x)\, x_j
\end{equation}
 and similarly, for $1\le k\le g$, 
\begin{equation}
\label{eq:gbiggerk}
\sum_{\ell=1}^g  f^\ell(x)\, 
   (I+\sum_{j=1}^g \Xi_j x_j  +\sum_{j=t+1}^h \Xi_j \Psi_{j-t})_{\ell,k}
=f^k(x) + \sum_{j,\ell=1}^g ({\Xi})_{\ell,k} f^\ell(x) \, x_j.
\end{equation}
 Combining equations \eqref{eq:glessk} and \eqref{e:agr3}, for $g<k\le t$,
\[
 \sum_{j=1}^g \left [\sum_{\ell=1}^g (\Xi_j)_{\ell,k} f^{\ell}(x)\right ]x_j =0.
\]
Hence, for each $1\le j\le g$ and $g<k\le t$,
\[
\sum_{\ell=1}^g (\Xi_j)_{\ell,k} f^{\ell}(x) =0.
\]
Since $\{f^1,\dots,f^g\}$ is linearly independent, it follows that
\begin{equation}
\label{e:moreXi=0}
 (\Xi_j)_{\ell,k} =0, \ \ 1\le j,\ell \le g, \, g<k\le t.  
\end{equation}

We next show  $\widehat{\Xi}\in M_g(\C)^g$ defined by 
\[
(\widehat{\Xi}_j)_{\ell,k} = (\Xi_j)_{\ell,k}, \ \ \ 1\le j,\ell,k\le g
\]
 is convexotonic.
 Using equation \eqref{e:moreXi=0}, for  $1\le j,\ell\le g$,
\begin{equation}
\label{e:agr4}
\begin{split}
 A_\ell A_j \QQ & = J_\ell J_j \QQ = -\sum_{s=1}^h (\Xi_j)_{\ell,s} J_s \QQ
  =  -\sum_{s=1}^t (\Xi_j)_{\ell,s} J_s \QQ \\
  &  =  -\sum_{s=1}^g (\Xi_j)_{\ell,s} J_s \QQ
  =  -\sum_{s=1}^g (\Xi_j)_{\ell,s} A_s \QQ.
\end{split}
\end{equation}
 Multiplying equation \eqref{e:agr4} on the left by $U^*$
 and using equation \eqref{eq:AR} gives
\begin{equation*}
 \begin{pmatrix} E_\ell& 0\\0 & R_\ell \end{pmatrix}\,
   (-U) \, \begin{pmatrix} E_j \\ 0 \end{pmatrix} 
  = \begin{pmatrix} \sum_{s=1}^g (\Xi_j)_{\ell,s} E_s \\ 0 \end{pmatrix}.
\end{equation*}
 Using equation \eqref{e:U}, it follows that 
\begin{equation}
\label{e:EUE}
 E_\ell (-U_{11}) E_j = \sum_{s=1}^g (\Xi_j)_{\ell,s} E_s =
   \sum_{s=1}^g (\widehat{\Xi}_j)_{\ell,s} E_s.
\end{equation}
By Lemma \ref{l:secret}, the tuple $\widehat{\Xi}$ is convexotonic.

Combining equation \eqref{e:agr3} and equation \eqref{eq:gbiggerk},
if $1\le k\le g$, then 
\[
\begin{split}
x_k & =  \sum_{\ell=1}^g  f^\ell(x)\, 
  (I+\Lambda_{\Xi}( \begin{pmatrix} x & 0 & \Psi(x)\end{pmatrix}))_{\ell,k} \\
 &=  f^k(x)+ \sum_{j,\ell=1}^g  (\Xi_j)_{\ell,k} f^\ell(x)\, x_j = f^k(x) + \sum_{j,\ell=1}^g (\widehat{\Xi}_j)_{\ell,k} f^\ell(x) \, x_j.
\end{split}
\]
Thus,
\[
 x = f(x)(I+\Lambda_{\widehat{\Xi}}(x))
\]
and consequently
\begin{equation}
\label{e:f}
 f(x) = x(I+\Lambda_{\widehat{\Xi}}(x))^{-1}
\end{equation}
 is convexotonic. 

We now complete the proof by showing, if $A$ 
is minimal for $\cD_A$ (we continue to assume $E$ is \bmin),
 then $A$ is unitarily equivalent to 
\begin{equation}
\label{e:BUE}
 B= U \begin{pmatrix} E& 0  \\ 0 & 0 \end{pmatrix} 
  = \begin{pmatrix} U_{11} E  & 0 \\ U_{21} E & 0 \end{pmatrix}  \in M_r(\C)^g
\end{equation}
and  $B$ spans an algebra.
To this end, using equations \eqref{e:BUE} and \eqref{e:EUE}, observe
\[
\begin{split}
 B_\ell B_j & =  \begin{pmatrix} U_{11}E_\ell U_{11} E_j & 0 \\ U_{21} E_\ell U_{11} E_j  & 0 \end{pmatrix}
   =  \sum_{s=1}^g (-\widehat{\Xi}_j)_{\ell,s}  \begin{pmatrix} U_{11} E_s & 0 \\ U_{21} E_s & 0 \end{pmatrix} 
  =  \sum_{s=1}^g (-\widehat{\Xi}_j)_{\ell,s}  B_s.
\end{split}
\]
 Thus $B$ spans an algebra and, by Proposition \ref{prop:properobvious}, 
 the convexotonic map $f$ of equation \eqref{e:f} is a bianalytic map 
 $f:\interior(\cB_B)\to \interior(\cD_B)$.  On the other hand, $\cB_B=\cB_E=\cB_A$. Thus,
 as $f:\interior(\cB_E)\to \interior(\cD_A)$ is bianalytic, $\cD_B=\cD_A$. Since $A$ is minimal
defining for $\cD_A$ and $A$ and $B$ have the same size, $B$ is minimal for $\cD_A$.
Hence  $A$ and $B$ are unitarily equivalent by Lemma \ref{l:Z+}. From the form of 
$B$, it is evident that $r\ge \max\{d,e\}$. On the other hand, if $r>d+e$, then
$B$ must have $0$ as a direct summand and so is not minimal. Thus $r\le d+e.$ \hfill\qedsymbol

\subsection{Corollary~\ref{t:B2B}}
\label{subsec:TBD}
This subsection begins by illustrating Corollary \ref{t:B2B}
 in the case of free automorphism of free matrix balls and 
free polydiscs
before turning to the proof of the corollary.

\subsubsection{Automorphisms of free polydiscs}
\label{e:FPD}
 Let $\{e_1,\dots,e_{g}\}$ denote the usual orthonormal basis for $\C^{g}$
 and let $E_j = e_j e_j^*$. The spectraball $\cB_E$ is then the \df{free polydisc}
 with
\[
\interior(\cB_E(n)) = \{X\in M_n(\C)^g: \|X_j\|<1\}.
\]
 Let $b\in \interior(\cB_E(1)) = \mathbb D^g$ be given.  

In the setting of  Corollary \ref{t:B2B}, we choose $C=E$. 
If $\sV,\sW$ are
$g\times g$ unitary matrices such that equation Corollary \ref{t:B2B}\eqref{i:CME}
holds, then there exists a $g\times g$ permutation matrix $\Pi$ and unitary diagonal matrices
$\rho$ and $\mu$ such that $\sW=\Pi \rho$ and $\sV=\mu \Pi$. We can in fact assume
$\mu=I_g$.
It is now evident that item \eqref{i:ELCE} of Corollary \ref{t:B2B} holds and
determines $\Xi$.    Conversely, given a triple $(b,\Pi,\rho)$, 
where $b\in \mathbb D^g$, $\Pi$ is a $g\times g$
permutation matrix and $\rho$ is a diagonal unitary matrix, the equations
\eqref{i:CME} and \eqref{i:ELCE} of Corollary \ref{t:B2B} hold with
$\sW=\Pi \rho$ and $\sV=\Pi.$
  Hence the automorphisms of $\cB_E$ are determined by 
triples $(b,\Pi,\rho)$.

By pre (or post) composing with a permutation,  we may assume $\Pi=I_g$. In this case
$M$ is the $g\times g$  diagonal matrix  with diagonal entries  $M_{jj} = \rho_j (1-|b_j|^2)$ and  
 $\Xi_k = -\rho_j b_j^*  E_k$.  The corresponding 
convexotonic map $\psi(x) = x (I-\Lambda_\Xi(x))^{-1}$ has entries
\[
 \psi^j(x) =  x_j (1+c_j^* x_j)^{-1},
\]
where $c_j=\rho_j b_j^*$.
Thus the mapping $\varphi(x) = \psi(x)\cdotb M + b$ has entries,
\[
 \varphi^j(x) = \rho_j x_j (1+c_j^* x_j)^{-1}(1-|b_j|^2) + b_j 
  = \rho_j (x_j +c_j)(1+c_j^* x_j)^{-1},
\]
where $c_j= \rho_j b_j^*$.
Hence, the automorphisms
of the free polydisc are given by
\[
\varphi(x) = 
  \left (\rho_{\pi(1)} (x_{\pi(1)}+c_{\pi(1)})(1+c_{\pi(1)}^*x_{\pi(1)})^{-1},\ldots, \rho_{\pi(g)} (x_{\pi(g)}+c_{\pi(g)})(1+c_{\pi(g)}^* x_{\pi(g)})^{-1},
  \right )
\]
for $c=(c_1,\dots,c_g)\in \mathbb D^g,$  unimodular $\rho_j$ and a permutation
$\pi$ of $\{1,\ldots,g\}.$

\subsubsection{Automorphisms of free matrix balls}
\label{e:MT}
  Let $(E_{ij})_{i,j=1}^{d,e}$ denote the matrix units in $M_{d\times e}(\C)$
 and view $E\in M_{d\times e}(\C)^{de}$.  We consider automorphisms of
 $\cB_E$, the \df{free $d\times e$ matrix ball}.  \index{matrix ball}

  Before
 proceeding further, note, since $\{E_{ij}: 1\le i\le d,\, 1\le j\le e\}$
 spans all of $M_{d\times e}(\C)$, 
  by the reverse implication in Corollary \ref{t:B2B},
 any choice of $b$ in the unit ball of $M_{d\times e}(\C)$
 and $d\times d$ and $e\times e$ unitary matrices $\sW$ and $\sV$
 determines uniquely a $g\times g$ invertible matrix $M$
 satisfying the identity of item \eqref{i:CME}
 of Corollary \ref{t:B2B}. Likewise a 
 convexotonic tuple is uniquely determined by the identity of item \eqref{i:ELCE}.
 The resulting bianalytic automorphism $\varphi$ of $\cB_E$
  satisfying $\varphi(0)=b$ and $\varphi^\prime(0)=M$  is
 then given by the formula in Corollary \ref{t:B2B}.
 Our objective in the remainder of this  example is to show this formula for
 $\varphi$ agrees with
 that of \cite[Theorem 13]{MT16}.
 Doing so requires 
 passing back and forth
 between row vectors of length $de$ and  matrices of size $d\times e$.

First note that
\[
 \Lambda_E(b) = b.
\]
From item \eqref{i:CME} of Corollary \ref{t:B2B}
 (which defines $M$ in terms of $b$, $\sV$ and $\sW$),
\[
\begin{split}
   \sum_{u,v} M_{(i,j),(u,v)} E_{u,v} 
   & =     (M\cdotb E)_{i,j}\\
    & =   D_{\Lambda_E(b)^*} \sW E_{i,j} \sV^* D_{\Lambda_E(b)} \\
  & =  \sum_{u,v}  [e_u^* D_{\Lambda_E(b)^*} \sW e_i]\, 
      [e_j^*  \sV^* D_{\Lambda_E(b)} e_v]\, e_u e_v^*.
\end{split}
\]
Hence, 
\[
 M_{(i,j),(u,v)} = [e_u^* D_{\Lambda_E(b)^*} \sW e_i]\, 
      [e_j^*  \sV^* D_{\Lambda_E(b)} e_v].
\]
Next observe that, 
\[
- E_{ij} \sV^* \Lambda_E(b)^* \sW E_{st} = - e_ie_j^* \sV^* b^* \sW e_se_t^*
 = -(e_j^* \sV^* b^* \sW e_s) E_{it}.
\]
Hence, letting $\beta_{js} =-(e_j^* \sV^* b^* \sW e_s)$ for $1\le j\le e$
and $1\le s\le d$, the tuple $\Xi \in M_{de}(\C)^{de}$ defined by 
 (for $1\le i,u\le d$ and $1\le v\le e$)
\[
(\Xi_{st})_{(i,j),(u,v)}  = \begin{cases} \beta_{js} \, & v=t, \, u=i \\ 
    0 \, & \mbox{ otherwise,} \end{cases}
\]
satisfies the identity of equation item \eqref{i:ELCE} of 
Corollary \ref{t:B2B}. 
Hence the free bianalytic automorphism of $\cB_E$ determined by $b$, $\sW$ and $\sV$ is 
\begin{equation}
\label{e:varphide}
\varphi(x) =  \psi(x)\cdotb M + b
\end{equation}
where $\psi=x(I-\Lambda_\Xi(x))^{-1}$ is the convexotonic map determined by $\Xi$.

We next express  formula  for $\varphi$ in equation \eqref{e:varphide}
in terms of the canonical matrix structure on $\cB_E$. 
Given a matrix $\yy=(\yy_{ij})_{i,j=1}^{d,e}$, let
\[
\row(y) = \begin{pmatrix} y_{11} & y_{12} &\dots & y_{1e} & y_{21}&\dots &y_{de} \end{pmatrix}.
\]
Similarly, given $z=(z_j)_{j=1}^{de}$, let
\[
\matt_{d\times e} (z) = \begin{pmatrix} z_1 & z_2 & \dots & z_e \\ z_{e+1} & z_{e+2} & \dots & z_{2e}\\
  \\ \vdots & \vdots & \cdots & \vdots \\ z_{(d-1)e} & z_{(d-1)e+1} & \dots & z_{de} \end{pmatrix}.
\]
Since $d$ and $e$  are fixed in this example,
it is safe to abbreviate $\matt_{d\times e}$ to simply $\matt$.
For a tuple $\yy = (\yy_{s,t})_{s,t=1}^{d,e}$ of indeterminates,
\[
 \begin{split}
 (\yy\cdotb M)_{u,v} &=  \sum_{i,j}  M_{(i,j),(u,v)} \yy_{i,j} \\
  &= \sum_{i,j} [e_u^* D_{\Lambda_E(b)^*} \sW]\,   \yy_{i,j} e_i e_j^* \,[\sV^* D_{\Lambda_E(b)} e_v]\\
  &=  e_u^* \, [D_{\Lambda_E(b)^*} \sW]\,   \matt(\yy)  \,[\sV^* D_{\Lambda_E(b)}] \,  e_v.
\end{split}
\]
Thus,
\begin{equation}
\label{e:matyM}
 \matt(\yy \cdotb M)  = D_{\Lambda_E(b)^*} \sW\,  \matt(\yy)  \, \sV^* D_{\Lambda_E(b)}.
\end{equation}

Let 
\[
\Gamma_{(i,j),(u,v)}(x) := \Big(\sum_{s,t=1}^{d,e}  \Xi_{st} x_{st}\Big)_{(i,j),(u,v)} = 
 \begin{cases} \sum_{s=1}^d  \beta_{js} x_{sv}  &  u=i \\
    0 & \mbox{ otherwise.}
 \end{cases}
\]
Thus, $\Gamma$ is a $de \times de$ linear matrix polynomial of the form,
\[
 \Gamma = I_d \otimes \beta \matt(x)
\]
and  $(I-\Gamma)^{-1} = I_d\otimes (I-\beta \matt(x))^{-1}$.
In the formula for 
the convexotonic map  $\psi$ determined by $\Xi$, 
 the indeterminates $x=(x_{st})_{s,t}$ are arranged in a row 
and we find,
\[
\begin{split}
 \row(\psi(x)) = \row(x) (I-\Lambda_{\Xi}(x))^{-1} & = 
\begin{pmatrix} x_{11} & x_{12} & \dots x_{1e} &  x_{21} & \dots x_{de} \end{pmatrix}
 \, \Big(I\otimes (I-\beta \matt(x))^{-1}\Big) \\
& =  \begin{pmatrix} \hat{x}_1 (I-\beta \matt(x))^{-1} & \dots & \hat{x}_d (I-\beta \matt(x))^{-1} \end{pmatrix},
\end{split}
\]
where $\hat{x}_j = \begin{pmatrix} x_{j1}  &\dots x_{je}\end{pmatrix}$. Thus, 
\[
\begin{split}
 \row(x) & (I-\Lambda_{\Xi}(x))^{-1}  \\
  & =  {\small \begin{pmatrix}  (\matt(x)[I-\beta \matt(x)]^{-1})_{11} &  
        (\matt(x)[I-\beta \matt(x)]^{-1})_{12} & 
          \dots & (\matt(x)[I-\beta \matt(x)]^{-1})_{de} \end{pmatrix}.}
\end{split}
\]
Hence, in matrix form, 
\[
 \matt(\psi(x)) = \matt(x) (I-\beta \matt(x))^{-1} =\matt(x) (I+ (\sV^* b^* \sW)\matt(x))^{-1}.
\]
 Let $c=\sW^* b \sV$ and note 
\[
 I-\Lambda_E(b) \Lambda_E(b)^* = I-bb^* = I- \sW c c^* \sW^* = \sW(I-cc^*)\sW^*
  = \sW (I-\Lambda_E(c)\Lambda_E(c)^*)\sW^*.
\]
Thus,
\begin{equation}
\label{e:DsW=sWD}
 D_{\Lambda_E(b)^*} \sW = \sW D_{\Lambda_E(c)^*}
\end{equation}
and similarly $\sV^* D_{\Lambda_E(b)} = D_{\Lambda_E(c)} \sV^*$.
Consequently, using, in order,
equations \eqref{e:varphide}, \eqref{e:matyM},  and  \eqref{e:DsW=sWD}
together with the definition of $c$ in the first 
three equalities followed by some algebra,
\[
\begin{split}
 \matt(\varphi(x)) &=  \matt(\psi(x)\cdotb M) + b \\
 &=  D_{\Lambda_E(b)^*} \sW \matt(\psi) \sV^* D_{\Lambda_E(b)}   + b \\
 &=  \sW [D_{\Lambda_E(c)}^* \matt(\psi) D_{\Lambda_E(c)} + c] \sV^*\\
&=  \sW D_{\Lambda_E(c)^*} [\matt(\psi)  + D_{\Lambda_{E(c)^*}}^{-2} c ] D_{\Lambda_E(c)} \sV^* \\
&=  \sW D_{\Lambda_E(c)^*} [\matt(x)(I+c^*\matt(x))^{-1}   + D_{\Lambda_{E(c)^*}}^{-2} c ] D_{\Lambda_E(c)} \sV^* \\
&=  \sW D_{\Lambda_E(c)^*}^{-1}[D_{\Lambda_{E(c)^*}}^2 \matt(x) +  c (I+c^*\matt(x)) ]
   \,[I+c^*\matt(x)]^{-1} D_{\Lambda_E(c)} \sV^* \\
&= \sW (I-cc^*)^{-\frac12} [(1-cc^*)\matt(x) + c+ cc^*\matt(x)] \,[I+c^*\matt(x)]^{-1}
             D_{\Lambda_E(c)} \sV^* \\
 &= \sW (I-cc^*)^{-\frac12} [\matt(x)+c] \,[I+c^*\matt(x)]^{-1} (I-c^*c)^{\frac12} \sV^*,
\end{split}
\]
giving the  standard formula for the automorphisms of $\cB_E$ that
send $0$ to $b.$ (See, for example, \cite{MT16}.)

\subsubsection{Proof of Corollary \ref{t:B2B}}
\label{sssec:proof1point3}
Suppose $E=(E_1,\dots,E_g)\in M_{d\times e}(\C)^g$ and 
  $C=(C_1,\dots,C_g)\in M_{k\times \ell}(\C)^g$ are linearly independent  and
  \bmin and $\varphi:\interior(\cB_E)\to \interior(\cB_C)$ is bianalytic. 

Let 
$\hC$ denote the tuple
\[
 \hC_j = \begin{pmatrix} 0_{k,k} & C_j \\  0_{\ell,k} & 0_{\ell,\ell} \end{pmatrix}\in M_{r}(\C),
\]
where $r=k+\ell$.
 Thus $\cB_C=\cD_{\hC}$ and, since $C$ is \bmin, $\hC$ is minimal for $\cD_{\hC}$
 by Lemma \ref{l:trans}\eqref{it:lt3}.

  Let $b=\varphi(0)$ and for notational convenience, let 
 $\Lambda =\Lambda_C(b) \in M_{k\times \ell}(\C)$.  Set 
\begin{equation}
\label{d:sG}
 \sG = \begin{pmatrix} I_k & \Lambda \\ 0  & D_{\Lambda} \end{pmatrix}^{-1}
  =\begin{pmatrix} I_k & -\Lambda D_\Lambda^{-1} \\[.1cm] 0 & D_\Lambda^{-1} \end{pmatrix}
   \in M_r(\C),
\end{equation}
and observe that $\sG^* \M_C(b) \sG =I$ and therefore $\M_C(b)^{-1} = \sG \sG^*.$ Hence
 there is a unitary matrix $T$ such that $\sG=\M_E(b)^{-\frac12}T.$  It follows from 
 Proposition \ref{prop:aff},
letting   $\AF\in M_{r\times r}(\C)^g$ denote the $g$-tuple with entries
\begin{equation}
\label{d:Fj}
  \AF_j = \sG^* \, \begin{pmatrix} 0 & (M\cdotb C)_j \\ 0 & 0 \end{pmatrix} \, \sG
     \in M_r(\C)^g,
\end{equation}
 and  $M=\varphi^\prime(0),$ %
 that the inverse of the mapping $\lambda(x) = x\cdotb M +b$ 
is an affine linear bijection from  $\cB_C=\cD_{\hC}$ to $\cD_\AF$  and  $\AF$
 is minimal for $\cD_{\AF}.$

The mapping 
\[
\psif:= \lambda^{-1} \circ \varphi:\interior(\cB_E)\to \interior(\cD_\AF)
\]
is a free bianalytic mapping with $\psif(0)=0$
 and $\psif^\prime(0)=I$, where $E$ is \bmin and $\AF$ is minimal
for $\cD_{\AF}$.   An application of Theorem \ref{t:pencil-ball-alt} now implies that
there is a convexotonic tuple $\Xi$ such that equation \eqref{e:JgetsXi} holds,
$f$ is the corresponding convexotonic map and there are unitaries
$V$ and $W$ of size $r$ such that 
\begin{equation}
\label{e:F=WhatEV}
 \AF = W \begin{pmatrix} 0_{d,r-e} & E  \\ 0_{r-d,r-e} & 0_{r-d,e} \end{pmatrix} V^*.
\end{equation}
In particular, $\varphi(x)=f(x)\cdotb M + b.$

From equation \eqref{e:F=WhatEV},
\[
 \sum \AF_j^* \AF_j = V \begin{pmatrix}  0 & 0 \\ 0 & \sum_j E_j^* E_j \end{pmatrix} V^*
\]
and consequently $\rk \sum \AF_j^* \AF_j = \rk \sum E_j^* E_j$. 
Since $E$ is \bmin,  $\ker(E)=\{0\}$. Equivalently, 
 $\rk \sum E_j^* E_j = e$.
On the other hand, from equation \eqref{d:Fj}, 
\[
 \sum \AF_j^* \AF_j  = \sG^* \begin{pmatrix} 0 & 0 \\ 0 & (M\cdot C)^*_j \Gamma  (M\cdotb C)_j
 \end{pmatrix} \, \sG,
\]
 where $\Gamma$ is the $(1,1)$ block entry of $\sG \sG^*$.  Observe that $\Gamma$ is positive
 definite and,  since $C$ is \bmin, $\ker(M\cdot C)=\{0\}.$ Hence
 $\rk \sum \AF_j^* \AF_j=\ell.$  Thus $e= \ell$.  Computing
$\sum \AF_j \AF_j^*$ using equation \eqref{e:F=WhatEV} shows $\rk \sum A_jA_j^*=d$.
 On the other hand, using equation \eqref{d:Fj},
\[
 \sum_{j=1}^g  \AF_j \AF_j^* = \sG 
\begin{pmatrix} \sum_{j=1}^g  (M\cdotb C)_j  D_{\Lambda}^{-2} (M\cdotb C)_j^*     & 0\\
  0 & 0 \end{pmatrix} \sG^*.
\]
Since $C$ is $k\times \ell$ and  \bmin, $\ker((M\cdot C)^*) =\{0\}$ and $D_{\Lambda}^{-2}$
 is positive definite, $\rk \sum_{j=1}^g  (M\cdotb C)_j  D_{\Lambda}^{-2} (M\cdotb C)_j^*=k.$
 Hence $d=\rk \sum_{j=1}^g A_j A_j^* =k.$ Thus $E$ and $C$ have the same size $d\times e$.

Since $E$ and $C$ are both $d\times e$ and $r=d+e,$ the matrices $V$ and $W$ decompose as
\[
 V=  
 \begin{pmatrix} V_{11} & V_{12}\\V_{21} & V_{22} \end{pmatrix},
 \quad \quad W=  
 \begin{pmatrix} W_{11} & W_{12}\\W_{21} & W_{22} \end{pmatrix}
\]
with respect to the decomposition $\C^r = \C^d\oplus \C^e$. In particular,
$V_{jj}$ and $W_{jj}$ are all square.
Comparing equation \eqref{e:F=WhatEV} and equation \eqref{d:Fj} gives
\begin{equation}
\label{e:compareFs}
 \begin{pmatrix} W_{11} E_j V_{12}^* & W_{11} E_j V_{22}^* \\
           W_{21}E_j V_{12}^* & W_{21} E_j V_{22}^* \end{pmatrix}
   = \begin{pmatrix} 0 & (M\cdotb C)_j D_\Lambda^{-1} \\ 
          0& - D_{\Lambda}^{-1} \Lambda^* (M\cdotb C)_j D_\Lambda^{-1} \end{pmatrix}.
\end{equation}
Multiplying both sides of equation \eqref{e:compareFs} by 
$\begin{pmatrix} W_{11}^* & W_{21}^*\end{pmatrix}$ and using the fact that $W$ 
 is unitary shows,
\[
 E_j V_{12}^* =0.
\]
Since $E$ is \bmin and $\sum E_j^* E_j V_{12}^* =0$ we conclude that $V_{12}=0$.
Since $V$ is unitary, $V_{22}$ is isometric and since $V_{22}$ is square
($e\times e$) it is unitary  (and thus $V_{21}=0$).  Further,
\begin{equation}
\label{e:WEV}
\begin{split}
   W_{11} E_j V_{22}^* & =  (M\cdotb C)_j D_\Lambda^{-1} \\
   W_{21} E_j V_{22}^* & =  - D_\Lambda^{-1} \Lambda^* (M\cdotb C)_j D_\Lambda^{-1}.
\end{split}
\end{equation}
Thus, $W_{21} E_j V_{22}^* =  - D_\Lambda^{-1} \Lambda^* W_{11} E_j V_{22}^*$
and hence $W_{21} E_j =  - D_\Lambda^{-1} \Lambda^* W_{11} E_j.$
It follows that
\[
 W_{21} \sum E_j E_j^* =  -D_\Lambda^{-1} \Lambda^* W_{11} \sum E_j E_j^*.
\]
Thus, again using that $E$ is \bmin (so that $\ker(E^*)=\{0\}$), 
\[
 W_{21} = -D_\Lambda^{-1}\Lambda^* W_{11}.
\]
Hence,
\[
 I = W_{11}^* W_{11} + W_{21}^* W_{21} = W_{11}^*[I+\Lambda D_{\Lambda}^{-2} \Lambda^*] W_{11}
  = W_{11}^* D_{\Lambda^*}^{-2} W_{11}
\]
and, since $W_{11}$ is $d\times d$, we conclude that it  is invertible and 
\[
 W_{11} W_{11}^* = D_{\Lambda^*}^2.
\]
Consequently there is a $d\times d$ unitary $\sW$ such that
\begin{equation}
\label{e:W11W21}
\begin{split}
W_{11} & =   D_{\Lambda^*} \sW \\
 W_{21} & =  -D_\Lambda^{-1} \Lambda^* D_{\Lambda^*} \sW = - \Lambda^* \sW.
\end{split}
\end{equation}
Combining the first bits of each of equations \eqref{e:WEV} and  \eqref{e:W11W21}
 and setting $\sV=V_{22}$
gives  Corollary \ref{t:B2B}\eqref{i:CME}. Namely,
\begin{equation*}
(M\cdotb C)_j =  D_{\Lambda^*} \sW E_j \sV^* D_\Lambda.
\end{equation*}

Observe (using $E$ and $C$ have the same size) that,
\begin{equation*}
 \AF = W\begin{pmatrix} E & 0\\ 0 & 0 \end{pmatrix} \,
  \begin{pmatrix} 0_{e\times d} &I_e \\ I_d & 0_{d\times e} \end{pmatrix} V^*.
\end{equation*}
The tuple  $\AF$ is, up to unitary equivalence, of the form of equation \eqref{e:UEE}
where
\[
 U=  \begin{pmatrix} 0 & V_{22}^* \\ V_{11}^* & 0\end{pmatrix}\, 
     \begin{pmatrix} W_{11} & W_{12}\\ W_{21} & W_{22} \end{pmatrix}
   = \begin{pmatrix} \sV^* W_{21} & * \\ * & * \end{pmatrix}.
\] 
Thus, $U_{11} = \sV^* W_{21}=-\sV^* \Lambda^* \sW$.  Since
the pair $(A,\Xi)$ satisfies equation \eqref{e:JgetsXi}, 
\[
 \begin{pmatrix} E_k & 0 \\ 0 & 0 \end{pmatrix} \, U \,
 \begin{pmatrix} E_j & 0 \\ 0 & 0 \end{pmatrix} =
 \sum_s (\Xi_j)_{k,s} \begin{pmatrix} E_s & 0 \\ 0 & 0 \end{pmatrix},
\]
 item \eqref{i:ELCE} holds.

To prove the converse, 
suppose $E,C\in M_{d\times e}(\C)^g$ and $b\in \cB_C(1)$ are given
and  there exists an invertible $M\in M_g(\C)$, a convexotonic tuple
$\Xi \in M_g(\C)^g$ and  unitaries $\sW$ and $\sV$ 
such that items \eqref{i:ELCE} and \eqref{i:CME} of Corollary \ref{t:B2B} 
hold. Let $\Lambda =\Lambda_C(b)$ and define $\sG$ and  $\AF$ as in equations \eqref{d:sG} and \eqref{d:Fj} respectively.
The map $\lambda(x)=x\cdotb M + b$ is again an affine linear bijection from $\cD_\AF$ to
 $\cB_C.$

Define $W_{11}$ and $W_{21}$ by equation \eqref{e:W11W21}.  It follows that
 $W_{11} W_{11}^* + W_{21} W_{21}^*=I$. Choose $W_{12}$ and $W_{22}$ such that
 $W=(W_{ij})_{i,j=1}^2$ is a (block) unitary matrix. Let $V_{22}=\sV$ and take
any unitary $V_{11}$ (of the appropriate size) and set
\[
 V = \begin{pmatrix} V_{11} & 0 \\0 & V_{22}\end{pmatrix}.
\]

Next,  using item \eqref{i:CME}, the definitions of $W_{11}$ and $W_{12}$
and $D_\Lambda^{-1}\Lambda^* D_{\Lambda^*} = \Lambda^*,$ 
\begin{equation*}
\begin{split}
\AF_k & = \sG^* \begin{pmatrix} 0 & (M\cdotb C)_k \\ 0 & 0 \end{pmatrix} \sG =  \begin{pmatrix} 0 & (M\cdotb C)_k D_\Lambda^{-1} \\ 
       0 & - D_\Lambda^{-1} \Lambda^* (M\cdot C)_k D_\Lambda^{-1} \end{pmatrix} \\
 &=  \begin{pmatrix} 0 & D_{\Lambda^*} \sW E_k \sV^*  \\
     0 & -\Lambda^* \sW E_k \sV^* \end{pmatrix} 
=  \begin{pmatrix} 0 & W_{11} E_k \sV^* \\ 0 & W_{21}E_k \sV^*\end{pmatrix}. 
\end{split}
\end{equation*}
Thus, using item \eqref{i:ELCE}, 
\[
\begin{split}
 \AF_j \AF_k  &=  \begin{pmatrix} 0 & W_{11} E_j \sV^* W_{21} E_k \sV^* \\
                  0 & W_{21} E_j \sV^* W_{21} E_k \sV^* \end{pmatrix} 
  =  \sum_s (\Xi_k)_{j,s} \begin{pmatrix} 0 & W_{11} E_s \sV^* \\ 
                      0 & W_{21} E_s \sV^* \end{pmatrix}
  =  \sum_s (\Xi_k)_{j,s} \AF_s.
\end{split}
\]
Thus $\AF$ spans an algebra with multiplication table given by $\Xi$. 
Consequently $\psif(x) = x(I-\Lambda_\Xi(x))^{-1}$   is convexotonic
from $\interior(\cB_\AF)$ to $\interior(\cD_\AF)$ by 
 Proposition \ref{prop:properobvious} .  On the other hand,
$\cB_\AF=\cB_E$, since 
\[
 \AF_j^* \AF_k = \begin{pmatrix} 0 & 0 \\ 0 & \sV E_j^* E_k \sV^* \end{pmatrix}
\]
(because $W_{11}^* W_{11}+W_{21}^* W_{21}=I$).
 Thus $\psif$ is convexotonic from $\interior(\cB_E)$ to $\interior(\cD_\AF).$ 
Finally, $\varphi =\lambda \circ \psif$ is convexotonic
from $\interior(\cB_E)$ to $\interior(\cB_C)$ with $\varphi(0)=b$ and
$\varphi^\prime(0)= M$.

The uniqueness is well known. Indeed, if $\varphi$ and $\zeta$ are both
 bianalytic from $\cB_E\to \cB_C$, send 
 $0$ to $b$ and have the same derivative at $0$, then $\psif=\varphi\circ \zeta^{-1}$ 
 is an analytic automorphism of $\cB_C$ sending $0$ to $0$ and having derivative
 the identity at $0$. Since $\cB_C$ is {\it circular}, the free version of Cartan's
 Theorem \cite{HKM11b} says $\psif(x)=x$ and hence $\zeta=\varphi$.\hfill\qedsymbol

\section{Convex sets defined by rational functions}
In this section we employ a variant of the main result of \cite{HM14} 
to extend  Theorem \ref{t:pencil-ball-alt}  to cover
birational maps from a matrix convex set to a spectraball.
A free set is \df{matrix convex}
if it is closed with respect to isometric conjugation.
We refer the reader to \cite{EW,HKMjems,Kr,FHL18,PSS} 
for the theory of matrix convex sets.
For expository convenience, by free rational mapping $p:M(\C)^g\to M(\C)^g$ 
we mean $p=\begin{pmatrix} p^1 & p^2 &\dots p^g\end{pmatrix}$ where
each $p^j=p^j(x)$ is a free rational function (in the $g$-variables 
$x=(x_1,\dots,x_g)$) regular at $0.$    Theorem \ref{t:rat2} immediately below is the main
result of this section. It is followed up by two corollaries.

\begin{thm}
\label{t:rat2}
 Suppose $\mfq:M(\C)^g\to M(\C)^g$ is a free rational mapping, $\mathscr{C}\subset M(\C)^g$
 is a  bounded  open matrix convex set containing the origin and $E\in M_{d\times e}(\C)^g$.
 If $E$ is linearly independent, 
 $\mathscr{C}\subset \dom(\mfq)$ and $\mfq:\mathscr{C}\to \interior(\cB_E)$ is bianalytic, then
 there exists an $r\le d+e$ and a tuple 
$A\in M_r(\C)^g$ such that $\mathscr{C}=\interior(\cD_A)$ and
$\mfq$ is, up to affine linear equivalence,  convexotonic.
\end{thm}

\begin{cor}
 \label{c:rat1}
   Suppose $\mfp:M(\C)^g\to M(\C)^g$ is a free rational mapping,  $E\in M_{d\times e}(\C)^g$
 is linearly independent and let
\[
  \mathscr{C} := \{X: X\in \dom(\mfp), \, \|\Lambda_E(\mfp(X))\|<1\}.
\]
Assume $\mathscr{C}$ is bounded,  convex and contains $0$. If $X_k\in \mathscr{C}(n)$ and the sequence $(X_k)_k$ converges to $X\in\partial \mathscr{C}$
 implies $\lim_{k\to\infty}\|\Lambda_E(\mfp(X_k))\|=1$,  
then there exists an $r\le d+e$ and a tuple 
$A\in M_r(\C)^g$ such that $\mathscr{C}=\interior(\cD_A)$ and
$\mfp:\interior(\cD_A)\to \interior(\cB_E)$ is bianalytic and,
up to affine linear equivalence,  convexotonic. 
\end{cor}

\begin{proof}
 By assumption $\mfp:\mathscr{C}\to \interior(\cB_E)$ is a proper map.
 By \cite[Theorem 3.1]{HKM11b}, $\mfp$ is bianalytic.  Hence
 Corollary \ref{c:rat1} follows from Theorem \ref{t:rat2}.
\end{proof}

\begin{cor}
\label{c:rat3}
  Suppose $p:M(\C)^g\to M(\C)^g$ is a free polynomial mapping,  $E\in M_{d\times e}(\C)^g$
 is linearly independent and let
\[
  \mathscr{C} := \{X: \|\Lambda_E(p(X))\|<1\}.
\]
If $\mathscr{C}$ is bounded,  convex and contains $0,$
then there exists an $r\le d+e$ and a tuple 
$A\in M_r(\C)^g$ such that $\mathscr{C}=\interior(\cD_A)$ and
$p:\interior(\cD_A)\to \interior(\cB_E)$ is bianalytic and,
 up to affine linear equivalence,  convexotonic. 
\end{cor}

\begin{proof}
By hypothesis $p:\mathscr{C}\to \interior(\cB_E).$ 
Let $X\in \partial \mathscr{C}$ be given. By convexity and continuity 
 $p(tX)\in \interior(\cB_E)$ for $0\le t < 1$ 
and $p(X)\in\cB_E$. If $p(X)\in \interior(\cB_E)$, then there exists $t_*>1$ such
$p(t_*X)\in\interior(\cB_E).$ But then $0,\, t_*X\in\mathscr{C}$ and 
$X\notin\mathscr{C}$, violating convexity of $\mathscr{C}$. Hence
 $p(X)\in\partial \cB_E$ and consequently $p$ is a proper map. 
Thus Corollary \ref{c:rat3} follows from Corollary \ref{c:rat1}.
\end{proof}

The proof of Theorem \ref{t:rat2} given here depends on two preliminary results.
Let $\rxy$ denote the  \index{$\rxy$}
\df{skew field of free rational functions}
in the
freely noncommuting variables  
$(x,y)=(x_1,\dots,x_g,y_1,\dots,y_g)$.
There is an involution $\widecheck{}$ on $\rxy$ determined by
$\wc{x_j}=y_j$.
A $p\in \rxy$ is \df{symmetric} if $\wc{p}=p$.
An important feature of the involution is the fact that, if
$p\in \rxy$ and $(X,X^*)\in \dom(p)$, then $\wc{p}(X,X^*)=p(X,X^*)^*$
and $p$ is symmetric if and only if $\wc{p}(X,X^*)=p(X,X^*)$ for all
$(X,X^*)\in \dom(p)\cap \dom(\wc{p})$.
These notions naturally extend to matrices over $\rxy$.

Proposition \ref{p:convexrat} below is  
a variant of the main result of \cite{HM14}. Taking advantage of recent advances
in our understanding of the singularities of free rational functions (e.g., \cite{Vol17}), the proof given
here is rather short, compared to that of the similar result in \cite{HM14}.

\begin{prop}
\label{p:convexrat}
 Suppose $\mfs(x,y)$ is a  $\mu\times \mu$ symmetric matrix-valued free rational function in 
the $2g$-variables $(x_1,\dots,x_g,y_1,\dots,y_g)$ that is regular at $0.$  Let
\[
 S= \{X\in M(\C)^g: (X,X^*) \in \dom(\mfs), \, \mfs(X,X^*) \succ 0\},
\]
let $S^0$ denote the (level-wise) connected component of $0$ of $S,$
and assume $S^0(1)\neq \emptyset.$
If  each $S^0(n)$ is
convex, then there is a positive integer $\NN$ and a tuple $A\in M_\NN(\C)^g$ such that
$S^0=\interior(\cD_A)$.
\end{prop}

\begin{proof}
 From \cite{KVV09,Vol17} the free rational function $\mfs$ has an observable and 
 controllable realization. By \cite{HMV}, since $\mfs$ is symmetric, this
realization can be symmetrized. Hence, 
 there exists a positive integer $t$, a tuple $T\in M_t(\C)^g,$ 
 a signature matrix $J\in M_t(\C)$ (thus $J=J^*$ and $J^2=I$) and matrices
 $D$ and $C$ of sizes $\mu\times \mu$ and $t\times \mu$ respectively such that
\[
 \mfs(x,y) = D + C^* L_{J,T}(x,y)^{-1}C
\]
and $\dom(\mfs) = \{(X,Y): \det(L_{J,T}(X,Y))\ne 0\},$ where
\[
 L_{J,T}(x,y) = J-\Lambda_T(x)-\Lambda_{T^*}(y) = J-\sum T_j x_j - \sum T_j^* y_j.
\]
 Let
$\tmfs(x,y)=\mfs(x,y)^{-1}$.  Thus $\tmfs(x,y)$ is also a
$\mu\times \mu$ symmetric matrix-valued free rational function.  It has a representation,
\[
 \tmfs(x,y) = \tilde{D} + \tilde{C}^* L_{\tilde{J},\tilde{T}}(x,y)^{-1} \tilde{C},
\]
with $\dom(\tmfs) =\{(X,Y): \det(L_{\tJ,\tT}(X,Y))\ne 0\}.$  Let
\[
 Q(x) = \Big(\frac{J}{2} -\Lambda_T(x)\Big) \oplus \Big(\frac{\tJ}{2}-\Lambda_{\tT}(x)\Big),
\]
 let   $P(x,x^*)=Q(x)+Q(x)^*$, 
let $\sI =\{X:\det(P(X))\ne 0\}$ and let $\sI^0$ denote its connected
component of $0$. Observe that 
$\{(X,X^*): X\in \sI\} = \{X:(X,X^*)\in \dom(\mfs) \cap \dom(\tmfs)\}.$ In particular,
if $X\in \sI^0$, then $(X,X^*)\in\dom(\mfs)\cap\dom(\tmfs)$.
On the other hand, if  $(X,X^*)\in \dom(\mfs)$ and $\mfs(X,X^*)\succ 0$,
then $\mfs(X,X^*)$ is invertible and hence $(X,X^*) \in \dom(\tmfs)$. 
Hence, if $X\in S^0$, then $(X,X^*)\in \dom(\mfs)\cap\dom(\tmfs)$ too.

 Suppose $X\in S^0$. Thus $tX\in S^0$ for $0\le t\le 1$ by
convexity. It follows that $t(X,X^*)\in \dom(\mfs)\cap \dom(\tmfs)$. Hence
$tX\in \sI$ for $0\le t\le 1$. Thus $X\in \sI^0$ and $S^0\subset \sI^0$.

Arguing by contradiction, suppose there exists $X\in \sI^0\setminus S^0$.  It follows that
there is a (continuous) path $F$ in $\sI^0$ such that $F(0)=0$ and $F(1)=X$.
There is a smallest $0<\alpha \le 1$ with the property $Y=F(\alpha)$ is in the boundary of $S^0$.
Since $Y\in \sI^0$, $(Y,Y^*)\in \dom(\mfs).$ Since  $Y\notin S^0$, 
 $\mfs(Y,Y^*)\succeq 0$ is not invertible.
It follows that $Y\in \sI^0$, but $(Y,Y^*)\notin  \dom(\tmfs)$,  a contradiction.
Hence $\sI^0=S^0$ is the component of the origin of the set of $X\in M(\C)^g$
such that $P(X)$ is invertible. 
  By a variant of the main result in \cite{HM12}, $S^0$ is the interior of a free spectrahedron.
\end{proof}

\begin{lemma}
\label{lem:Emfq}
 If  $\mfq:M(\C)^g\to M(\C)^g$ 
is a free rational mapping
 and $E\in M_{d\times e}(\C)^g$ is linearly independent,
 then 
\begin{enumerate}[\rm(1)]
 \item \label{i:Emfq1}
    the domains of $\mfq$ and $\mfQ(x) := \Lambda_E(\mfq(x))$ coincide; 
 \item \label{i:Emfq2}
    $\dom(\wc{\mfq}) =\dom(\mfq)^*:=\{X: X^* \in \dom(\mfq)\}$;   and 
 \item \label{i:Emfq3}
  the domain of
\begin{equation}
\label{e:Emfq}
\mfr(x,y):= \begin{pmatrix} I_{d\times d}  & Q(x)\\ 
      \wc{Q}(y) & I_{e\times e}\end{pmatrix}
\end{equation}
is $\dom(\mfq)\times \dom(\mfq)^*= \{(X,Y): X,Y^* \in \dom(\mfq)\}$.
\end{enumerate}
\end{lemma}

\begin{proof}
 The inclusion $\dom(\mfq)\subset \dom(\mfQ)$ is evident. To prove the converse,
 let $1\le k\le g$ be given.  Using the linear independence of $\{E_1,\dots,E_g\},$
 choose a linear functional $\lambda_k$ on the span of $\{E_1,\dots,E_g\}$
 such that $\lambda_k(E_j)=1$ if $j=k$ and $0$ otherwise. It follows that
 the domain of  $\lambda_k(\mfQ(x))= q^k(x)$ contains  $\dom(\mfQ)$. 
 Hence $\dom(\mfQ) \subset \dom(\mfq)$, proving item \eqref{i:Emfq1}.

 Item \eqref{i:Emfq2} is evident as is the inclusion 
 $\dom(\mfr) \supset \dom(\mfq)\times \dom(\mfq)^*$ 
of \eqref{i:Emfq3}. 
 For $1\le j\le g$, let
\[
 F_j = \begin{pmatrix} 0 & E_j \\ 0 & 0\end{pmatrix}
\]
 and let $F_j= F_{j-g}^*$ for $g<j\le 2g$. Observe that 
 $r(x,y)=\Lambda_F(\mfq(x),\wc{\mfq}(y))$.  It follows
 from
 item \eqref{i:Emfq1} applied to $(\mfq(x),\wc{\mfq}(y))$
 and $F$ that
\[
\dom(r)=[\dom(\mfq)\times M(\C)^g] \, \cap \, [M(\C)^g\times \dom(\wc{\mfq})]
   = \dom(\mfq)\times \dom(\mfq)^*,
\]
proving item \eqref{i:Emfq3} and the lemma.
\end{proof}

\begin{proof}[Proof of Theorem \ref{t:rat2}]
It is immediate that
\[
 \mathscr{C}\subset S:=\{X:X\in \dom(\mfq), \ \ \|\Lambda_E(\mfq(X))\| <1\}.
\]
Let $S^0$ denote the connected component of $S$ containing $0$.
Since $\mathscr{C}$ is open, connected and contains
 the origin, $\mathscr{C}\subset S^0.$

 Let $Q = \Lambda_E\circ\mfp$ and 
 let $\mfr$ denote the ($(d+e)\times (d+e)$ symmetric matrix-valued) free rational function
 defined in equation \eqref{e:Emfq}. By Lemma \ref{lem:Emfq}, 
 $\{X: (X,X^*)\in \dom(\mfr)\} = \dom(\mfq)$ and moreover, for 
 $X\in \dom(\mfq),$ we have $\mfq(X)\in \interior(\cB_E)$ if and only  if
 $\mfr(X,X^*)\succ 0$. Thus,
\[
 S= \{X: (X,X^*) \in \dom(\mfr), \, \mfr(X)\succ 0\}.
\]

Arguing by contradiction, suppose  $Y\in S^0$,  but $Y\notin \mathscr{C}$. 
By connectedness, there is a continuous path $F$ in $S^0$ such that
$F(0)=0$ and $F(1)=Y$. Let $0<\alpha\le 1$ be the smallest number such
that $X=F(\alpha)\in  \partial\mathscr{C}$.  
Since $\mfq:\mathscr{C}\to \interior(\cB_E)$ is bianalytic, it is proper.
Hence, if $X\in \dom(\mfq)$,
then $\mfq(X)\in \partial \cB_E$ and consequently $X\notin S$. On the other hand,
if $X\notin \dom(\mfq)$, then $X\notin S$. In either case we obtain a contradiction.
Hence $S^0\subset \mathscr{C}$.

Since $\mathscr{C}=S^0$ is convex (and so connected), Proposition \ref{p:convexrat} implies 
there is a positive integer $\NN$ and tuple $A\in M_\NN(\C)^g$ such that
$\mathscr{C}=\interior(\cD_A)$.   Since $\interior(\cD_A)$ is bounded,
the tuple $A$ is linearly independent.  Without loss of generality,
we may assume that $A$ is minimal for $\cD_A.$
Since $p^{-1}:\interior(\cD_A)\to \interior(\cB_E)$ is bianalytic and
$A$ and $E$ are linearly independent,  Theorem \ref{t:pencil-ball-alt}
and Remark \ref{r:aff}\eqref{i:affa} together imply  $p^{-1}$, and hence $p$, is,
 up to affine linear equivalence, convexotonic and $r\le d+e$ 
 by Theorem \ref{t:pencil-ball-alt}. 
\end{proof}

\appendix

\section{Context and motivation}
\label{r:motivate+}

The main development over the past two decades in convex programming
has been the advent of 
linear matrix inequalities (LMIs); with
the subject generally going under the heading
of semidefinite programming (SDP).
SDP is  a generalization of linear programming
and many branches of science have a collection
of paradigm problems that reduce to SDPs, but not to linear programs.
There is highly developed  software for solving
optimization problems presented as LMIs.
 In $\R^g$ sets defined by LMIs
  are very special cases of  convex sets known as spectrahedra.
 However, as to be discussed, in the noncommutative case
convexity is closely tied to free spectrahedra.

The study of free spectrahedra and their bianalytic equivalence
derives motivation from
systems engineering and connections to other areas of mathematics.
Indeed the paradigm problems in linear systems engineering textbooks
 are {\it dimension free}  in that what is given is a signal flow diagram
and the algorithms and resulting software toolboxes handle {\it any}  system
having this signal flow diagram.
Such a problem leads  to a matrix inequality
whose  solution (feasible) sets $D$ is 
{\it free semialgebraic} \cite{dOHMP09}. Hence $D$ is closed under direct sums and
simultaneous unitary conjugation, i.e., it is a free sets.
In this dimension free setting,  if  $D$ is convex, then it is a free spectrahedron \cite{HM12,Kr}.
For optimization and design purposes,
it is hoped that $D$ is convex (and hence a spectrahedron), and algorithm designers
put great effort into converting (say by change of variables)
the problem they face to one that is convex.

If the domain $D$  is not
convex one might attempt to  map it bianalytically to a free
spectrahedron.
The classical problems of linear control that reduce to convex problems
all require a change of variables, see \cite{Skelton}.
One  bianalytic map composed with the inverse of another
leads to a bianalytic map between free spectrahedra;
thus maps between free spectrahedra characterize the non-uniqueness
of bianalytic mappings from the solution set $D$ of a
system of matrix inequalities  to a free spectrahedron.

Studying bianalytic
maps between free spectrahedra is a free analog of rigidity problems in several
complex variables
\cite{DAn93, For89, For93, HJ01, HJY14, Kra92}.
Indeed, there is a large literature
on bianalytic maps on convex sets. For example,
Forstneri\v c \cite{For93} showed that any proper
map between balls with sufficient regularity at the boundary must be rational.
The conclusions we see here in  Theorems \ref{t:pencil-ball-alt},  \ref{t:B2B}
and \ref{t:1point1} 
are vastly more rigid than mere birationality.

\end{document}